\documentclass[a4paper,11pt]{amsart}
\usepackage[utf8x]{inputenc}
\usepackage{color}
\usepackage{float}
\usepackage{amsthm}
\usepackage{amssymb}
\usepackage{amsmath}
\usepackage{verbatim}
\usepackage{graphicx}
\usepackage{tikz}
\usepackage{tikz-cd}
\usepackage{stmaryrd} 
\usetikzlibrary{chains,arrows,positioning}
\usetikzlibrary{matrix,arrows}
\usetikzlibrary{shapes.geometric}

\usepackage[bbgreekl]{mathbbol}

\DeclareSymbolFontAlphabet{\mathbbm}{bbold}
\DeclareSymbolFontAlphabet{\mathbb}{AMSb}%

\numberwithin{equation}{section}

\theoremstyle{plain}
\newtheorem{theorem}{Theorem}[section]
\newtheorem{recursionstep}[theorem]{Recursion Step}
\newtheorem{recursiveconstruction}[theorem]{Recursive Construction}
\newtheorem{prop}[theorem]{Proposition}
\newtheorem{cor}[theorem]{Corollary}
\newtheorem{lemma}[theorem]{Lemma}

\theoremstyle{definition}
\newtheorem{defn}[theorem]{Definition}
\newtheorem{example}[theorem]{Example}

\theoremstyle{remark}
\newtheorem{remark}[theorem]{Remark}

\newcommand{\To}{\longrightarrow}
\newcommand{\otau}{\overline\tau}

\newcommand{\trans}{\operatorname{\mathbb T}}
\newcommand{\Hom}{\operatorname{Hom}}

\newcommand{\x}{\times}
\newcommand{\Q}{Q}

\newcommand{\inv}{^{-1}}

\newcommand{\pr}{\mathrm{pr}}

\newcommand{\R}{\mathcal{R}}

\newcommand{\C}{\mathbb{C}}
\renewcommand{\Q}{\mathbb{Q}}
\newcommand{\Z}{\mathbb{Z}}

\newcommand{\Newton}{\operatorname{Newton}}
\newcommand{\Coeff}{\operatorname{Coeff}}
\newcommand{\Log}{\operatorname{LogCoeff_T}}
\newcommand{\AugNewton}{\operatorname{AugNewton}}
\newcommand{\ConvexHull}{\operatorname{ConvexHull}}
\newcommand{\Span}{\operatorname{Span}}
\newcommand{\Image}{\operatorname{im}}
\newcommand{\crit}{d_{\operatorname{crit}}}
\newcommand{\pcrit}{p_{\operatorname{crit}}}
\newcommand{\wcrit}{w_{\operatorname{crit}}}

\newcommand{\mon}{\operatorname{mon}}
\newcommand{\lowestface}{\operatorname{LowestFace}}

\newcommand{\dcoeff}{d_{\operatorname{coeff}}}
\newcommand{\AugDelta}{\mathbbm{\Delta}}
\newcommand{\NDelta}{\Delta}
\newcommand{\m}{\mathfrak m}

                                                                                                                                                                                                                                                                                                                                                                                                                                                                                                                                                                                                                                                                                                                                                                                                                                                                                                                                                                                                                                                                                                                                                                                                                                                                                                                                                                                                                                                                                                                                                                                                                                                                                                                                                                                                                                                                                                                                                                                                                                                                                                                                                                                                                                                                                                                                                                                                                                                                                                                                                                                                                                                                                                                                                                                                                                                                                                                                                                                                                                                                                                                                                                                                                                                                                                                                                                                                                                                                                                                                                                                                                                                                                                                                                                                                                                                                                                                                                                                                                                                                                                                                                                                                                                                                                                                                                                                                                                                                                                                                                                                                                                                                                                                                                                                                                                                                                                                                                                                                                                                                                                                                                                                                                                                                                                                                                                                                                                                                                                                                                                                                                                                                                                                                                                                                                                                                                                                                                                                                                                                                                                                                                                                                                                                                                                                                                                                                                                                                                                                                                                                                                                                                                                                                                                                                                                                                                                                                                                                                                                                                                                                                                                                                                                                                                                                                                                                                                                                                                                                                                                                                                                                                                                                                                                                                                                                                                                                                                                                                                                                                                                                                                                                                                                                                                                                                                                                                                                                                                                                                                                                                                                                                                                                                                                                                                                                                                                                                                                                                                                                                                                                                                                                                                                                                                                                                                                                                                                                                                                                                                                                                                                                                                                                                                                                                                                                                                                                                                                                                                                                                                                                                                                                                                                                                                                                                                                                                                                                                                                                                                                                                                                                                                                                                                                                                                                                                                                                                                                                                                                                                                                                                                                                                                                                                                                                                                                                                                                                                                                                                                                                                                                                                                                                                                                                                                                                                                                                                                                                                                                                                                                                                                                                                                                                                                                                                                                                                                                                                                                                                                                                                                                                                                                                                                                                                                                                                                                                                                                                                                                                                                                                                                                                                                                                                                                                                                                                                                                                                                                                                                                                                                                                                                                                                                                                                                                                                                                                                                                                                                                                                                                                                                                                                                                                                                                                                                                                                                                                                                                                                                                                                                                                                                                                                                                                                                                                                                                                                                                                                                                                                                                                                                                                                                                                                                                                                                                                                                                                                                                                                                                                                                                                                                                                                                                                                                                                                                                                                                                                                                                                                                                                                                                                                                                                                                                                                                                                                                                                                                                                                                                                                                                                                                                                                                                                                                                                                                                                                                                                                                                                                                                                                                                                                                                                                                                                                                                                                                                                                                                                                                                                                                                                                                                                                                                                                                                                                                                                                                                                                                                                                                                                                                                                                                                                                                                                                                                                                                                                                                                                                                                                                                                                                                                                                                                                                                                                                                                                                                                                                                                                                                                                                                                                                                                                                                                                                                                                                                                                                                                                                                                                                                                                                                                                                                                                                                                                                                                                                                                                                                                                                                                                                                                                                                                                                                                                                                                                                                                                                                                                                                                                                                                                                                                                                                                                                                                                                                                                                                                                                                                                                                                                                                                                                                                                                                                                                                                                                                                                                                                                                                                                                                                                                                                                                                                                                                                                                                                                                                                                                                                                                                                                                                                                                                                                                                                                                                                                                                                                                                                                                                                                                                                                                                                                                                                                                                                                                                                                                                                                                                                                                                                                                                                                                                                                                                                                                                                                                                                                                                                                                                                                                                                                                                                                                                                                                                                                                                                                                                                                                                                                                                                                                                                                                                                                                                                                                                                                                                                                                                                                                                                                                                                                                                                                                                                                                                                                                                                                                                                                                                                                                                                                                                                                                                                                                                                                                                                                                                                                                                                                                                                                                                                                                                                                                                                                                                                                                                                                                                                                                                                                                                                                                                                                                                                                                                                                                                                                                                                                                                                                                                                                                                                                                                                                                                                                                                                                                                                                                                                                                                                                                                                                                                                                                                                                                                                                                                                                                                                                                                                                                                                                                                                                                                                                                                                                                                                                                                                                                                                                                                                                                                                                                                                                                                                                                                                                                                                                                                                                                                                                                                                                                                                                                                                                                                                                                                                                                                                                                                                                                                                                                                                                                                                                                                                                                                                                                                                                                                                                                                                                                                                                                                                                                                                                                                                                                                                                                                                                                                                                                                                                                                                                                                                                                                                                                                                                                                                                                                                                                                                                                                                                                                                                                                                                                                                                                                                                                                                                                                                                                                                                                                                                                                                                                                                                                                                                                                                                                                                      \renewcommand{\R}{\mathbb R}
\newcommand{\MonSeq}{\operatorname{MonSeq}}

\newcommand{\KK}{\mathbf K}
\newcommand{\ValK}{\operatorname{Val}_{\KK}}
\newcommand{\VALVK}{\operatorname{Val}_{{V_{\KK}}}}

\newcommand{\Val}{\operatorname{Val_T}}

\newcommand{\VAL}{\operatorname{Val}}
\newcommand{\Trop}{\operatorname{Trop}}

\tikzset{labl/.style={anchor=south, rotate=90, inner sep=.5mm}}

\title[A positive/tropical critical point theorem]{A positive/tropical critical point theorem and mirror symmetry}
\author{Jamie Judd and Konstanze Rietsch}

\begin{document}

\begin{abstract}
Call a Laurent polynomial $W$ `complete' if its Newton polytope is full-dimensional with zero in its interior. Suppose $W$ is a Laurent polynomial with coefficients in the positive part of the field  of (generalised) Puiseaux series. Here a Puiseaux or generalised Puiseux series (with exponents in $\R$) is called `positive' if the coefficient of its leading term is in $\R_{>0}$. We show that $W$ has a unique {\it positive} critical point $\pcrit$, i.e. all of whose coordinates are positive, if and only if $W$ is complete.  For any complete, positive Laurent polynomial $W$ in $r$ variables we also obtain from its positive critical point $\pcrit$ a canonically associated `tropical critical point' $\crit\in \R^{r}$ by considering the valuations of the coordinates of $\pcrit$. Moreover we give a finite recursive construction of $\crit$ in terms of a generalisation of the Newton polytope that we call the `Newton datum' of $W$. 

 We show that this result is compatible with a general form of mutation, so that it can be applied in a cluster varieties setting. 
We also show that our theorem carries over to the case where the exponents of monomials appearing in $W$ are not integral but in $\R$, even though $W$ is then no longer Laurent. 
 
Finally, we describe applications to both algebraic and symplectic toric geometry inspired by mirror symmetry. 
On the one hand, in the algebraic context of a complete toric variety $X_{\Sigma}$ we apply our results to obtain for any divisor class $[D]$ satisfying a certain integrality property, a canonical choice of torus-invariant representative. This  generalises the standard toric boundary divisor of $X_\Sigma$ to divisor classes other than the anti-canonical class.    
On the other hand, our result generalises a result of \cite{FOOO:I} and relates to the construction of canonical non-displaceable Lagrangian tori for toric symplectic orbifolds using \cite{FOOO:survey,Woodward:11}.
\end{abstract}

\maketitle

\section{Introduction}

\subsection{}\label{s:intro1} We begin by giving a concrete statement of our main result. Consider the field $\KK$ of generalised Puiseaux series, whose elements are essentially Laurent series in one variable $t$ but with $\R$-exponents that tend to $+\infty$. This field has an $\R$-valued valuation, and it has a positive part $\KK_{>0}$ consisting of those series whose leading term coefficients are positive. Let us write $x^m$  for the Laurent monomial $x_1^{m_1}\dots x_r^{m_r}$, where $m\in\mathbb Z^r$. 

\begin{theorem}\label{t:mainintro} Let $W=\sum \gamma_i x^{v_i}$ be a Laurent polynomial satisfying the  following two properties. 
\begin{enumerate}
\item 
The coefficients $\gamma_i$ lie in $\KK_{>0}$ (`positivity').
\item   
The Newton polytope of $W$, that is the convex hull of $\{v_1,\dots, v_n\}$, is full-dimensional with zero in the interior (`completeness').
\end{enumerate} 
Then $W$ has a unique critical point in $(\KK_{>0})^r$. We call this point the \emph { positive critical point} of $W$.
\end{theorem}
Moreover this theorem is optimal in the sense that any positive Laurent polynomial with a unique positive critical point must be complete, see Corollary~\ref{c:converse}.

Taking the valuation of the positive critical point $\pcrit$ we obtain a distinguished point $d_{\rm crit}$ in $\R^r$ associated to the positive, complete  Laurent polynomial $W$. We call this distinguished point the \emph{tropical critical point} of $W$.

Consider the tropicalization of $W=\sum \gamma_i x^{v_i}$ which is, concretely, the piecewise affine function
\[
\Trop(W)(d)=\min_{i}(\{c_i+\langle v_i,d\rangle\}),
\]
where $c_i=\VAL_{\KK}(\gamma_i)$ and $\langle \, ,\, \rangle$ is the standard inner product on $\R^r$.
Associated to $\Trop(W)$ we have the following subset of $\R^r$, 
\begin{equation}\label{e:PWintro}
\mathcal P_{W}=\{d\in \R^r\mid \Trop(W)(d)\ge 0\}.
\end{equation}
This set is either empty or it is a convex polytope. The second aspect of our result is the following. 

\begin{theorem}\label{t:main2intro}  Suppose $W$ is a complete, positive Laurent polynomial over $\KK$ with tropical critical point $\crit\in \R^r$. Then $\Trop(W)(\crit)$ is the maximal value of $\Trop(W)$. 
In particular whenever $\mathcal P_{W}$ is nonempty then $\crit$ lies in its relative interior.   
\end{theorem}

By Theorem~\ref{t:main2intro}, if $\Trop(W)$ attains its maximum at a unique point, then this point is $\crit$. However even if not, our work includes an  explicit recursive construction of the tropical critical point for any $W$.
To this end we define a `Newton datum' in Section~\ref{s:critconstr}, which we think of as encoding similar but more detailed information than the Newton polytope of $W$. The same Section~\ref{s:critconstr} also contains the construction of $\crit$ in terms of the Newton datum of $W$. This construction is illustrated in Example~\ref{ex:GWexample}.

After proving the main Theorems~\ref{t:mainintro} and \ref{t:main2intro}, and explicitly constructing $\crit$ we extend and strengthen our results in a variety of ways. We show that the positive critical point is defined over Puiseaux series whenever the Laurent polynomial is, so that in particular the tropical critical point is rational in that case. We extend beyond the realm of Laurent polynomials by allowing $W$ to have monomials with rational or real exponents. Moreover we show a non-degeneracy property of the positive critical point. Finally, we describe applications of our results to toric geometry and we show that our positive critical point is compatible with cluster mutation. 
We will describe the  applications in more detail in Section~\ref{s:applicationsintro}, after first giving the background that motivated our results and connection to other work.

\subsection{}\label{s:backgroundintro}
The particular the approach to tropicalisation that we use in  this paper comes from the work of Lusztig who indexed his canonical basis of a quantum universal enveloping algebra using tropicalization and the Langlands dual flag variety \cite{Lusztig94}. Indeed there is a representation theoretic background to this paper which we now describe.

In the context of mirror symmetry for flag varieties there appears to be a favoured anti-canonical divisor -- for example the union of all Schubert and opposite Schubert divisors in a full flag variety. This leads to an interesting question which the second author was asked by Victor Ginzburg: What is the associated line in the $2\rho$-representation? (Given that the elements of the $2\rho$-representation can be interpreted as sections of the anti-canonical bundle by the Borel-Weil construction.) Indeed, from the perspective of representation theory it is surprising that there should be a distinguished line in the middle of the $2\rho$-representation.  
 
An answer to the above question of Ginzburg's is given for $SL_n/B$ in a precursor \cite{Judd:Flag} to this paper and involves computing the tropical critical point of Berenstein and Kazhdan's `potential function’ $W^{\lambda}_{BK}$ from their theory of geometric crystals \cite{BeKa}.  
Namely \cite{Judd:Flag} proves the uniqueness and existence of a positive critical point over $\KK$ for  $W=W^{\lambda}_{BK}$.\footnote{Here $\lambda$ is a dominant weight for $SL_n$ and $W^\lambda_{BK}$ agrees with the superpotential for $SL_n/B$ with quantum parameters $q_i$ set to $t^{\left<\lambda,\omega_i^\vee\right>}\in \KK$, where  $\omega_i^\vee$ is the $i^{\rm {th}}$ fundamental co-weight. In particular, it is represented by a positive Laurent polynomial with coefficients in $\KK$.} Moreover it is shown 
that  this tropical critical point for $\lambda=2\rho$ picks out a lattice point in the associated Gelfand-Zetlin polytope and precisely indexes the special basis element in the $2 \rho$-representation corresponding to the  anti-canonical divisor from above.
We note that the Berenstein-Kazhdan potential function agrees with the superpotential mirror dual to the flag variety given in \cite{R}. 
Thus this result fits into the framework of mirror symmetry. 

The second precursor to the present paper is a paper of Galkin~\cite{Galkin} related to Jacobi rings and quantum cohomology. Like our paper, Galkin's paper also generalises a result for the superpotential of a flag variety. Namely it proves the analogue of our Theorem~\ref{t:mainintro} with $\KK$ replaced by $\R$. This had previously been done specifically for the superpotential of a partial flag variety $SL_n(\C)/P$ (expressed via the Laurent polynomials from \cite{Givental:QToda,BCFKvSPF}) in \cite{rietschNagoya}. 

Thus our main result, Theorem~\ref{t:mainintro}, can be viewed as a simultaneous generalisation of  \cite{Judd:Flag} and \cite{Galkin}. We also note that Galkin's theorem is in fact an ingredient that is used in our construction of the leading term coefficient of $\pcrit$. The extension of Galkin's theorem from $\R$ to $\KK$ is key for applications analogous to Judd's work in the flag variety case, as well as further symplectic applications, see Section~\ref{s:applicationsintro}. In each of the  applications, tropicalisation via the valuation on generalised Puiseaux series plays a central role.  

After proving the Theorem we also learned of a third precursor. Namely, the problem of constructing of critical points of superpotentials over generalised Puiseaux series arises in the context the study of non-displaceable Lagrangians in compact toric manifolds \cite{FOOO:survey}. In this context our theorem strengthens a result of Fukaya-Oh-Ohta-Ono \cite[Proposition~4.7]{FOOO:I}, which deals with a similar scenario but with an additional assumption for analyticity.

\subsection{}\label{s:applicationsintro} We now outline three applications of the results from Section~\ref{s:intro1}.

\subsubsection{}\label{s:applicationsintro1}
 The first application of Theorems~\ref{t:mainintro} and \ref{t:main2intro} is an analogue for toric varieties of the result for flag varieties in \cite{Judd:Flag}. Namely if $X=X_\Sigma$ is a toric variety \cite{Fulton:toricbook} with a torus-invariant Weil divisor $D$, then one can associate to $X$ and $D$ a Laurent polynomial mirror $W=W_{(X,D)}$. Its summands correspond to primitive generators of the rays of the fan $\Sigma$, and the coefficients (in $\KK_{>0}$) depend on $D$. (See also Section~\ref{s:applicationsintro2} below for the analogous Laurent polynomial in the symplectic toric setting and related references.) If $X$ is complete then $W_{(X,D)}$ is complete and thus has a canonical positive critical point $\pcrit$. 
In Section~\ref{s:toric} we use the associated tropical critical point $\crit$ of $W_{(X,D)}$ to construct a distinguished torus-invariant divisor in the divisor class of $D$, whenever $\crit$ is a lattice point. In the case that $D$ is anti-canonical the special divisor in the class of $D$ will always be the toric boundary divisor, and this is the analogue of the special vector in the $2 \rho$-representation of Section~\ref{s:backgroundintro}.

\subsubsection{}\label{s:applicationsintro2}

The second application is related to Lagrangian torus fibers in compact symplectic toric orbifolds, see \cite{FOOO:I, FOOO:II} 
and \cite{Woodward:11}.   
Recall that a Lagrangian $L$ is called non-displaceable if $\psi(L)\cap L\ne \emptyset$ for any Hamiltonian diffeomorphism $\psi$ of $X$.  For example for $X=\mathbb C P^r$, viewed as a symplectic toric manifold with moment map $\mu:\mathbb C P^r \to \Delta_r$ given by 
\[[z_0:\dots : z_r] \mapsto   \left (\frac{|z_1|^2}{\sum_{i=0}^r |z_i|^2},\dots, \frac{|z_r|^2}{\sum_{i=0}^r |z_i|^2}\right ),\]
there is a distinguished Lagrangian torus fibre $L$  which is a {\it Clifford torus} and which is non-displaceable by \cite{BiranEntovPolterovich:04}. It is the fiber of a special point in the moment polytope $\Delta_r$, namely $L=\mu\inv(\frac 1{r+1},\dots, \frac 1{r+1})$. In particular for $r=1$ the fiber $L$ is a great circle. (Note that the great circle $L$ is readily seen to be non-displaceable since any Hamiltonian diffeomorphism $\psi$ preserves the areas of the two disks bounded by $L$, so that $\psi(L)$ cannot lie entirely on one side or the other of $L$.)

This  toric symplectic manifold (or orbifold) $X$ has a canonically associated `gauged potential' function $W_X$ due to Woodward, which is defined geometrically in terms of equivariant  holomorphic disk counts and worked out concretely in \cite{Woodward:11}. The explicit function $W_X$ is a positive Laurent polynomial and, if $X$ is projective, then it is also `complete' in the sense of Theorem~\ref{t:mainintro}. Its formula agrees with mirror symmetric superpotentials written down by  \cite{Batyrev:QCoh,Givental:Toric} and also with \cite[(5.16)]{Hori-Vafa} and the `leading order potential function' for a toric symplectic manifold of \cite{FOOO:I} -- but the Woodward gauged potential does not include the possibly infinitely many extra terms that can occur in the full potential function of \cite{FOOO:I}. We refer to \cite[Section~1]{GonzalezWoodward:Adv19} for an overview. In particular $W_X$ has the property that the polytope $\mathcal P_{W_X}$ from \eqref{e:PWintro} recovers the moment polytope $P$ of $X$. Moreover, if $X$ is a projective toric variety embedded via very ample $T$-invariant Cartier divisor $D$, then $W_X=W_{(X,D)}$ from Section~\ref{s:applicationsintro1}.

By applying our Theorems~\ref{t:mainintro} and \ref{t:main2intro} to $W_X$ we obtain a canonical point  in the  moment polytope $P$ of $X$, namely the tropical critical point $\crit$ of the Laurent polynomial $W_X$. 
 It follows by results of \cite{Woodward:11,FOOO:I, FOOO:II} that the moment map fiber $L_X:=\mu_X\inv(\crit)$ is a non-displaceable Lagrangian torus in $X$. Thus the Lagrangian $L_X$   can be thought of as a canonical `central' non-displaceable Lagrangian torus in $X$, in a sense generalising the Clifford torus $L$ in $\mathbb {C P}^r$.

In the case where $(X,\omega)$ is a  smooth symplectic toric manifold we also see that our tropical critical point $\crit$ of the leading order potential function agrees with a point $u_0$ constructed in the moment polytope $P$ in \cite{FOOO:I, McDuff:probes} which was proved to have a non-displaceable Lagrangian torus fiber in \cite{FOOO:I}. We note that this marks $u_0$ out as a very special point.
Indeed, our result implies a characterisation of $u_0$, as the valuation of the unique positive critical point of $W_X$. If $X$ is additionally monotone, meaning that the first Chern class $c_1(X)$ is a positive multiple of the class $[\omega]\in H^2(X)$, then the special point $\crit$ recovers the well-known central point of the moment polytope $P$ of $X$, see \cite{AbreuMacarini:TAMS}. Moreover in the monotone surface case the fiber of this central point is the only non-displaceable Lagrangian torus fiber in $X$ as the others can all be explicitly displaced by `probes', see  \cite{McDuff:probes}. In higher dimensions the uniqueness of the non-displaceable torus fiber is a conjecture.
For more general $X$ there can be other non-displaceable Lagrangian torus fibers, even families of them, and it would be interesting to know if the one associated to $\crit$ enjoys any other special properties that might characterise it.

We remark that this application to Lagrangians in toric symplectic manifolds also has an analogue for the flag manifold $U(n)/T_c$ which is due to Nishinou, Nohara and Ueda. Namely  the fiber under the Gelfand-Zetlin moment map of the center of a Gelfand-Zetlin polytope is a non-displaceable Lagrangian torus, see  \cite{NNU:GC}. 

\subsubsection{}\label{s:clusterintro}
Laurent polynomials arise in mirror symmetry for toric varieties, but they also play a role for more general varieties which have a toric degeneration. In this setting, because the toric degenerations will not be unique, it is natural to expect a multitude of Laurent polynomials associated to the same variety. The way one generally tries to makes sense of a plethora of Laurent polynomials that one thinks of as mirror to a single variety, is by relating these various Laurent polynomials to one another through birational transformations called {\it mutations}. This can be mutation as in the sense of cluster varieties \cite{FZ1, GHKK}; for an explicit example related to mirror symmetry see \cite{RW}.  Or it can be a more general type of mutation, as in \cite{ACGK:mutations,CKPT}. We show that as long as mutation preserves positivity and Laurentness of our function $W$, then it also preserves the positive critical point. This opens the way to an extension of the above applications from toric varieties to other settings such as non-toric Fano varieties.

\subsection{}
The paper is organised as follows. An important part of our set-up is the construction of the topicalisation of a torus via a local field with a positive semifield as used by Lusztig. This construction is laid out in Section~\ref{s:tropicalisation}.  Then in Section~\ref{s:theorems} we are able to restate the two main theorems from above in a coordinate-free way.

The longest section is Section~\ref{s:mainproof}, which is devoted to the proof of the theorems stated in Section~3. The first three subsections of Section~\ref{s:mainproof}, culminating in the key Corollary~\ref{c:valcrit}, are devoted to determining $\crit$ under the assumption that $\pcrit$ exists. Section~\ref{s:leadingcoeff} is then concerned with constructing the coefficient of the leading term of $\pcrit$. Finally Section~\ref{s:extension} shows that the leading term which was constructed in the previous subsections extends to a well-defined solution $\pcrit$ of the critical point equations of $W$, and moreover that this solution is unique. This proves the first theorem. The second theorem is deduced from results proved along the way. 

The next main section, Section~\ref{s:rationalityetc} contains the various extensions and refinements of the main Theorem.
The application to toric varieties is contained in Section~\ref{s:toric}. Then in 
Section~\ref{s:symplectic} we give the symplectic application in connection with \cite{FOOO:survey}. 
Finally, in Section~\ref{s:mutations} we show that the positive critical point is preserved under mutations. 
\subsection{Acknowledgements}
The authors thank Mohammad Akhtar for useful discussions at the start of this project. The second author also thanks Denis Auroux, Agnes Gadbled, Sergey Galkin, Yank\i\ Lekili, Dima Panov and Lauren Williams and Chris Woodward for helpful conversations.

\section{Positivity for tori and tropicalisation}\label{s:tropicalisation}

Suppose $T$ is an algebraic torus of dimension $r$ over a field $K$. Consider the character group $M:=X^*(T)$ and cocharacter group $N:=X_*(T)$ of $T$ (with the group structure written additively). We have $M\cong \Z^r$ and $N\cong \Z^r$ and a  dual pairing $\langle\ ,\ \rangle:M\x N\to\Z$ which extends to a dual pairing of the real vector spaces $M_\R=M\otimes \R$ and $N_\R=N\otimes \R$. 

Suppose for a moment that $K=\C$. Consider the category $\mathcal T$ whose objects are algebraic tori over $K$, and whose morphisms $T^{(1)}\to T^{(2)}$ are rational maps with positive real coefficients, with regard to a/any choice of bases of characters for $T^{(1)}$ and $T^{(2)}$.  
The {\it tropicalisation functor} $\Trop$ in this setting is a functor
\[
\Trop:\mathcal T\to\mathcal {PL}
\] 
from $\mathcal T$ to the category $\mathcal {PL}$ of  finite-dimensional real vector spaces with piecewise linear (PL) maps. Informally $\Trop$ associates to a torus $T$ the vector space $N_\R$, and to a subtraction-free rational map $T^{(1)}\to T^{(2)}$, the PL map $N^{(1)}_\R\to N^{(2)}_\R$ obtained in suitable coordinates by replacing multiplication by addition and addition by $\min$.   

In the following two subsections we give a more intrinsic description of tropicalisaton  following an original construction due to Lusztig~\cite{Lusztig94}.

\subsection{Positivity and the field of generalised Puiseux series}\label{s:genPuiseux}

 Let $K$ be an infinite field and assume $K$ has a `positive' subset $K_{>0}$, satisfying 
\[k, k'\in K_{>0} \implies \ k+k'\in K_{>0}, \ \ k k'\in K_{>0},\ \  1/k\in K_{>0}, \]
compare \cite[Section 2.1]{Lusztig94}. Observe that in particular $K_{>0}$ is a subgroup of the multiplicative group $K^*$ of $K$. Examples include $K=\C$ with $K_{>0}=\R_{>0}$, or $K=\C((t))$ where $K_{>0}$ consists of Laurent series with positive leading term coefficient.  Our preferred example is the following field of generalised Puiseaux series, also referred to as the `universal Novikov field' in \cite{FOOO:survey}. 

\begin{defn}[Generalised Puiseux series~\cite{Markwig:genPuiseux}]\label{d:genPuiseux}
Let $\mathbf K$ denote the field of {\it generalised 
Puiseux series} in a variable $t$. These are series whose exponent sets are described by
\[
\MonSeq=\{ A\subset \R \mid  \operatorname{Cardinality}(A\cap \R_{\le x})<\infty \text{ for arbitrarily large $x\in \R$} \}.
\]
In other words the exponent sets $A\in \operatorname{MonSeq}$ may be thought of as strictly monotone sequences which are either finite, or  are countable and tending to infinity. We write $(\mu_k) \in \MonSeq$ if $(\mu_k)_k=(\mu_0, \mu_1,\mu_2,\dotsc )$ is such a strictly monotone increasing sequence, and then
\begin{equation}\label{e:K}
\KK=\left\{c(t)=\sum_{(\mu_k) \in \MonSeq} c_{\mu_k} t^{\mu_k}\mid c_{\mu_k}\in\C \right\}.
\end{equation}
The field $\KK$ has a positive subset in the above sense given by
\begin{equation}\label{e:Kpos}
\KK_{>0}=\left\{c(t)\in \KK\mid c(t)=\sum_{(\mu_k) \in \MonSeq} c_{\mu_k} t^{\mu_k} , \ c_{\mu_0}\in \R_{>0} \right\}.
\end{equation}
Consider the  valuation, 
$%
\ValK:\KK\to \R\cup\{\infty\}$,
defined on nonzero elements by
$\ValK\left(c(t)\right)=\mu_0$
if $c(t)=\sum  
c_{\mu_k} t^{\mu_k}$, where the lowest order term is assumed to have non-zero coefficient, $c_{\mu_0}\ne 0$, and we set $\ValK(0)=\infty$. We also consider the associated local ring
\[
\mathcal O_{\KK}:=\{c\in\KK\mid\ValK(c)\ge 0\},
\]
with its maximal ideal $\mathfrak m=\{c\in\KK\mid\ValK(c)> 0\}$.
\end{defn}

The field $\KK$ is algebraically closed, and it is complete for the `$t$-adic topology' induced from the norm associated to this valuation, see \cite{Markwig:genPuiseux}.

\begin{defn}\label{d:Coeff}We define a group homomorphism $\Coeff:\KK\setminus\{0\}\to \C^*$ which takes $c(t)$ to the coefficient of its lowest order non-vanishing term
\begin{equation}\label{e:Coeff}
\Coeff(c(t)):=c_{\mu_0}.
\end{equation}
\end{defn}

\begin{remark}
The field $\KK$ is a completion of the usual field of Puiseux series $\mathcal K=\bigcup_{\ell\in \Z_{>0}}\C\left(\left(t^{\frac 1 {\ell}}\right)\right)$. Furthermore the field of Puiseux series $\mathcal K$ is the algebraic closure of the field of Laurent series $\mathbf L=\C((t))$. Both $\mathcal K$ and $\mathbf L$ have positive parts which are described by $\mathcal K_{>0}=\mathcal K\cap \KK_{>0}$ and $\mathbf L_{>0}=\mathbf L\cap \KK_{>0}$. And $\ValK$ restricts to the usual valuations $\operatorname{Val}_{\mathcal K}:\mathcal K\setminus\{0\}\to\mathbb Q$ and $\operatorname{Val}_{\mathbf L}:\mathbf L\setminus\{0\}\to\Z$.
\end{remark}

\subsection{Algebraic tori over $\KK$ and tropicalisation}\label{s:tptropicalisation}
Suppose $K$ is an infinite field with a positive part $K_{>0}$, as in Section~\ref{s:genPuiseux}. Recall that $T$ is an algebraic torus of dimension $r$ over $K$.
We can describe the $K$-valued points of $T$ as group homomorphisms from the character group to the multiplicative group of $K$, namely we identify $T(K)=\Hom(M,K^*)$. The group homomorphisms which take values in $K_{>0}$ define a positive part $T(K_{>0})=\Hom(M,K_{>0})$ of $T(K)$. If $v\in M$ and $p\in T(K)$ we will write either $p^v$ or $\chi^v(p)$ for the associated evaluation ``$p(v)$'' in $K^*$. We call $\chi^v$ the (multiplicative) character associated to $v$. 

We will be primarily interested in the case where $K=\KK$, where we define the tropicalization functor using an identical approach to the one introduced Lusztig (albeit in his case in the setting of the subfield $\mathbf L$ of Laurent series in $t$).   

By a {\it Laurent polynomial} on $T$ we will mean a $\KK$-linear combination of characters $\chi^v$. We may choose a basis of characters so that the coordinate ring of $T$ is described as $\KK[x_1^{\pm 1}, \dotsc, x_r^{\pm 1}]$ in which case this definition recovers Laurent polynomials in the variables $x_i$. 
A {\it positive Laurent polynomial}  is a linear combination of characters  with coefficients in $\KK_{>0}$.  Suppose $T^{(1)}$ and $T^{(2)}$ are tori over $\KK$. By a {\it positive rational map} $
\phi:T^{(1)}\to T^{(2)}$ we mean a rational map such that for any character $\chi$ of $T^{(2)}$ the composition $\chi\circ \phi:T^{(1)}\to \KK$ is given by a quotient of positive Laurent polynomials on $T^{(1)}$.

\begin{defn}[The tropicalisation of the torus $T$]\label{d:TropTorus} Define an equivalence relation $\sim$ on $T(\KK_{>0})$ by $x\sim x'$ if  and only if $\ValK(\chi(x))=\ValK(\chi(x'))$ for all characters $\chi$ of $T$. 
Then
\[
\Trop(T):=T(\KK_{>0})/\sim.
\]
The set $\Trop(T)$ inherits from the group structure of $T(\KK_{>0})$ a structure of abelian group (which we denote as addition).
\end{defn}

\begin{defn}[The map $\Val$ and identifying $\Trop(T)$ with $N_\R$]\label{d:TropN} 
Let $\Val$ denote the group homomorphism from the multiplicative group $T(\KK_{>0})$ to the additive group $N_{\R}$, 
\[
\Val:T(\KK_{>0})\to N_\R,
\]
which is characterised by the property that for any character $\chi$ of $T$ and $x\in T(\KK_{>0})$ the $\KK$-valuation of $\chi(x)$ satisfies
\[
\ValK(\chi(x))= \langle\chi,\Val(x)\rangle,
\]
where $\langle\ ,\ \rangle$ is the pairing between $M_\R$ and $N_\R$. Note that $\Val(x)$ is well-defined, as follows for example by choosing a basis for $M$ and using the additivity property $\ValK(\chi^{m_1+m_2}(x))=\ValK(\chi^{m_1}(x))+\ValK(\chi^{m_2}(x))$.

Since, by definition, $\Val(x)$ depends only on the $\KK$-valuations $\ValK(\chi(x))$, we observe that the map $\Val$ descends to a homomorphism,
$\Trop(T)\to N_\R$, which we may also call $\Val$, by abuse of notation.

The map $\Trop(T)\to N_\R$ defined by $\Val$ is an isomorphism. Its inverse, 
\[
\iota:N_\R \To \Trop(T) : d\mapsto [x_{\mon}(d)],
\]
is defined by sending $d\in N_\R$ to the equivalence class of the `monomial' element $x_{\mon}(d)$ in $T(\KK_{>0})$, which is characterised by 
\[
\chi(x_{\mon}(d))=t^{\langle\chi,d\rangle}
\quad \text {for all}\quad \chi\in X^*(T).
\] 
This bijection endows $\Trop(T)$ with the structure of an $\R$-vector space. 
\end{defn}

\begin{defn}[The tropicalisation of a positive rational map]\label{d:tropmap} Suppose $\phi:T^{(1)}\to T^{(2)}$ is a positive rational map between two tori. Then $\phi(\KK_{>0}):T^{(1)}(\KK_{>0})\to T^{(2)}(\KK_{>0})$ is everywhere well-defined and is compatible with the equivalence relation $\sim$. Note that the compatibility with $\sim$ uses the positivity of the leading terms. The tropicalisation $\Trop(\phi)$ is defined to be  the resulting map 
 \[\Trop(\phi):\Trop(T^{(1)})\to\Trop( T^{(2)})\]
between equivalence classes. The map $\Trop(\phi)$ is piecewise-linear with respect to the linear structures on the $\Trop(T^{(i)})$ from Definition~\ref{d:TropN}.
\end{defn}

\begin{example}\label{ex:W} A positive Laurent polynomial $W=\sum_{i=1}^n\gamma_i x^{v_i}$, where $x^{v_i}=\prod_{j=1}^r x_j^{v_{i,j}}$ and $\gamma_i\in\KK_{>0}$, can be considered as a positive rational map from the torus $(\KK\setminus \{0\})^r$ to $\KK\setminus \{0\}$, and its tropicalisation can be identified with the piecewise linear map $\R^r\to \R$ given by
\[
\Trop(W)(d_1,\dotsc, d_r)=\min (\{c_i + \sum_{j=1}^{r} v_{i,j}d_j\mid i=1,\dotsc n \}),
\]
where $c_i:=\ValK(\gamma_i)$.

\end{example}

\subsection{Leading terms and exponential map for $T(\KK_{>0})$}\label{s:exp}
For any point $p\in T(\KK_{>0})$ the map which associates to $v\in M$ the leading term of $\chi^v(p)$ defines a group homomorphism,  
\[
p_0 \in\Hom(M,\{c t^\mu\mid c\in\R_{>0},\mu\in\R\}).
\]
We call $p_0$ the {\it leading term} of $p$. 
Observe that $p_0$ lies in the group  
\[
\Hom(M,\{c t^\mu\mid c\in\R_{>0},\mu\in\R\})\subset \Hom(M,\KK_{>0})=T(\KK_{>0}),
\] 
which we call the `leading term subgroup' of $T(\KK_{>0})$. This subgroup is a product of the groups $T(\R_{>0})$ and $\Hom(M,\{ t^\mu\mid \mu\in\R\})$, and the map $p\mapsto p_0$ is a surjective group homomorphism, which is in a sense a projection from $T(\KK_{>0})$ to this subgroup.
 
Recall that $\m=\{\gamma\in \KK\mid  \ValK(\gamma)>0\}$. A general element $p$ of $T(\KK_{>0})$ is the product of its leading term $p_0$ and a factor from 
\begin{equation*}
T_e(\KK_{>0}):=\{p\in T(\KK_{>0})\mid \chi^v(p)\in 1 +\mathfrak m
\text{ for all $v\in M$}\}.
\end{equation*}

  We may describe an element of $T(\KK_{>0})$ entirely in `logarithmic terms' using the factorisation above and the following exponential maps.

\begin{defn}[$u\mapsto e^u$]
We have the  exponential map
$N_\R\to T(\R_{>0})
$
which sends $u\in N_\R$ to the element $e^u$ defined by the property that $\chi^v(e^u)=e^{\langle v,u\rangle}
$ for all $v\in M$. This is just the usual exponential map of the real Lie group $T(\R_{>0})$, and it is
an isomorphism.
\end{defn}

\begin{defn}[$d\mapsto t^d$]\label{d:t^d} We also have an analogous isomorphism
\begin{eqnarray*}
N_\R&\to &\Hom(M,\{ t^\mu\mid \mu\in\R\}),
\end{eqnarray*}
where the image $t^d$ of $d\in N_\R$ is defined by $\chi^v(t^d)=t^{\langle v,d\rangle}$, for all $v\in M$.
\end{defn}

Observe that the leading term group is isomorphic via the above two exponential maps (or rather their inverse maps) to $N_\R\oplus N_\R$, and any leading term $p_0$ of a $p\in T(\KK_{>0})$ is of the form $e^u t^d$ for unique $(u,d)\in N_\R\oplus N_\R$. 

\begin{defn}[$\exp_T$ and $\log_T$]\label{d:expm}
Let $\exp:\m\to \KK_{>0}$ be
 the exponential map of $\KK$, which is defined in terms of its power series. Its image is $\KK_{1}:=\{k\in\KK_{>0}\mid k\in 1+\mathfrak m
  \}$, and we have an inverse map $\log: \KK_1\to \m$ defined in terms of the power series for the logarithm.
 
Suppose  $N_\m=N\otimes_\Z \m$. We let $\exp_T:N_{\m}\to T_e(\KK_{>0})$ be the map $w\mapsto\exp_T(w)$ defined by the property that for any $v\in M$ with associated character $\chi^v$,  
\[
\chi^v(\exp_T(w))=\exp(\langle v, w\rangle).
\]
Here, by abuse of notation, $\langle\ ,\ \rangle:M\x N_{\m}\to \m$ is the $\m$-linear extension of the pairing between $M$ and $N$. 

The character $\chi^v(\exp_T(w))$ always has valuation $0$ and leading coefficient $1$ for any $w\in N_\m$. 
Indeed, the map $\exp_T$ is an isomomorphism from the additive group $N_{\m}$ to the multiplicative group $T_e(\KK_{>0})$,
and has an inverse $\log_T:T_e(\KK_{>0})\to N_\m$.
\end{defn}

Combining the three types of `exponential map' above we obtain a group isomorphism
\begin{eqnarray}\label{e:totalexp}
N_\R\oplus N_\R \oplus N_\m &\overset\sim\longrightarrow &T(\KK_{>0})\\ \nonumber
(u,d,w) &\mapsto  &e^u t^d \exp_T(w).
\end{eqnarray}
The inverse of this map is defined by $p\mapsto (u,d,\log_T(e^{-u}t^{-d}p))$,  where the leading term $p_0$ of $p$ is $e^u t^d$.

\begin{defn} 
Suppose $p\in T(\KK_{>0})$ and $p=e^u t^d\exp_T(w)$, as in \eqref{e:totalexp} above. Then we note that $d=\Val(p)$. We also introduce a notation for  $u$, namely we call $u$ the `logarithmic leading coefficient' of $p$ and write $u=\Log(p)$.

In terms of coordinates, the map $\Log:T(\KK_{>0})\to N_\R$ is simply given by 
\[(c^{(1)}(t),\dotsc, c^{(d)}(t))\mapsto (\log(c^{(1)}_{m^{(1)}_0}),\log(c^{(2)}_{m^{(2)}_0}),\dotsc, \log(c^{(d)}_{m^{(d)}_0})),
\]
where $c^{(j)}_{m^{(j)}_0}=\Coeff(c^{(j)}(t))\in \R_{>0}$, the coefficient of the leading term of  $c^{(j)}(t)$.

\end{defn}

Thus for $p\in T(\KK_{>0})$, the leading term $p_0$ of $p$ is given by
\[
p_0=e^{\Log(p)}t^{\Val(p)}, 
\]
and any $p\in T(\KK_{>0})$  is of the form $e^{\Log(p)}t^{\Val(p)}\exp(w)$ for some (unique) $w\in \mathfrak m$.

\begin{defn}\label{d:rationalT} For the purpose of rationality considerations (Section~\ref{s:rationality}),
 we also make the analogous definitions over the field  $\mathcal K$ of Puiseaux series,
\[
\mathcal K:=\bigcup_{n\in\Z_{>0}}\C[[t^{\frac 1n}]]\subset \KK.
\]
Namely, this field comes with positive part $\mathcal K_{>0}=\mathcal K\cap \KK_{>0}$, and valuation $\VAL_{\mathcal K}\to \Q\cup\{\infty\}$, and associated ring of integers $\mathcal O_{\mathcal K}$ and maximal ideal $\mathfrak m_{\mathcal K}$. Moreover we note that we have an analogue   $T(\mathcal K_{>0})$ of $ T(\KK_{>0})$ which can be described by
\begin{eqnarray}\label{e:totalexpPuiseaux}
%\EXP_T: 
N_\R\oplus N_{\mathbb Q} \oplus N_{\mathfrak m_{\mathcal K}} &\overset\sim\longrightarrow &T(\mathcal K_{>0})\\ \nonumber
(u,d,w) &\mapsto  &e^u t^d \exp_T(w).
\end{eqnarray}
\end{defn}

%\begin{remark}We think of $\Log$ as an analogue of the map $\Val:T(\KK_{>0})\to N_\R$, but relating to the {\it coefficients} of the leading terms as opposed to the leading terms themselves. 
%We refer to $\Log(p)$ as the {\it logarithmic leading coefficient} of $p\in T(\KK_{>0})$. \end{remark}

%\begin{defn}[Logarithmic leading coefficient]\label{d:Log} Recall the `leading coefficient' map $\Coeff:\KK\to \C^*$ from Definition~\ref{d:genPuiseux} and note that it restricts to a map  $\Coeff:\KK_{>0}\to \R_{>0}$.  We define a new map $\Log: T(\KK_{>0})\to N_\R$ as follows.  Suppose $p\in T(\KK_{>0})$. We define $\Log(p)$ to be the unique element of $N_\R$ such that for every $v\in M$, the character $\chi^v$ satisfies
%\begin{equation}\log(\Coeff(\chi^v(p)))=\langle v,\Log(p)\rangle.\end{equation} 
%The leading term $p_0$ of $p$ can then be written as \[p_0=e^{\Log(p)}t^{\Val(p)}, \]
%compare Section~\ref{s:exp}.
%\end{defn}

\section{The positive critical point theorem}\label{s:theorems}

In this section we state our main result.

 \begin{defn}[Positive and tropical critical points]\label{def:TropCritPt}
Suppose $W$ is a Laurent polynomial on the torus $T$ over $\KK$. We say a point $p\in T(\KK)$ is a critical point for $W$ if the gradient of $W$ (in terms of some/any coordinates $x_i$) vanishes at $p$. This property of $p$ is independent of the choice of the coordinates. If $p$ lies in $T(\KK_{>0})$ we say that $p$ is a {\it positive critical point} of $W$.
%If $W$ is a positive Laurent polynomial then there exists a tropicalisation of $W$, see Example~\ref{ex:W}. If $p$ is a critical point of $W$ and lies in $T(\KK_{>0})$, then we call $p$   a {\it positive critical point} of $W$. 
For a positive critical point $p$, we call the equivalence class $[p]\in\Trop(T)$ a {\it tropical critical point} of $W$.
\end{defn}

\begin{defn}[Complete Laurent polynomials]\label{d:good} Suppose we have identified regular functions on $T$ as Laurent polynomials in variables $x_1,\dotsc, x_r$. Then associated to a  positive Laurent polynomial $W=\sum_{i=1}^n \gamma_i x^{v_i}$ with $\gamma_i\in\KK_{>0}$ and $v_i\in \Z^{r}$, we consider its Newton polytope
\[
\operatorname{Newton}(W)=\operatorname{ConvexHull}\left(\{v_i\mid i=1\,\dotsc, n\}\right)\subset \R^r.
\]
Let us call $W$ {\it complete} if its Newton polytope is $r$-dimensional with zero in its interior. This polytope can be considered to be in $M_\R$, without choice of coordinates. The property of a Laurent polynomial being {\it complete} is independent of the choice of basis of characters of $T$. By abuse of notation we may, also in the absence of chosen coordinates $x_i$, write $x^v$ for the function $T(\KK)\to\KK$ associated to $v\in M$, and $p^v$ for the value of $x^v$ on $p\in T(\KK)$. 
\end{defn}

For Laurent polynomials with coefficients in $\R_{>0}$ whose Newton polytopes are complete  in the above sense, it was shown by Galkin that there exists a unique positive critical point, referred to by him as the `conifold point'.

\begin{prop}\label{p:Galkin}\cite{Galkin}
Suppose $L\in \R[x_1^{\pm 1},\dotsc, x_r^{\pm 1}]$ is a Laurent polynomial with positive coefficients, and that  the Newton polytope of $L$ is full-dimensional and contains $0$ in its interior. Then $L$ has a unique critical point in $\R_{>0}^r$. This critical point is non-degenerate and a global minimum for $L|_{\R_{>0}^r}$.
\end{prop}

Our main theorem is an analogue of Galkin's result over the field $\KK$ of generalised Puiseux series. 

\begin{theorem}\label{t:main} Suppose $T$ is a  torus over $\KK$ and $W$ is positive Laurent polynomial on $T$ which is complete  in the sense of Definition~\ref{d:good}. Then $W$ has a unique positive critical point $\pcrit\in T(\KK_{>0})$.
%Moreover $\pcrit$ is a non-degenerate critical point with Hessian in $\KK_{>0}$. 
\end{theorem}

As an application of this theorem we obtain from any positive, complete  Laurent polynomial $W$ a distinguished point in $\Trop(T)$ given by the equivalence class $[\pcrit]$ of the positive critical point.  Equivalently, we have that $\Val(\pcrit)$ is a point in $N_\R$ canonically associated to $W$. We call this point {\it the tropical critical point} of $W$.  

Let us consider associated to $W$ the  following subset of $N_\R$, 
\begin{equation}\label{e:PW}
\mathcal P_{W}=\{d\in N_\R\mid \Trop(W)(d)\ge 0\}.
\end{equation}
This set is either empty or it is a convex polytope. We also show the following. 

\begin{theorem}\label{t:critmax}
Suppose $W$ is a complete, positive Laurent polynomial over $\KK$ with tropical critical point $\crit\in N_\R$. Then we have 
\[
\Trop(W)(\crit)=\max_{d\in N_\R}(\Trop(W)(d)).
\] 
In particular whenever $\mathcal P_{W}$ is nonempty then $\crit$ lies in its relative interior.  
\end{theorem}

\section{Proof of Theorem~\ref{t:main} and  Theorem~\ref{t:critmax}}\label{s:mainproof}

We begin the proof of Theorems~\ref{t:main} and \ref{t:critmax} by choosing notation. Let $T$ be an $r$-dimensional algebraic torus over $\KK$ with character lattice $M$ and cocharacter lattice $N$. Let $x^{v}$ denote the function on the torus $T$ associated to $v\in M$. We consider a  positive Laurent polynomial, 
% that is, a regular function on $T(\KK)$, which we denote
\begin{equation}\label{e:genW}
W=\sum_{i=1}^n \gamma_i x^{v_i},
\end{equation}
with $v_i\in M$ assumed to be distinct, and coefficients $\gamma_i\in \KK_{>0}$. Recall that 
%We assume that $W$ is `complete', that is, the Newton polytope
\[
\Newton(W)=\ConvexHull(\{v_i\mid i=1,\dotsc, n\})\subset M_\R.
\]
%is $r$-dimensional and has $0$ in its interior. 
Since $W$ is positive we consider its tropicalisation in the sense of Section~\ref{s:tropicalisation}.  This tropicalisation   is the piecewise linear map $\Trop(W):N_\R\to \R$ explicitly given by 
\begin{equation}\label{e:WtropProofSection}
\Trop(W)(d)=\min (\{c_i + \langle v_i,d\rangle\mid i=1,\dotsc, n \}),
\end{equation}
where $c_i:=\ValK(\gamma_i)$ and $\langle\ ,\ \rangle$ is the dual pairing between $M$ and $N$.

Our main goal is to construct a critical point $\pcrit\in T(\KK_{>0})$ for $W$ and show that  it is unique. The first three subsections will be concerned with determining the valuation $\Val(\pcrit)\in N_\R$, or equivalently $\ValK(\pcrit^v)$ for all  $v\in M$. 
\subsection{The augmented Newton polytope and the maximum value of $\Trop(W)$.}

We note that the Newton polytope of $W$ depends only on the exponent vectors $v_i$ and takes no account of the coefficients  $\gamma_i$ or their valuations. It is useful to think of $\Newton(W)$ as the projection of a more general polytope. The following polytope associated to $W$ generalises the `t-Newton polygon' from \cite{Markwig:genPuiseux} and comes up for a different purpose in \cite{Maclagan:intropaper, Sturmfels:notes2}.

\begin{defn}[Augmented Newton Polytope] \label{d:aug} For $W$ given by \eqref{e:genW}, let 
\[
\AugNewton(W):=\ConvexHull(\{(c_i,v_i)\mid i=1,\dotsc, n\})\subset \R\oplus M_\R.
\]
We refer to $\AugNewton(W)$ as the {\it augmented Newton polytope} associated to $W$. It comes with a projection 
\begin{align*}
\pr:\AugNewton(W)&\quad\to\quad \Newton(W)\\
(c,v)&\quad\mapsto\quad  v.
\end{align*}
\end{defn} 

We now show for later use how the augmented Newton polytope encodes the maximal value of $\Trop(W)$.

\begin{remark}\label{r:AugNewton} To aid with intuition, let us interpret the points of $\AugNewton(W)$ as linear functions on $\R\oplus N_\R$, or, by restriction, as functions on $\{1\}\x N_\R$. In particular note that for a vertex $(c_i,v_i)$ this restriction gives $(1,d)\mapsto c_i+\langle v_i,d\rangle$.  The map $\Trop(W):N_\R\to \R$, see \eqref{e:WtropProofSection}, can now be interpreted as taking the (pointwise) min of {\it all} of these functions associated to points of $\AugNewton(W)$. Namely,
\[
\Trop(W): d\mapsto \min_{w\in\AugNewton(W)} \langle w,(1,d)\rangle.
\]            
This defines the same function as \eqref{e:WtropProofSection}, since the above $\min$ will always be attained on some extremal $w=(c_i,v_i)$. 

If $W$ is complete, then this implies that $\AugNewton(W)$ has a nonempty intersection with the line $\R\oplus\{0\}$. As a consequence the function $\Trop(W)$ is bounded from above by the (constant) function  
\[
d\mapsto \min_{(c,0)\in\AugNewton(W)} \langle (c,0) ,(1,d)\rangle=\min_{(c,0)\in\AugNewton(W)} c.
\] 
The main lemma of this section shows that this bound is best possible.
\end{remark}

\begin{defn}\label{d:minimalheightAugNewton}
Given a complete  Laurent polynomial $W$ with its associated polytope $\AugNewton(W)$ we refer to
\[\tau=\min(\{c\in\R\mid (c,0)\in\AugNewton(W)\})
\]
as the {\it minimal height} above $0$ of $\AugNewton(W)$. We refer to $(\tau,0)$ as the {\it lowest point} above $0$ in $\AugNewton(W)$. 
\end{defn}

\begin{lemma}\label{l:max}
Let $W$ be  any positive, complete  Laurent polynomial, and let $\tau$ be the minimal height above $0$ of $\AugNewton(W)$.
Then 
\[
\tau =\max_{d\in N_\R} \Trop(W)(d).
\]
\end{lemma}

To prove this lemma we require the notion of a lowest face of a polytope with respect to a vector in the dual space. 

\begin{defn}[Lowest face]\label{d:lowestface} Suppose $\Delta$ is a polytope in a real vector space $V$ (such as $\R\oplus M_\R$), and $\alpha$ is a nonzero vector  in the dual space $V^*$. Let $\langle\ ,\ \rangle: V\x V^*\to\R$ denote the dual pairing. We define 
\[
\lowestface_\alpha(\Delta) := \{w\in \Delta \mid \langle w,\alpha\rangle \le \langle v, \alpha\rangle\text{ for all } v\in\Delta \}.
\] 
We let
\begin{equation}\label{e:shortnotation}
\langle\Delta,\alpha\rangle:=\min\{\langle v,\alpha\rangle\mid v\in\Delta\}.
\end{equation}
Then we have, equivalently, $ \lowestface_\alpha(\Delta)=\{w\in \Delta\mid \langle w,\alpha\rangle=\langle\Delta, \alpha\rangle\}$.
\end{defn}

\begin{remark}\label{r:lowestface} If the polytope $\Delta$ in the above definition is not full-dimensional, then it can happen that $\lowestface_\alpha(\Delta)$ equals to all of $\Delta$. In this case we include $\Delta$ itself among the faces of $\Delta$. If $\Delta$ is full-dimensional however, then we use the word `face' to mean `proper face'. 
\end{remark}

\begin{remark}\label{r:transversalF}
Suppose $\Delta=\AugNewton(W)$ and $\pr:\AugNewton(W)\to \Newton(W)$, as in Definition~\ref{d:aug}. We note that if $0$ lies in the interior of $\Newton(W)$, i.e. if $W$ is complete, then every face of $\Delta$ intersects $\pr\inv(0)$ in at most one point. 
\end{remark}

\begin{proof}[Proof of Lemma~\ref{l:max}]
We use the notations as above. In particular note that as explained in Remark~\ref{r:AugNewton},
\begin{equation}\label{e:TropAug}
\Trop(W)(d)=\min_{w\in\AugNewton(W)} \langle w, (1,d)\rangle=\langle \AugNewton(W),(1,d)\rangle.
\end{equation}
For any $d\in N_\R$ we therefore have
\begin{equation}\label{e:troplesstau}
\Trop(W)(d)
%=\langle \AugNewton(W),(1,d)\rangle
%le  \langle (c_j,v_j),(1,d)\rangle
\le \langle (\tau,0),(1,d)\rangle=\tau,
\end{equation}
simply because $(\tau,0)$ lies  in the polytope. 
Thus we know that $\Trop(W)(d)\le \tau$ for any $d$. 

On the other hand $d\in N_\R$ can be chosen in such a way that $(\tau,0)$ itself lies on the lowest face $\lowestface_{(1,d)}(\AugNewton(W))$
 of $(1,d)$. In that case by Definition~\ref{d:lowestface},
\[
\langle (\tau,0),(1,d)\rangle= \langle \AugNewton(W),(1,d)\rangle.
%\rangle\le\min\{ c_i + \langle v_{i},d\rangle\mid i=1,\dotsc n \}.
%= \Trop(W)(d).
\]
The left hand side above equals to $\tau$, and  the right hand side equals to $\Trop(W)(d)$. Thus we see that $\tau=\Trop(W)(d)$ and the value $\tau$ is attained. Hence $\tau$ is the maximal value of $\Trop(W)$.
\end{proof}

Following on from the proof of Lemma~\ref{l:max}, we can characterise for which $d\in N_\R$ the function $\Trop(W)$ attains its maximal value $\tau$.

\begin{lemma}\label{l:tauattainment} Suppose $W$ is a positive, complete  Laurent polynomial and $(\tau,0)$ is the lowest point above $0$ in $\AugNewton(W)$. Then for $d\in N_\R$ we have
\begin{equation*}\label{e:tauattainment}
 \Trop(W)(d)=\tau \iff (\tau,0)\in\lowestface_{(1,d)}(\AugNewton(W)).
 \end{equation*}
\end{lemma}

\begin{proof}
By Lemma~\ref{l:max}, $\tau$ is the maximal value of $\Trop(W)$. We saw in the proof of Lemma~\ref{l:max} that if $\lowestface_{(1,d)}(\AugNewton(W))$ contains $(\tau, 0)$, then this maximal value $\tau$ of $\Trop(W)$ is attained at $d$. Thus the implication $\impliedby$ is already proved. For the other direction suppose that 
$\Trop(W)(d)=\tau$. Then, since
$
\Trop(W)(d)=\langle\AugNewton(W),(1,d)\rangle,
$
see \eqref{e:TropAug}, we have that
\[
\langle\AugNewton(W),(1,d)\rangle=\tau=\langle(\tau,0),(1,d)\rangle.
\]
This implies that $(\tau,0)\in\lowestface_{(1,d)}(\AugNewton(W))$, see Definition~\ref{d:lowestface}. Thus we have proved the lemma.
\end{proof}

\begin{remark}\label{r:facetcase} 
Let $F$ denote the {\it minimal-dimensional} face of $\AugNewton(W)$ containing the point $(\tau,0)$. In the case that $F$ has codimension $1$ in $\R\oplus M_\R$, it turns out that there is a {\it unique} element $d\in N_\R$ for which $\Trop(W)(d)=\tau$. Namely, if $\Trop(W)(d)=\tau$ then Lemma~\ref{l:tauattainment} implies $F\subseteq \lowestface_{(1,d)}(\AugNewton(W))$. For dimension reasons this implies that $F=\lowestface_{(1,d)}(\AugNewton(W))$. But then $\lowestface_{(1,d)}(\AugNewton(W))$ has codimension $1$ in $\R\oplus M_\R$, and therefore it determines the vector $(1,d)$ in $\R\oplus N_\R$ (which must be perpendicular to its `lowest face') uniquely. 
\end{remark}

\subsection{The tropical critical conditions}

We now introduce a set of conditions which we will show must be satisfied by the valuation of any positive critical point of~$W$.

\begin{defn}[The functions $\delta_i$] \label{d:deltai} For $i=1,\dotsc, n$ we consider the piecewise linear functions $\delta_i:N_\R\to \R_{\ge 0}$ associated to the summands of $W=\sum_i\gamma_i x^{v_i}$ given by
\begin{equation}\label{e:deltai}
\delta_i(d):=c_i+\langle v_i,d\rangle-\Trop(W)(d),
\end{equation} 
where $c_i=\ValK(\gamma_i)$. 
Sometimes $d$ will be fixed and we may write $\delta_i$ for $\delta_i(d)$ in this context.
\end{defn}

\begin{remark} Suppose $x\in T(\KK_{>0})$ has $\Val(x)=d$. It will be useful at times to group the summands of $W(x)$ according to the valuations, $\ValK(\gamma_i x^{v_i})=c_i+\langle v_i,d\rangle$.  
Note that the minimal valuation achieved by any $ \gamma_i x^{v_i}$ is $\Trop(W)(d)$, and we have
\begin{equation}
\ValK(\gamma_i x^{v_i})=\Trop(W)(d)+\delta_i(d).
\end{equation}  
Thus grouping summands of $W(x)$ by their valuations is equivalent to grouping them according to the value of $\delta_i(d)$.  
\end{remark}

\begin{defn}[Tropical critical conditions for $W$] \label{d:tropcritcond}
We say that $d\in N_\R$ satisfies the {\it tropical critical conditions for $W$}  if
\begin{equation}\label{e:tropcritcond}
\ConvexHull^\circ(\{v_i\mid\delta_i(d)=\varepsilon\})\cap\Span(\{v_i\mid\delta_i(d)<\varepsilon\})\ne\emptyset, 
\end{equation}
 for all $\varepsilon\ge 0$. Here $\ConvexHull^\circ$ stands for the relative interior of the convex hull, by which we mean the set of linear combinations $\sum_{i=1}^n r_i v_i$ where $\sum_{i=1}^n r_i=1$ and all $r_i$ are {\it strictly} positive. We set $\ConvexHull^\circ(\emptyset)=\{0\}$.  Also $\Span(\emptyset)=\{0\}$.
\end{defn}

\begin{remark}\label{r:epsilon0}
Observe that the condition \eqref{e:tropcritcond} is automatically satisfied for $d$ if $\varepsilon$ does not equal to any of the values $\delta_i(d)$, since then the convex hull is just $\{0\}$ and automatically lies in the span. The condition \eqref{e:tropcritcond} is also automatically satisfied whenever $\Span(\{v_i\mid\delta_i(d)<\varepsilon\})=M_\R$, as happens for large enough $\varepsilon$ if $\Newton(W)$ is full-dimensional.

If $\varepsilon=0$, on the other hand, then we always have a nontrivial convex hull in \eqref{e:tropcritcond}, since for any $d$ there exist some $i$ such that $\delta_i(d)=0$. In this case the span in  \eqref{e:tropcritcond} is automatically  $\{0\}$, and the associated tropical critical condition says that 
\[0\in\ConvexHull^\circ(\{v_i\mid\delta_i(d)=0\}).\] 

More generally, if $\{v_i\mid\delta_i(d)=\varepsilon\}$ is nonempty and $\varepsilon >0$, then the tropical critical condition for $\varepsilon$ says that
\begin{equation}\label{e:genep}
0\in \ConvexHull^\circ(\{\bar v_i\mid\delta_i(d)=\varepsilon\}),
\end{equation}
where  
 $\bar v_i$ is the image of $v_i$ under the quotient map
\[
 \Span(\{v_i\mid\delta_i(d)\le\varepsilon\})\to
 \Span(\{v_i\mid\delta_i(d)\le\varepsilon\})/\Span(\{v_i\mid\delta_i(d)<\varepsilon\}), 
 \]
 and the convex hull in \eqref{e:genep} is in the quotient space above. 
\end{remark}

\begin{lemma}\label{l:btropcrit} Let $W$ be a positive Laurent polynomial. If $p$ is a {\it positive} critical point of $W$,  
then $b=\Val(p)$ satisfies the tropical critical conditions for $W$.
\end{lemma}

\begin{defn}[The {\it gradient function} of $W$]\label{d:gradient}
Recall that $W=\sum_{i=1}^n \gamma_i x^{v_i}$. Associated to any $u\in N$ consider the associated $T$-invariant vector field $\partial_u$ acting by $\partial_u(x^v)=\langle v,u\rangle x^v$. We have
\begin{equation}\label{e:genWder}
\partial_u W=\sum_{i=1}^n \gamma_i \langle v_i,u\rangle  x^{v_i}.
\end{equation}
Let $M_\C=M_\R\otimes\C$ and $M_\KK=M_\R\otimes\KK$. The derivatives of $W$ all together are encoded in the $M_\KK$-valued function  
\begin{equation}
\label{e:Gofx}G(x)=\sum_{i=1}^n \gamma_i   x^{v_i} \, v_i,  
\end{equation}
for which $\langle G(x),u\rangle =\partial_u W (x)$. We call $G$ the gradient function of $W$.
\end{defn}

\begin{proof}[Proof of Lemma~\ref{l:btropcrit}]
First observe that $x$ is a critical point of $W$ if and only if  $G(x)=0$, where $G$ is the gradient function from Definition~\ref{d:gradient}.
Let us assume $x\in T(\KK_{>0})$ with valuation $\Val(x)=d\in N_\R$. 
Note that we have
\[
\ValK(\gamma_i x^{v_i})=c_i+\langle v_i,d\rangle,
\]
and recall that
\begin{equation}
\label{e:Gofxexpansion}
\Trop(W)(d)=\min(\{c_i+\langle v_i,d\rangle\mid i=1,\dotsc n\}),
\end{equation}
as in Example~\ref{ex:W}. 
Therefore we can expand $G(x)\in M_\KK$ in terms of $t$ giving
\[G(x)=t^{\Trop(W)(d)}\sum_{\varepsilon\ge 0} g_\varepsilon(x) t^{\varepsilon},   
\]
for vector-valued coefficients $g_{\varepsilon}(x)\in  M_\C$. The point $x$ is a critical point of $W$ if and only if $g_{\varepsilon}(x)=0$  for all $\varepsilon\ge 0$.

Now suppose $p$ is a critical point of $W$ with $\Val(p)=b$, and we have fixed $\varepsilon\ge 0$. Expanding $ \gamma_i   p^{v_i}$ with regard to $t$ we get an element of $\KK$ of the form
\begin{equation}\label{e:lambdai}
 \gamma_i   p^{v_i}=t^{c_i+\langle v_i,b\rangle}\lambda_i=t^{c_i+\langle v_i,b\rangle}(\sum_{\delta\ge 0}\lambda_{i,\delta} t^{\delta}),
 \end{equation} 
 and it follows that
\begin{equation}\label{e:Glambdai}
G(p)=\sum_{i=1}^n t^{c_i+\langle v_i,b\rangle}(\sum_{\delta\ge 0}\lambda_{i,\delta} t^{\delta}) v_i.
 \end{equation} 
 Then $\lambda_{i,\delta}$ contributes a summand $\lambda_{i,\delta} v_i$ to $g_{\varepsilon}(p)$ precisely if $\varepsilon + \Trop(W)(b)=\delta+ c_i+\langle v_i,b\rangle $. Let  $\delta_i=\delta_i(b)$, compare Definition~\ref{d:deltai}. Then the condition for $\lambda_{i,\delta}$ to contribute to $g_{\varepsilon}(p)$ becomes $\delta=\varepsilon-\delta_i$, and it implies $\delta_i\le\varepsilon$.
 Therefore we have
\begin{equation}\label{e:critPtepsilon}
%\CritEq(\varepsilon)
g_{\varepsilon}(p)=\sum_{\{i\mid  \delta_i\le\varepsilon\}} \lambda_{i,\varepsilon -\delta_i} v_i=0.
\end{equation}
To summarise, if $b=\ValK(p)$ is the valuation of a critical point $p$ of $W$, and the $\lambda_i\in\KK$  are defined from $p$ and $W$ as above, then the equation \eqref{e:critPtepsilon} must be satisfied for every $\varepsilon\ge 0$, and vice versa. 
 
Rewriting the right hand equality in \eqref{e:critPtepsilon} we get
\begin{equation}\label{e:keyeq}
\sum_{\{i\mid \delta_i=\varepsilon\}} \lambda_{i,0} v_i =- \sum_{\{i\mid  \delta_i<\varepsilon\}} \lambda_{i,\varepsilon-\delta_i} v_i.
\end{equation}
We now use that $W$ is positive, and $p$ is a {\it positive} critical point. Thus by positivity of $p^{v_i}$ and $\gamma_i$ in \eqref{e:lambdai} we have that $\lambda_i\in \KK_{>0}$, and therefore $\lambda_{i,0}>0$ for every $i$. Dividing both sides of \eqref{e:keyeq} by $\sum_{\{i\mid  \delta_i=\varepsilon\}}\lambda_{i,0}$ we get a point 
\[
v:=\frac{\sum_{\{i\mid \delta_i=\varepsilon\}} \lambda_{i,0} v_i }{\sum_{\{i\mid \delta_i= \varepsilon\}}\lambda_{i,0}}=-\frac{ \sum_{\{i\mid \delta_i<\varepsilon\}} \lambda_{i,\varepsilon-\delta_i} v_i}{\sum_{\{i\mid \delta_i=\varepsilon\}}\lambda_{i,0}}.
\]
By the first description of $v$ we see that $v$  lies in $\ConvexHull^\circ(\{v_i\mid\delta_i(b)=\varepsilon\})$ inside $M_\R$. The second description tells us that $v\in \Span_\C(\{v_i\mid\delta_i(b)<\varepsilon\})$. In fact since $v\in M_\R$, the imaginary part of $v$ must be zero and we have $v\in \Span_\R(\{v_i\mid\delta_i(b)<\varepsilon\})$. Therefore the intersection 
\[
\ConvexHull^\circ(\{v_i\mid\delta_i(b)=\varepsilon\})\cap\Span_\R(\{v_i\mid\delta_i(b)<\varepsilon\})
\]
is nonempty. This is true for any $\varepsilon\ge 0$, thus $b=\Val(p)$ satisfies the tropical critical conditions for $W$ as claimed.
\end{proof}

The next lemma interprets the $\varepsilon=0$ tropical critical conditions.
 
\begin{lemma}\label{l:tropcritmax} Let $W$ be a positive, complete  Laurent polynomial, and let  $\tau$ be the minimal height above $0$ of $\AugNewton(W)$. The following two conditions on  $d\in N_\R$ are equivalent. 
\begin{enumerate}
\item  $d$ satisfies the $\varepsilon=0$ tropical critical conditions for $W$ from Definition~\ref{d:tropcritcond}.
\item  $\lowestface_{(1,d)}(\AugNewton(W))$ is the minimal face $F$ of $\AugNewton(W)$ containing $(\tau,0)$. In particular $\Trop(W)(d)=\tau$ by Lemma~\ref{l:tauattainment}.
\end{enumerate}
\end{lemma}

\begin{proof}
Recall that $d\in N_\R$ satisfies the $\varepsilon=0$ tropical critical conditions for $W$  if
 \begin{equation}\label{e:ep0tropcrit}
 0\in\ConvexHull^\circ(\{v_i\mid \delta_i(d)=0
 %\Trop(W)(d)=\langle (c_i,v_i),(1,d)\rangle
  \}),
 \end{equation}
 see Remark~\ref{r:epsilon0}. We prove first (1) $\implies$ (2).
 Let $0=\sum_{\{i\mid \delta_i=0\}} r_i v_i$ be an expression of $0\in M_\R$ as strict convex combination, so $r_i\in\R_{>0}$ and $\sum r_i=1$. Then we define $(c,0):= \sum_{\{i\mid \delta_i=0\}} r_i (c_i,v_i)$.  Note that $(c,0)$ clearly lies in $\AugNewton(W)$. We begin by proving that $(c,0)=(\tau,0)$ and lies in $\lowestface_{(1,d)}(\AugNewton(W))$.
 
 Note that  $\delta_i(d)=0$ implies that $\Trop(W)(d)=\langle (c_i,v_i),(1,d)\rangle$. Thus we have
\[\sum_{\{i\mid \delta_i=0\}} r_i \langle(c_i,v_i),(1,d)\rangle=\sum_{\{i\mid \delta_i=0\}}r_i\Trop(W)(d)=\Trop(W)(d).\]  

On the other hand
\[\sum_{\{i\mid \delta_i=0\}} r_i \langle(c_i,v_i),(1,d)\rangle=  
\langle(c,0),(1,d)\rangle=c.\]

Thus $c$ is a value of $\Trop(W)$; namely $c=\Trop(W)(d)$. By Lemma~\ref{l:max} we know that the only point of the form $(c,0)$ in $\AugNewton(W)$ for which $c$ is a value of $\Trop(W)$ is the `lowest' point $(\tau, 0)$. Thus $c=\tau$ and, as also shown in Lemma~\ref{l:max}, this is the maximal value attained by $\Trop(W)$.

So far we have shown that $\Trop(W)(d)=\tau$. By Lemma~\ref{l:tauattainment}  we now see that $(\tau,0)\in\lowestface_{(1,d)}(\AugNewton(W))$. Let us show that $(\tau,0)$ lies in the relative interior of this face. Then the minimality of the face will follow and the proof of (2) will be complete.
 We use the following Claim. 
\vskip .2cm
\begin{paragraph}{\it Claim:} If  $\Trop(W)(d)=\tau$ then we have that
\begin{equation}
\label{e:lowestfacetau}
\lowestface_{(1,d)}(\AugNewton(W))=\ConvexHull(\{(c_i,v_i)\mid \delta_i(d)=0\}).
\end{equation}
\end{paragraph}

\begin{paragraph}{\it Proof of Claim:}
Since  $\Trop(W)(d)=\tau$, we have that
\[
\lowestface_{(1,d)}(\AugNewton(W))=\{(c,v)\in\AugNewton(W)\mid \langle (c,v),(1,d)\rangle=\tau\},
\]
by Definition~\ref{d:lowestface}. 
As a face of $\AugNewton(W)$ this convex set can also be expressed as a convex hull by
\begin{equation*}
\lowestface_{(1,d)}(\AugNewton(W))=\ConvexHull(\{(c_i,v_i)\mid \langle (c_i,v_i),(1,d)\rangle=\tau\}).
\end{equation*}
On the other hand
\[
\langle (c_i,v_i),(1,d)\rangle =\tau \iff  \delta_i(d)=0,
\]
since $\tau=\Trop(W)(d)$. Therefore \eqref{e:lowestfacetau} holds and 
this proves the Claim. 
\end{paragraph}
\vskip .2cm
Now the $\varepsilon=0$ tropical critical condition \eqref{e:ep0tropcrit} implies,  that
\[
(\tau,0)\in\ConvexHull^\circ(\{(c_i,v_i)\mid \delta_i(d)=0\}).
\]
Thus the $\varepsilon=0$ tropical critical condition can be interpreted as saying that $(\tau,0)$ lies in the relative interior of $\lowestface_{(1,d)}(\AugNewton(W))$, by  \eqref{e:lowestfacetau}. 

We now prove (2)$\implies$(1). Assume we have that the point $(\tau,0)$ lies in the interior of $\lowestface_{(1,d)}(\AugNewton(W))$. Then $\Trop(W)(d)=\tau$ and applying the Claim from above, we have that $\lowestface_{(1,d)}(\AugNewton(W))$ is described by \eqref{e:lowestfacetau}. Therefore we have that
\[
(\tau,0)\in\ConvexHull^\circ(\{(c_i,v_i)\mid \delta_i(d)=0\}).
\]
This implies the tropical critical condition \eqref{e:ep0tropcrit} by projection to $M_\R$. 
\end{proof}

\begin{remark} %We take another look at the $\varepsilon=0$ tropical critical condition for $W$.  
Note that Lemma~\ref{l:tropcritmax} implies the existence of a $d\in N_\R$ satisfying the $\varepsilon=0$ tropical critical condition. 
\end{remark}
 
\begin{lemma}\label{l:codim1case}
Let $W$ be a positive, complete   Laurent polynomial and  let  $\tau$ be the minimal height above $0$ of $\AugNewton(W)$.  Let  $F$ be the minimal face of $\AugNewton(W)$ containing the point $(\tau,0)$. If $F$ has codimension $1$ in $\R\oplus M_\R$, then there exists a unique point $d\in N_\R$ satisfying the tropical critical conditions.   
\end{lemma}

\begin{proof}
If $d$ satisfies the tropical critical conditions, then in particular it satisfies the tropical critical condition for $\varepsilon=0$. Therefore by Lemma~\ref{l:tropcritmax} we see that $\Trop(W)(d)=\tau$. This together with the codimension $1$ condition implies  that $d$ is uniquely determined, see Remark~\ref{r:facetcase}. Moreover, $d$, if it exists,  is the unique element of $N_\R$ for which 
\begin{equation}\label{e:defofd}
\lowestface_{(1,d)}(\AugNewton(W))=F.
\end{equation}

Now suppose $d$ is the element defined by \eqref{e:defofd}. It automatically satisfies the tropical critical condition for $\varepsilon=0$, by Lemma~\ref{l:tropcritmax}. Let us assume that $\varepsilon>0$. 
Note that by \eqref{e:defofd} combined with \eqref{e:lowestfacetau} from the proof of Lemma~\ref{l:tropcritmax},  
\[
F=\ConvexHull(\{(c_i,v_i)\mid \delta_i(d)=0\}).
\]
Therefore
\[
\pr(F)=\ConvexHull(\{v_i\mid \delta_i(d)=0\}).
\]
Then clearly, since $\varepsilon>0$, we have that 
\[
\pr(F)\subset \Span(\{v_i\mid\delta_i(d)<\varepsilon\}).
\]
By our assumptions, $F$ has codimension $1$ in $\R\oplus M_\R$, and we have that $F$ is transversal to $\pr\inv(0)$, see Remark~\ref{r:transversalF}. Therefore the projection, $\pr(F)$, is full-dimensional in $M_\R$, that is, there is no proper linear subspace of $M_\R$ which contains $\pr(F)$. It follows that
\[
\Span(\{v_i\mid\delta_i(d)<\varepsilon\})=M_\R.
\]  
This implies that the tropical critical condition \eqref{e:tropcritcond} is satisfied for $\varepsilon>0$.
\end{proof}

We now prove uniqueness in general.

\begin{lemma} [Uniqueness]\label{l:uniqueness} Suppose $W$ is a positive, complete  Laurent polynomial over $\KK$ as above. There is at most one element $d\in N_\R$ satisfying the  tropical critical conditions for $W$ from Definition~\ref{d:tropcritcond}. 
\end{lemma}

\begin{proof}

Suppose $d$ and $d'$ in $N_\R$ both satisfy the tropical critical conditions from Definition~\ref{d:tropcritcond}. Let us write $\delta_i$ for $\delta_i(d)$ and $\delta_i'$ for $\delta_i(d')$, compare Definition~\ref{d:deltai}. By Lemma~\ref{l:tropcritmax}, $\Trop(W)(d)=\Trop(W)(d')$. Therefore we have that  
\begin{equation}\label{e:deltadifference}
\delta'_i-\delta_i=\langle v_i,d'\rangle -\langle v_i, d\rangle=\langle v_i, d'-d\rangle.
\end{equation}

\paragraph{{\it Claim:}} If $d$ and $d'$ both satisfy the tropical critical conditions for $W$ then $\delta_i=\delta'_i$ for all $i$.
\vskip .2cm 

 Note that this claim implies the lemma. Namely because of \eqref{e:deltadifference} the claim implies that $\langle v_i, d'-d\rangle=0$ for all $i$. On the other hand since $\Newton(W)$ %, that is, the convex hull of the $v_i$, 
 is full-dimensional, we have that the $v_i$ span all of $M_\R$. Therefore it follows that $d'=d$.   
\vskip .2cm

\paragraph{{\it Proof of the Claim:}} 
Recall that the tropical critical conditions for $d$ say that
\[
\ConvexHull^\circ(\{v_i\mid 
\delta_i=\varepsilon\})\cap\Span(\{v_i\mid \delta_i<\varepsilon\rangle \})\ne\emptyset,
\]
and these conditions are non-trivial only if $\varepsilon\in \{\delta_i\mid i=1,\dotsc, n\}$. Note that the conditions for $d'$ are the same but with $\delta_i$ replaced by $\delta'_i$ everywhere. 
We will prove the claim by induction ``on $\varepsilon$'' as follows.

\vskip .2cm
\paragraph{{\bf Induction hypothesis:}} $\delta'_i=\delta_i$ whenever $\delta_i<\varepsilon$ or $\delta'_i<\varepsilon$. 
\vskip .2cm 
\paragraph{{\bf Induction step:}} Assuming the induction hypothesis, $\delta'_i=\delta_i$ whenever $\delta_i\le \varepsilon$ or $\delta'_i\le \varepsilon$. 
\vskip .2cm

This should be thought of as an induction on $\varepsilon$'s lying in the finite ordered set $\{\min(\delta_i,\delta_i')\mid i=1,\dotsc, n \}$ with natural ordering inherited from $\R_{\ge 0}$. 

\vskip .2cm 
For the start of the induction we observe that the induction hypothesis is automatically satisfied if $\varepsilon=0$. Thus we can move straight to the induction step. For this we only need to show that $\delta'_i=\delta_i$ whenever $\delta_i= \varepsilon$ or $\delta'_i=\varepsilon$. 
%For each $\varepsilon\ge 0$ let $D_\varepsilon=\{i\mid \delta_i(d)=\varepsilon\}$, and similarly $D'_{\varepsilon}=\{i\mid \delta_i(d')=\varepsilon\}$. Furthermore let $E_\varepsilon=\{i\mid \delta_i(d)<\varepsilon\}$ and  $E'_{\varepsilon}=\{i\mid \delta_i(d')<\varepsilon\}$. If $\varepsilon=0$ then clearly $E_0=E'_0=\emptyset$. We also saw already that $D_0=D'_0=\{i\mid  (c_i,v_i)\in F \}$. We now prove by induction of $\varepsilon$ that $D_{\varepsilon}=D'_{\varepsilon}$ and $E_{\varepsilon}=E'_{\varepsilon}$ in general. 
%By induction hypothesis, for every $\rho<\varepsilon$ we have that $D_\rho=D'_\rho$.

By the tropical critical conditions for $d$ we have the existence of a $v\in M_\R$ such that 
\[
v=\sum_{\{ i\mid \delta_i=\varepsilon \}} r_iv_i=\sum_{\{i\mid \delta_i<\varepsilon\}} \nu_iv_i,
\]
where $r_i\in \R_{>0}$ with $\sum r_i=1$, and $\nu_i\in \R$. As a consequence 
\[
\langle v, d'-d\rangle=\sum_{\{ i\mid \delta_i=\varepsilon \}} r_i\langle v_i,d'-d\rangle=\sum_{\{i\mid \delta_i<\varepsilon\}} \nu_i\langle v_i,d'-d\rangle,
\]
and thus by \eqref{e:deltadifference}, 
\[
\sum_{\{ i\mid \delta_i=\varepsilon \}} r_i(\delta'_i-\delta_i)=\sum_{\{i\mid \delta_i<\varepsilon\}} \nu_i(\delta'_i-\delta_i).
\]
By the induction hypothesis the right-hand side vanishes.
Thus also 
\[\sum_{\{ i\mid \delta_i=\varepsilon \}} r_i(\delta'_i-\delta_i)=0.
\]
However the induction hypothesis implies that if $\delta_i=\varepsilon$, then $
\delta'_i\ge\varepsilon$. %Since if $\delta'_i$ were smaller than $\varepsilon$ it would have to equal $\delta_i$, which contradicts $\delta_i=\varepsilon$. 
Thus each summand  
$r_i(\delta'_i-\delta_i)\ge 0$ and therefore the summands must individually vanish. Since $r_i$ is non-zero, it follows that $\delta'_i=\delta_i$ if $\delta_i=\varepsilon$. 

If we reverse the roles of $d$ and $d'$ above, using instead the tropical critical conditions for $d'$, then the same argument will also prove  that $\delta_i'=\delta_i$ if $\delta'_i=\varepsilon$. Thus the induction step is complete. 
\end{proof}

\subsection{Construction of a canonical point in $N_\R$ associated to $W$}\label{s:critconstr}
In this subsection we prove existence of a point satisfying the tropical critical conditions associated to a positive, complete  Laurent polynomial $W$.
The proof involves an inductive construction and therefore we introduce some notation for our general set-up.

\begin{defn}[Newton datum]\label{d:goodNewton} Suppose that we have a short exact sequence of real vector spaces
\begin{equation}\label{e:goodLaurent}
0\to\R\overset\eta\to V\overset\beta\to U\to 0,
\end{equation}
together with a finite  subset   $\mathcal W=\{w_i\}$ of $V$. We call  the tuple $\Xi=(U,V,\eta,\beta,\mathcal W)$ a {\it Newton datum}, and refer to $\NDelta:=\ConvexHull(\beta(\mathcal W))$ as the {\it Newton polytope} of $\Xi$. We call  $\Xi=(U,V,\eta,\beta,\mathcal W)$ a {\it complete  Newton datum} if $\NDelta$ is a full-dimensional polytope that contains $0$ in its interior. We refer to $ \AugDelta:=\ConvexHull(\mathcal W)$ as the {\it full polytope} of $\Xi$. If we have a splitting of the exact sequence \eqref{e:goodLaurent} then we say that the Newton datum $\Xi$ is {\it split}. 
 \end{defn}

\begin{example}[The split complete  Newton datum $\Xi_W$ of $W$]\label{ex:XiW}
Let $W:T\to\KK$ be a positive, complete  Laurent polynomial as in Definition~\ref{d:good}, and recall Definition~\ref{d:aug}. Then we get a complete  Newton datum by setting $V=\R\oplus M_\R$, $U=M_\R$, $\beta=\pr$, choosing the map $\eta:\R\to \R\oplus M_\R$ to be $c\mapsto (c,0)$, and setting $\mathcal W=\{(c_i,v_i)\mid i=1,\dotsc, n\}$.  In this case $\AugDelta=%\ConvexHull(\mathcal W)
\AugNewton(W)$ and $\NDelta=\Newton(W)$.  This complete  Newton datum comes with a splitting $\tilde\beta:U\to V$ given by $v\mapsto(0,v)$. We denote the resulting split complete  Newton datum associated to $W$  by $\Xi_W$. 
%We note that $\Xi_W$ in fact depends only on the piecewise linear map $\Trop(W)$, and may also be referred to as the complete  Newton datum associated to $\Trop(W)$.
\end{example}

Our initial goal is to construct, given a general complete  Newton datum $\Xi=(U,V,\eta,\beta,\mathcal W)$, a particular point in $V^*$. 
We think of this point as being canonically associated to $\Xi$. In the case where $\Xi$ is split, the associated canonical point in $V^*$ also gives rise to a point in $U^*$. 

The construction of the canonical point will in general be a recursive one and involve constructing out of $\Xi=:\Xi_1$ a sequence of $\Xi_i$. We use the following auxiliary definition.

\begin{defn}\label{d:minimalheight}
Suppose $\Xi=(U,V,\eta,\beta,\mathcal W)$ is a complete  Newton datum. We associate to $\Xi$ the set
\begin{equation}\label{e:Wbar}
\overline{\mathcal W}:=\{w\in\mathcal W\mid \beta(w)\ne 0\}
\end{equation}
and call it the {\it reduced set} of $\Xi$. If $\dim(U)>0$, then $\overline{\mathcal W}\ne \emptyset$, by the completeness assumption. We refer to the convex hull $\overline\AugDelta=\ConvexHull(\overline{\mathcal W})$ as the {\it reduced polytope} of $\Xi$. If $\overline{\mathcal W}=\emptyset$ we set $\overline\AugDelta=\emptyset$.
We associate real numbers $\tau, \otau$  to $\Xi$ using the full polytope $\AugDelta$ and the reduced polytope $\overline\AugDelta$, by setting
\begin{align*}
\tau:&=\min(\{c\in\R\mid \eta(c)\in\AugDelta\})\\
\otau:&=\min(\{c\in\R\mid \eta(c)\in\overline\AugDelta\}).
\end{align*}
Here we let $\min(\emptyset):=\infty$ for the case that $\overline\AugDelta=\emptyset$. We refer to $\tau$ and $\otau$ as the {\it minimal height} above $0$ of $\AugDelta$ and $\overline\AugDelta$, respectively. Note that since $0$ is assumed to lie in the interior of Newton polytope $\NDelta$ of $\Xi$, we have that $\NDelta=\beta(\AugDelta)=\beta(\overline{\AugDelta})$.

Suppose $\overline\AugDelta\ne 0$. We note that by the construction of the reduced polytope %$\overline\AugDelta$ 
we have that the lowest point $\eta(\bar\tau)$ above $0$ is never a vertex of $\overline\AugDelta$. We will denote the minimal face of $\overline\AugDelta$ containing $\eta(\bar\tau)$ by $F$. We also let $E$ be the linear span of the translate $F-\eta(\bar\tau)$. Thus $E$ is a positive-dimensional subspace of $V$ and is the minimal subspace for which $F\subset \eta(\bar\tau)+E$.
\end{defn}

\begin{remark} In the setting of $\Xi_W$, see Example~\ref{ex:XiW}, since $\AugDelta=\AugNewton(W)$, the minimal height above $0$ of $\AugDelta$ recovers the minimal height $\tau$ associated to $W$ in the earlier Definition~\ref{d:minimalheightAugNewton}. Also in this setting, the reduced polytope $\overline{\AugDelta}$ agrees with the polytope $\AugNewton(W-W_{\operatorname{const}})$ where $W_{\operatorname{const}}\in \KK$ is the constant term of $W$. Note that both $W$ and $W-W_{\operatorname{const}}$ have the same Newton polytope by the completeness assumption. They also have the same set of critical points.
\end{remark}

\vskip .2cm

\begin{recursionstep} \label{r:recursion} \rm Assume $\Xi=(U,V,\eta,\beta,\mathcal W)$ is a complete  Newton datum  such that $\dim(U)>0$. Consider the associated reduced polytope $\overline \AugDelta$ with minimal height $\otau$. We construct a new datum $\Xi'=(U',V',\eta',\beta',\mathcal W')$ with reduced polytope $\overline{\AugDelta'}$ having minimal height $\otau'$, along with  connecting maps from $\Xi$ to $\Xi'$, as follows.
As in Definition~\ref{d:minimalheight}, let $F$ be the minimal-dimensional face of $\overline \AugDelta$ containing $\eta(\otau)$ and
let $E$ be the span  of the translation of the face $F$ through $0$.
Now $E$ is a linear subspace of $V$ and $\dim(E)>0$. We define
\begin{itemize}
\item $V':=V/E$ and $U':=U/\beta(E)$,
\item associated  projections 
\[
 \sigma:V\to V'=V/E ,\qquad \pi:U\to U'=U/\beta(E),
\]
\item
maps $\eta':\R\to V'$ and $\beta':V'\to U'$ induced by the maps $\eta$ and $\beta$,
\item a set  $\mathcal W':=\sigma(\overline{\mathcal W})$ in $V'$ for $\overline{\mathcal W}$ as in \eqref{e:Wbar}. 
\end{itemize}
Let us denote the tuple $(U',V',\eta',\beta',\mathcal W')$ constructed above by $\Xi'$. We also set ${\AugDelta'}:=\ConvexHull(\mathcal W')$. We also set $\overline{\mathcal W'}:=\{w\in \mathcal W'\mid\beta'(w)\ne 0\}$. If $\overline{\mathcal W'}\ne \emptyset$ we define $\overline{\AugDelta'}=\ConvexHull(\overline{\mathcal W'})$. Otherwise we set $\overline{\AugDelta'}:=\emptyset$. We now prove that $\Xi'$ is again a complete  Newton datum.
\end{recursionstep}

\begin{lemma}\label{l:goodLaurent} Let $\Xi=(U,V,\eta,\beta,\mathcal W)$  be a complete  Newton datum with $\dim(U)>0$ and all notations as above. Let $\Xi'=(U',V',\eta',\beta',\mathcal W')$ be the tuple constructed from $\Xi$ in the recursion step. Then the following properties hold. 
\begin{enumerate}
\item $\Xi'$ is a complete  Newton datum, and we have $\dim(U')<\dim(U)$.
\item The minimal height $\otau'$ of  the reduced polytope $\overline{\AugDelta'}$ of $\Xi'$ is related to the minimal height $\otau$ of $\overline{\AugDelta}$ from $\Xi$ by the inequality $\otau<\otau'$.
\item We have the inclusion of convex sets, $\sigma(\overline\AugDelta)\subset \eta'( \otau+\R_{\ge 0})+\overline{\AugDelta'}$.
\end{enumerate}
\end{lemma}

\begin{proof} 
We first prove (1). We have commutative diagram where the top row is known to be an exact sequence,
\[
\begin{array} {ccccccccc}0 &\To & \R &\overset{\eta} \To & V & \overset{\beta}\To & U &\To &0\\
 & &id\downarrow\quad &  & \sigma\downarrow\quad &  & \pi\downarrow \quad& &\\
0 &\To & \R &\overset{\eta'} \To & V' &\overset{\beta'} \To & U' &\To &0.
\end{array}
\]
We need to show that the bottom row is also exact. Clearly $\beta$ and $\pi$ are both surjective, hence the commutativity of the second square implies that $\beta'$ is also surjective. The commutativity of the diagram also implies that the image of $\eta'$ lies in the kernel of $\beta'$.

Now by the recursion step we have $V'=V/E$. To show that $\eta'$ is injective and $\Image(\eta')=\ker(\beta')$ we need to show precisely that $\ker(\beta)\cap E=\{0\}$. Or, if we translate by $\eta(\otau)$, then it suffices to prove that $\ker(\beta)\cap F=\{\eta(\otau)\}$, for the minimal-dimensional face $F$ of $\overline\AugDelta$ containing $\eta(\otau)$.

Let $I$ be the intersection of $\overline\AugDelta$ with $\ker(\beta)$.  Then either $I=\{\eta(\otau)\}$, in which case we are done, or $I$ is an interval in the line $\ker(\beta)=\Image(\eta)$, and $\eta(\otau)$ is an endpoint of the interval $I$. 
In the latter case suppose there exists a face $G$ of $\overline \AugDelta$ which contains the interval $I$.  We may suppose that $G$ is minimal with this property. Then $\eta(\otau)$ is on the boundary of the face $G$. Therefore there is a proper face of $G$ which contains $\eta(\otau)$, and this face does not contain the interval $I$ (by minimality of $G$). This face must contain $F$, since $F$ was the minimal face of $\overline \AugDelta$ that contained $\eta(\otau)$.  Thus $F$ also does not contain the interval $I$, intersecting it only in $\eta(\otau)$. It follows that $F\cap\ker(\beta)=\{\eta(\otau)\}$, which concludes the proof of exactness of the bottom row.   

Now let $\AugDelta'$ denote the convex hull of $\mathcal W'$. It remains to observe that the projection $\beta'(\AugDelta')$ of the polytope $\AugDelta'$ is full-dimensional with zero in the interior, which follows from the fact that $\beta'(\AugDelta')=\pi(\beta(\overline\AugDelta))$ and $\beta(\overline\AugDelta)=\beta(\AugDelta)$ is full-dimensional with zero in the interior.

We now prove (2). Note that  $\AugDelta'=\sigma(\overline\AugDelta)$. This is the full polytope of the complete  Newton datum $\Xi'$.   Consider the face $F$ of $\overline\AugDelta$ and its image under $\sigma$. We show that this image is a vertex of $\AugDelta'$. 

Namely suppose $H$ is an affine hyperplane in $V$ intersecting $\overline\AugDelta$ in $F$. Then
\[
H\cap\overline\AugDelta\ =\ F\subset\  \eta(\otau)+E.
\]
Therefore the projection $H'=\sigma(H)$ is an affine hyperplane in $V'$ which intersects $\overline{ \AugDelta'}$ precisely in a point. Indeed, this point is the projection the face $F$ of $\overline{\AugDelta}$, and it also equals $\sigma(\eta(\otau))=\eta'(\otau)$. The point $\eta'(\otau)$ is therefore a vertex of $\AugDelta'$, as the intersection of $\AugDelta'$ with the hyperplane $H'$. 

Since $\AugDelta'=\ConvexHull(\mathcal W')$, the vertex $\eta'(\otau)$ is necessarily a point of $\mathcal W'$. As it is also an element in the kernel of $\beta'$, it is removed in the construction of $\overline{\mathcal W'}$. Thus the reduced polytope $\overline{\AugDelta'}=\ConvexHull(\overline{\mathcal W'})$ does not contain $\eta'(\otau)$, and therefore its lowest point above $0$, namely $\eta(\otau')$, necessarily satisfies $\otau'>\otau$. This proves part (2) of the lemma.

We now prove (3). First assume that $\overline{\AugDelta'}\ne\emptyset$. For the left-hand side of the inclusion, $\sigma(\overline\AugDelta)=\AugDelta'$. This equals the convex hull of $\mathcal W'$. The set $\mathcal W'$ is the union of the set $\overline{\mathcal W'}$ and a set of points which can be written in the form $\eta(c)$, where $c\ge \otau$, by part (2). Thus (3) follows.   
\end{proof}

\begin{remark}\label{r:transversalE}
Note that in the proof of Lemma~\ref{l:goodLaurent} we showed that the subspace $E$ from Definition~\ref{d:minimalheight} is transversal to the line $\tau(\R)$. Indeed this property is equivalent to the injectivity of the map $\eta':\R\to V/E$.
\end{remark}

\begin{defn}[Canonical point of a complete  Newton datum]\label{d:specialpoint}
Suppose that $\Xi=(U,V,\eta,\beta,\mathcal W)$ is a complete  Newton datum. We define a canonical point $\widetilde\crit(\Xi)\in V^*$ associated to $\Xi$ as follows. 

Define $\Xi_1=(U_1,V_1,\eta_1,\beta_1,\mathcal W_1):=\Xi$ and construct a finite sequence of complete  Newton data $\Xi_k=(U_k,V_k,\eta_k,\beta_k,\mathcal W_k)$ by applying the recursion step whenever $\dim (U_k)>0$, setting  $\Xi_{k+1}:=(\Xi_k)'$. In particular we have $V_{k+1}=V_k/E_k$ and we call the projection map $\sigma_k:V_k\to V_{k+1}$. Let $m$ be the first index for which $\dim (U_m)=0$. Thus $V_m$ is $1$-dimensional and $\Xi_m$ is the final complete  Newton datum in the sequence. We identify $V_m^*$ with $\R$ via the isomorphism 
\[
\eta_m^*:V_m^*\to \R,
\]
and we have a sequence of injections
\[
\R=V_m^*\overset{\sigma_{m-1}^*}\To V^*_{m-1} \overset{\sigma_{m-2}^*}\To V^*_{m-2}\To \cdots \To V_2^*\overset{\sigma_{1}^*}\To V^*_{1}=V^*.
\] 
We set $\widetilde\crit(\Xi)\in V^*$ be the image of $1$ under the above composition of maps. 
\end{defn}

\begin{defn}[Canonical point of a split complete  Newton datum]\label{d:splitspecialpoint}
Suppose that $\Xi=(U,V,\eta,\beta,\mathcal W)$ is a split complete  Newton datum with  splitting $\tilde\beta:U\to V$. We define a point in $U^*$ by setting
\[
\crit(\Xi)=\crit(\Xi,\tilde\beta):=\tilde\beta^*(\widetilde\crit(\Xi)).
\]
\end{defn} 

\begin{remark} We note that $\widetilde\crit(\Xi)\in V^*$  is nonzero by its construction. The canonical point $\crit(\Xi,\tilde\beta)\in U^*$ of a split complete  Newton datum on the other hand may be equal to $0$.

If $\Xi=\Xi_W$ is the split complete  Newton datum from Example~\ref{ex:XiW} which is associated to a complete, positive $W:T\to\KK$, then $\crit(\Xi_W)\in N_\R$ since $U=M_\R$ and  $N_\R=M_\R^*$. Moreover $\crit(\Xi_W)\in N_\R$ can be interpreted as a tropical point of $T$ via the identification $N_\R\hat = \Trop(T)$ from Definition~\ref{d:TropN}. 
\end{remark}

\vskip .2cm 

The key property of the canonical point associated to $\Xi_W$ is that it satisfies the tropical critical conditions for $W$, as we will prove next. 
 Note that the property of satisfying the tropical critical conditions for $W$ in fact only depends on the complete  Newton datum $\Xi_W$.

\begin{prop}\label{p:specialpointtropcrit} Let $W$ be a positive, complete  Laurent polynomial, and let $\Xi_W$ be the split complete  Newton datum associated to $W$  in Example~\ref{ex:XiW}. Then the canonical point $\crit(\Xi_W)\in N_\R$ from Definition~\ref{d:splitspecialpoint} satisfies the tropical critical  conditions for~$W$, see Definition~\ref{d:tropcritcond}.  
\end{prop}

We first give a more direct description of $\crit(\Xi)$ which will be useful in the proof of the proposition. Namely we describe the point $\widetilde\crit(\Xi)$ and $\crit(\Xi)$ 
concretely as follows. 

\begin{lemma}\label{l:1a} Let $\Xi=(U,V,\eta,\beta,\mathcal W)$ be a complete  Newton datum with full polytope~$\AugDelta$ and let $\widetilde{\crit}(\Xi)$  be its associated canonical point in $V^*$. Recall also the maps $\sigma_i$ from Definition~\ref{d:specialpoint}. We define a subspace $\widetilde E$ in $V$ by setting $\widetilde E:=\ker(\sigma_{m-1}\circ\dotsc\circ\sigma_1)$. 
\begin{enumerate}\item
The subspace $\widetilde E$ is the kernel of $\widetilde\crit(\Xi)$, and $\widetilde\crit(\Xi)$   maps to~$1$ under the map $\eta^*:V^*\to \R$. Moreover these two properties uniquely characterise the canonical point  $ \widetilde\crit(\Xi)$.
\item
 We have that $\AugDelta\subset \eta( \tau+\R_{\ge 0})+\widetilde E$, where $\tau$ is the minimal height above $0$ of $\AugDelta$ from Definition~\ref{d:minimalheight}.
 \end{enumerate}
\end{lemma}
 \begin{remark}\label{r:splitcrit}
If additionally the complete  Newton datum $\Xi$ is split, then we may identify  $V^*$ with $\R\oplus U^*$ such that $\eta^*$ is the projection onto the first coordinate. In this setting the above lemma implies that the canonical point $\crit(\Xi)$ from Definition~\ref{d:splitspecialpoint} is the unique element $a\in U^*$ such that $(1,a)$ vanishes on $\widetilde E$.
\end{remark}
\begin{proof}
Let ${\boldsymbol \sigma}_{[m]}:V\to V_m$ be the composition of  projections, ${\boldsymbol \sigma}_{[m]}=\sigma_{m-1}\circ\dotsc\circ\sigma_1$, so that $\widetilde E$ the kernel of ${\boldsymbol \sigma}_{[m]}$. Since $V_m\cong V/\widetilde E$ is $1$-dimensional, see Definition~\ref{d:specialpoint}, the linear subspace $\widetilde E$ is a hyperplane in $V$. Moreover, by construction, $\widetilde E$ is transversal to the line $\eta(\R)$, compare Remark~\ref{r:transversalE}.

By definition we have that $ \widetilde\crit(\Xi)\in V^*$ is equal to the image of $1$ under the composition ${\boldsymbol \sigma}_{[m]}^*\circ(\eta_m^*)\inv$, that is under the diagonal map in the commutative diagram, 
\begin{equation*}
\begin{tikzcd}
V_m^* = (V/\widetilde E)^*\arrow{r}{{\boldsymbol \sigma}_{[m]}^*} \ar[d,"\sim" labl,"\eta_m^*" ] &V^* \arrow{d}{\eta^*}\\
\R \arrow{r}{=}\arrow{ur}&\R.
\end{tikzcd}
\end{equation*}
Thus we see that $ \widetilde\crit(\Xi)\in V^*$ is the unique point in $V^*$ which vanishes on $\widetilde E$ and maps to $1$ under $\eta^*$. Thus (1) is proved.

Let us now prove (2). By definition of $\tau$ and $\overline{\AugDelta}$ we have that
\begin{equation}\label{e:inclusionconvex}
\AugDelta\subset \eta( \tau+\R_{\ge 0})+\overline{\AugDelta}
\end{equation}
Recall $\Xi_k=(U_k,V_k,\eta_k,\beta_k,\mathcal W_k)$, the $k$-th complete  Newton datum from Definition~\ref{d:specialpoint}, and $\overline{\AugDelta_{k}}$ its reduced polytope with minimal height above $0$ denoted $\otau_k$. In particular, $\overline{\AugDelta_1}=\overline{\AugDelta}$ and the  minimal height above zero $\otau_1$ satisfies $\otau_1\ge\tau$. We prove (2) by applying the maps $\sigma_1, \sigma_2,\dotsc, \sigma_{ m-1} $ successively to both sides of 
\eqref{e:inclusionconvex}. Note first that for any $k$ we have
\begin{equation}\label{e:inclusionconvexk}
\sigma_k(\overline\AugDelta_k)\subset \eta_{k+1}( \otau_k+\R_{\ge 0})+\overline{\AugDelta_{k+1}}\subset  \eta_{k+1}( \tau+\R_{\ge 0})+\overline{\AugDelta_{k+1}},
\end{equation}
where the first inclusion follows from Lemma~\ref{l:goodLaurent}(3), and the second from the inequality $\bar\tau_1\ge \tau$ together with Lemma~\ref{l:goodLaurent}(2).
This gives that
\begin{equation}\label{e:inclusionconvex2}
{\boldsymbol \sigma}_{[m]}(\AugDelta)=\sigma_{m-1}\circ\dotsc\circ\sigma_1(\AugDelta)\subset \eta_{m}( \tau+\R_{\ge 0})+\overline{\AugDelta_{m}}.
\end{equation}
Now we have $\overline{\AugDelta_{m}}=\emptyset$, compare Recursion Step~\ref{r:recursion}. Therefore we obtain that
\[
{\boldsymbol \sigma}_{[m]}(\AugDelta)\subset \eta_{m}( \tau+\R_{\ge 0}).
\]
Taking the inverse image ${\boldsymbol \sigma}_{[m]}\inv$, we see that $\AugDelta\subset \eta( \tau+\R_{\ge 0}) +\widetilde E$.
\end{proof}

\begin{remark}\label{r:inclusionconvexj} 
Using the notation of the above proof, we also claim that for every index $j$ the intermediate inclusion,
\begin{equation}\label{e:intermediateinclusions}
\overline{\AugDelta_j}\subset \eta_{j}( \otau_j+\R_{\ge 0})+{\boldsymbol \sigma}_{[j]}(\widetilde E),
\end{equation}
holds. In particular, for $j=1$ we have $\overline{\AugDelta}\subset \eta( \otau+\R_{\ge 0}) +\widetilde E$.  

Namely, if we take as our starting point instead of  \eqref{e:inclusionconvex} the analogous inclusion 
\[\sigma_j(\overline{\AugDelta_j})={\AugDelta_{j+1}}\subset \eta_{j+1}( \otau_{j}+\R_{\ge 0})+\overline{\AugDelta_{j+1}},
\] 
and apply $\sigma_{m-1}\circ\dotsc\circ\sigma_{j+1}$ to both sides (using again \eqref{e:inclusionconvexk}), we obtain
\[\sigma_{m-1}\circ\dotsc\circ\sigma_{j}(\overline{\AugDelta_j})\subset \eta_m( \otau_j+\R_{\ge 0}).
\]
The inverse image $ (\sigma_{m-1}\circ\dotsc\circ\sigma_{j})\inv(\eta_m( \otau_j+\R_{\ge 0}))$ of the right hand side is just 
\[
 \eta_j( \otau_j+\R_{\ge 0})+\ker(\sigma_{m-1}\circ\dotsc\circ\sigma_{j})= \eta_j( \otau_j+\R_{\ge 0})+{\boldsymbol \sigma}_{[j]}(\widetilde E).
 \]
  Thus  \eqref{e:intermediateinclusions} follows.
\end{remark}

\begin{lemma}\label{l:critmax}  Let $\Xi_W$ be the split complete  Newton datum associated to a positive, complete  Laurent polynomial $W$ as in Example~\ref{ex:XiW}, and let $\crit=\crit(\Xi_W)\in N_\R$ be its canonical point. Then $\Trop(W)(\crit)$ is the maximum of the piecewise linear function $\Trop(W)$.
\end{lemma}

\begin{proof}
By Lemma~\ref{l:max} the maximal value of $\Trop(W)$ is $\tau$, the minimal height above $0$ of $\AugNewton(W)$. We denote $\AugNewton(W)$ also by $\AugDelta$. We now consider the hyperplane $\widetilde E\subset V=\R\oplus M_\R$ defined in Lemma~\ref{l:1a}, and recall that $(1,\crit)\in V^*$ is characterised by the property of vanishing on $\widetilde E$.

By Lemma~\ref{l:1a}(2) we have that $
\AugDelta\subset \eta(\tau+\R_{\ge 0})+\widetilde E$. Now pairing with $(1,\crit)$ gives
\[
\langle (1,\crit),w\rangle \ge \tau,\quad {\text{ for all $w\in\AugDelta$.}}
\]
This inequality becomes an equality if we choose $w=(\tau,0)$. 
In other words $(1,\crit)$ applied to elements of $\AugDelta$ attains its minimal value at $(\tau,0)$, and therefore we have that
\[
(\tau,0)\in\lowestface_{(1,\crit)}(\AugNewton(W)).
\]
By Lemma~\ref{l:tauattainment} this inclusion implies that $\Trop(W)(\crit)=\tau$.
\end{proof}

\

\begin{proof}[Proof of Proposition~\ref{p:specialpointtropcrit}] 

 We use the notations from earlier in this section. In particular we have the split complete  Newton datum $\Xi_W=\Xi_1=(V_1,U_1,\eta_1,\beta_1,\mathcal W_1)$, with $U_1=M_\R$, $V_1=\R\oplus M_\R$ and $\mathcal W_1=\{(c_i,v_i)\mid i=1,\dotsc n\}$. Let $\AugDelta=\AugDelta_1=\AugNewton(W)$, that is, the convex hull of $\mathcal W_1$. 
 
Recall  the sequence of complete  Newton data, $\Xi_j=(V_j,U_j,\eta_j,\beta_j,\mathcal W_j)$, where $j=1,\dotsc, m$, which we construct out of $\Xi_1$ using the Recursion Step~\ref{r:recursion}. We have associated to each $\Xi_j$ a reduced set, $\overline {\mathcal W}_j$,  with convex hull the reduced polytope, $\overline \AugDelta_j$, and the minimal height above $0$ of $\overline{\AugDelta}_j$ is denoted $\otau_j$, compare Definition~\ref{d:minimalheight}. Note that the real numbers $\otau_j$ are strictly increasing, by Lemma~\ref{l:goodLaurent}.(2). 
We also denote by $F_j$ the minimal face of $\overline{\AugDelta}_j$ containing the point $\eta_j(\otau_j)$. 
Recall that $V_{j+1}=V_{j}/E_{j}$, where $E_{j}$ is the span of the translated face $F_{j}-\eta_{j}(\otau_{j})$.

The proof of the proposition will rely on understanding the faces $F_j$ and for each $F_j$ writing its special point $\eta_j(\otau_j)$ as a convex combination of vertices. 
We now recall the key property of the canonical point $\crit(\Xi_W)$.

Recall the definition of $\widetilde E=\ker(\sigma_{m-1}\circ\dotsc\circ \sigma_1)$ from Lemma~\ref{l:1a}. We have that $\widetilde E$ is a hyperplane in $V_1=\R\oplus M_\R$, which is transversal to $\R\oplus\{0\}$ and whose projection to any $V_j$ contains the subspace $E_j$. 
Consider  also the projection $\boldsymbol \sigma_{[j]}=\sigma_{j-1}\circ\dotsc\circ\sigma_1: V_1\to V_j$, and denote by $\widetilde E_{j-1}$ the kernel of this projection map. Thus we have $\boldsymbol \sigma_{[j]}(\widetilde E_{j})=E_j$ and we have a flag of subspaces of $V_1$,  
\begin{equation}\label{e:flag}
\{0\}\subset \tilde E_1\subset\cdots \subset \tilde E_{m-1}\subset\tilde E_m=\tilde E,
\end{equation}
for which we may identify $V_j$ with $V_1/\tilde E_{j-1}$. Under this identification the subspace $E_j$ of $V_j$ is identified with $\widetilde E_j/\widetilde E_{j-1}$.

Let $a= \crit(\Xi_W)$.  Then by Lemma~\ref{l:1a}(1), $a$ is the unique element of $N_\R$ such that that $(1,a)\in V_1^*$ vanishes on the hyperplane $\widetilde E$,  i.e.
\begin{equation}\label{e:orthogonality}
\langle e, (1,a)\rangle= 0, \quad \text{ for all $e\in\widetilde E$.}
\end{equation}

The condition \eqref{e:orthogonality}  is equivalent to the statement that  pairing with $(1,a)$ gives rise to 
 a well-defined linear map on each $V_j=V_1/\widetilde E_{j-1}$, denoted
\[
h_{a,j}=\langle\ \underline{\ }\ , (1,a)\rangle: V_j\to \R.
\]
Moreover, we note that $h_{a,j}$ vanishes on the subspace $E_j$.

We now use the canonical point $a$ to give a description of the face $F_j$  as an intersection of $\overline\AugDelta_j$ with a hyperplane.   

\vskip .2cm 

\paragraph{\it Claim 1:} Let $a=\crit(\Xi_W)$ as above. For every $j$ the minimal height $\otau_j$ above $0$ of $\overline\AugDelta_j$  is equal to the minimal value on $\overline\AugDelta_j$ attained by the linear map $h_{a,j}$. In particular, $\overline\AugDelta_j$ lies in the half-space $\{h_{a,j}\ge \otau_j\}$,
\begin{equation}\label{e:halfspace}
\overline\AugDelta_j\subset \{h_{a,j}\ge \otau_j\}.
\end{equation}

Moreover, the intersection with the boundary, $ \{h_{a,j}=\otau_j\}$, of the  half-space recovers the minimal face $F_j$ containing $\eta_j(\otau_j)$,
\begin{equation}\label{e:hyperplaneF}
\overline\AugDelta_j\cap \{h_{a,j}=\otau_j\}=F_j.
\end{equation} 
\vskip .2cm

\paragraph{\it Proof of Claim 1:} Recall that $F_j\subset \eta_j(\otau_j) +E_j$, indeed by minimality of $E_j$ we have that
\begin{equation}\label{e:FjEj}
F_j= 
\overline\AugDelta_j\cap(\eta_j(\otau_j) +E_j).
\end{equation}
Since $h_{a,j}$ vanishes on $E_{j}$, we see that for any point $f\in F_j$,
\[
h_{a,j}(f)=h_{a,j}(\eta_j(\otau_j))=h_{a,j}(\boldsymbol{\sigma}_{[j]}(\otau_j,0))=\langle(1,a),(\otau_j,0)\rangle=\otau_j.
\]
Therefore $F_j$ lies in the hyperplane $\{h_{a,j}=\otau_j\}$ in $V_j$. It follows from Remark~\ref{r:inclusionconvexj}, by applying $h_{a,j}$ to both sides of the inclusion \eqref{e:intermediateinclusions}, that the polytope $\overline\AugDelta_j$ must lie in the upper half-space $\{h_{a,j}\ge \otau_j\}$. Thus we have shown \eqref{e:halfspace}, and the inclusion $``\supseteq"$ of \eqref{e:hyperplaneF}. 

Consider now the intersection 
\begin{equation*}\label{e:hyperplaneFagain}
\widetilde F_j:=\overline\AugDelta_j\cap \{h_{a,j}=\otau_j\},
\end{equation*} 
 which contains $F_j$, and because of \eqref{e:halfspace} must be a face of $\overline\AugDelta_j$. If it is not equal to $F_j$ then let $\boldsymbol{\sigma}_{[j]}(c_i,v_i)$ be a vertex of the face $\widetilde F_j$ not in $F_j$.  We have that 
 \[
 \boldsymbol{\sigma}_{[j]}(c_i,v_i)\notin \eta_j(\otau_j)+ E_j,
 \]
 because of \eqref{e:FjEj}. Moreover  
  \[
 \boldsymbol{\sigma}_{[j]}(c_i,v_i)\notin \eta_j(c)+ E_j,
 \]
 for any other $c\in \R$, since $h_{a,j}( \boldsymbol{\sigma}_{[j]}(c_i,v_i))=\otau_j\ne c$.
 Therefore  this vertex of $\widetilde F_j$ is not in  $\eta_j(\R)+ E_j $, which says precisely that its image in $V_{j+1}$ does not lie in the  kernel of $\beta_{j+1}$. So $\boldsymbol{\sigma}_{[j+1]}(c_i,v_i)$ is an element of the reduced set $\overline{\mathcal W}_{j+1}$, and in particular lies in the reduced polytope, $\overline\AugDelta_{j+1}$. Thus we have constructed an element of $\overline\AugDelta_{j+1}$ on which $h_{a,j+1}$ takes the value $\otau_{j}$. 
 
 However by \eqref{e:halfspace} applied to $V_{j+1}$ we know that  $\overline\AugDelta_{j+1}$ lies in the half-space $ \{h_{a,j}\ge \otau_{j+1}\}$. This contradicts the strict inequality $\otau_{j+1}>\otau_j$ which we have by Lemma~\ref{l:goodLaurent}.(2). Thus the claim is proved.

 \vskip .2cm

For $a=\crit(\Xi_W)$, let $\delta_i=\langle (c_i,v_i),(1,a)\rangle-\Trop(W)(a)$ for every $(c_i,v_i)\in \mathcal W_1$. Note that for $\tau$ as in Lemma~\ref{l:critmax} and $h_a:=h_{a,1}$, i.e. the linear map on $V_1$ given by  $(1,a)$, 
 we have 
 \[
 \delta_i=h_{a}(c_i,v_i)-\tau.
 \]
Let us set
\begin{equation} \label{e:grading}
\mathcal F_\varepsilon=\{(c_i,v_i)\in \mathcal W_1\mid  h_a(c_i,v_i)=\tau+\varepsilon\}=\{(c_i,v_i)\in \mathcal W_1\mid \delta_i=\varepsilon\}.
\end{equation}
We can now describe the face $F_j$ as a convex hull. 

 \vskip .2cm
\paragraph{\it Claim~2:} 
Let  $\varepsilon_j:=\otau_j-\tau$. Then 
\begin{equation}\label{e:FjasCH}
F_j=\ConvexHull(\boldsymbol{\sigma}_{[j]}(\mathcal F_{\varepsilon_j})).
\end{equation}

\vskip .2cm
 
 \paragraph{\it Proof of Claim 2:~} 
 We first prove the inclusion $``\subseteq"$. Since $F_j$ is a face of $\overline\AugDelta_j$, its vertices are certain elements $w_i=\boldsymbol\sigma_{[j]}(c_i,v_i)\in \overline {\mathcal W}_j$. By \eqref{e:hyperplaneF}  these lie on the hyperplane $\{h_{a,j}=\otau_j\}$, that is $h_{a,j}(w_i)=c_i+\langle v_i,a\rangle=\otau_j$.  This implies that $w_i\in\boldsymbol\sigma_{[j]}(\mathcal F_{\varepsilon_j})$. Thus $F_j$ lies inside the convex hull from \eqref{e:FjasCH}. 
 
For the other inclusion we need to show that  $ \boldsymbol\sigma_{[j]}(\mathcal F_{\varepsilon_j})$ lies in $F_j$. Note that
 \[
 \boldsymbol\sigma_{[j]}(\mathcal F_{\varepsilon_j})=\left( \boldsymbol\sigma_{[j]}(\mathcal F_{\varepsilon_j})\cap\overline {\mathcal W}_j\right) \cup \left(\boldsymbol\sigma_{[j]}(\mathcal F_{\varepsilon_j})\cap{\beta}_j\inv(0)\right),
 \] 
by the construction of $\overline{\mathcal W}_j$. In this union the left hand set consists of all those $w_i\in \overline{\mathcal W}_j$ which also lie in the hyperplane $\{h_{a,j}=
\otau_j\}$. Thus it is clearly contained in $ F_j$ by \eqref{e:hyperplaneF}. The right hand set lies in the intersection of ${\beta}_j\inv(0)=\eta_j(\R)$ with the hyperplane $\{h_{a,j}=
\otau_j\}$. But this intersection is a single point, namely $\eta_j(\otau_j)$. Thus it also lies in $F_j$, and Claim~2 follows.

\vskip .2cm

We now claim that $a$ satisfies the tropical critical conditions for $W$. 

\vskip.2cm
\paragraph{\it Claim~3:} Let $a=\crit(\Xi_W)$ and $\delta_i=\langle (c_i,v_i),(1,a)\rangle-\Trop(W)(a)$, as above. Then for any $\varepsilon\ge 0$
\begin{equation}\label{e:intersection}
\ConvexHull^\circ(\{v_i\mid 
\delta_i=\varepsilon\})\cap\Span(\{v_i\mid \delta_i<\varepsilon\rangle \})\ne\emptyset.
\end{equation}
\vskip .2cm
 
 \paragraph{\it Proof of Claim 3:~} 
 Let us first consider the case where $\varepsilon=\varepsilon_j=\otau_j-\tau$ for some fixed $j$. By Claim~2 we have
\begin{equation*}\label{e:FjasCH2}
F_j=\ConvexHull(\boldsymbol{\sigma}_{[j]}(\mathcal F_{\varepsilon_j})).
\end{equation*}
%Note that this in particular implies $\mathcal F_{\varepsilon_j}\ne \emptyset$. 
By the definition of $F_j$ we have that  $\eta_j(\otau_j)\in F_j^\circ$, and thus \eqref{e:FjasCH} implies that $\eta_j(\otau_j)$ can be written in the form
\[
\eta_j(\otau_j)=\sum_{(c_i,v_i)\in\mathcal F_{\varepsilon_j}} r_i\, \boldsymbol{\sigma}_{[j]}((c_i,v_i))=\boldsymbol{\sigma}_{[j]}\left (\sum_{(c_i,v_i)\in\mathcal F_{\varepsilon_j}} r_i\, (c_i,v_i)\right),
\]
for some $r_i\in \R_{>0}$ with $\sum r_i=1$. Let 
\begin{equation*}
v=\sum_{(c_i,v_i)\in\mathcal F_{\varepsilon_j}} r_i\, v_i.
\end{equation*}
Then by \eqref{e:grading},
\begin{equation}\label{e:vinCH}
v\in \ConvexHull^\circ(\{v_i\mid \delta_i=\varepsilon_j\}).
\end{equation}
On the other hand 
%$\boldsymbol\sigma_{[j]}\left(\sum_{(c_i,v_i)\in\mathcal F_{\varepsilon_j}} r_i\, (c_i,v_i)\right)$ lies in the kernel of $\beta_j$ and 
we have
\[
(c,v)=\sum_{(c_i,v_i)\in\mathcal F_{\varepsilon_j}} r_i\, (c_i,v_i)\in (\otau_j,0)+\widetilde E_{j-1}. 
\]
Therefore applying the projection $\beta_1$ we have
\begin{equation}\label{e:vinproj}
v=\beta_1(c,v)\in \beta_1\left((\otau_j,0)+\widetilde E_{j-1}\right)=\beta_1\left(\widetilde E_{j-1}\right).
\end{equation}

Recall that $E_{\ell}$ is the span in $V_{\ell}$ of $F_{\ell}-\eta_{\ell}(\otau_{\ell})$. And by Claim~2, the face $F_{\ell}$ is the convex hull of $\boldsymbol{\sigma}_{[\ell]}(\mathcal F_{\varepsilon_{\ell}})$. Therefore
\[
E_{\ell}=\boldsymbol{\sigma}_{[\ell]} \left(\Span\left(\{(c_i,v_i)-(\otau_{\ell},0)\mid (c_i,v_i)\in \mathcal F_{\varepsilon_{\ell}}\}\right)\right).
\]
In terms of the flag of subspaces \eqref{e:flag} we have $E_\ell=\widetilde E_{\ell}/\widetilde E_{\ell-1}$, and so we see that 
\[
\widetilde E_{\ell}= \Span\left(\{(c_i,v_i)-(\otau_{\ell},0)\mid (c_i,v_i)\in \mathcal F_{\varepsilon_{\ell}}\}\right)+\widetilde E_{\ell-1}.
\]
Applying the above equality recursively we obtain the following key description of $\widetilde E_{j-1}$ as a span,
\begin{equation}\label{e:keyspandescr}
\widetilde E_{j-1}= \Span\left(\{(c_i,v_i)-(\otau_{\ell},0)\mid  \ell\le j-1, \,(c_i,v_i)\in \mathcal F_{\varepsilon_{\ell}} \}\right).
\end{equation}
Applying the projection $\beta_1$ to both sides of \eqref{e:keyspandescr} gives us that
\begin{equation}\label{e:Spaneq}
\beta_1(\widetilde E_{j-1})=%\beta_1\left(\Span\left(\{(c_i,v_i)-(\otau_{\ell},0)\mid  \ell\le j-1, \,(c_i,v_i)\in \mathcal F_{\varepsilon_{\ell}} \}\right)\right)\\=
\Span\left(\{v_i\mid  \ell\le j-1, \,(c_i,v_i)\in \mathcal F_{\varepsilon_{\ell}}\}\right).
\end{equation}
Recall that $\varepsilon_\ell=\otau_\ell-\tau$, and thus by Lemma~\ref{l:goodLaurent}.(2) we have that $\varepsilon_1<\varepsilon_2<\cdots <\varepsilon_j$. Now \eqref{e:Spaneq} together with the definition of $\mathcal F_\varepsilon$ implies that
\begin{equation}\label{e:projEj-1}
\beta_1(\widetilde E_{j-1})\subseteq\Span(\{v_i\mid \delta_i\le\varepsilon_{j-1}\}).
\end{equation}
Combining \eqref{e:projEj-1} with the special property of $v$ from \eqref{e:vinproj} we see that 
\begin{equation}\label{e:vinSpan}
v\in\Span(\{v_i\mid \delta_i<\varepsilon_j\}).
\end{equation}
The two observations about $v$,  \eqref{e:vinCH} and \eqref{e:vinSpan}, imply that the tropical critical condition \eqref{e:intersection} holds for $\varepsilon=\varepsilon_j$, completing the proof for such $\varepsilon$.

Now it remains to consider the case where $\varepsilon$ is not one of the $\varepsilon_j$.  In this case there exists a $j$ such that  $\varepsilon_{j-1}< \varepsilon <\varepsilon_{j}$. 
Since \eqref{e:intersection} holds trivially if $\mathcal F_{\varepsilon} =\emptyset$, compare Remark~\ref{r:epsilon0}, we can assume $\varepsilon$ is such that $\mathcal F_{\varepsilon}$ is nonempty. 

Let $(c_i,v_i)$ be an arbitrary element in $\mathcal F_{\varepsilon}$. The projection $\boldsymbol\sigma_{[j]}(c_i,v_i)\in V_j$ satisfies
\[
h_{a,j}(\boldsymbol\sigma_{[j]}(c_i,v_i))=\tau+\varepsilon<\tau+\varepsilon_j=\otau_j.
\]
Therefore by  \eqref{e:halfspace} in Claim~1, $\boldsymbol\sigma_{[j]}(c_i,v_i)\notin\overline\AugDelta_j$. This implies that 
\[
\boldsymbol\sigma_{[j]}(c_i,v_i)\in \ker(\beta_{j}),
\]
by definition of $\overline\AugDelta_{j}$. Thus, since $V_j=V_1/\widetilde E_{j-1}$,
\begin{equation}\label{e:kerneleq}
(c_i,v_i)\in\eta_{1}(\R)+ \widetilde E_{j-1}.
\end{equation}
The equation \eqref{e:kerneleq} holds for all $(c_i,v_i)\in \mathcal F_{\varepsilon}$, and therefore  for any convex combination. We take a strict convex combination of all of the elements of $\mathcal F_{\varepsilon}$ to obtain an element $(c,v)\in \eta_{1}(\R)+ \widetilde E_{j-1}$. Note that then 
\[
v\in\ConvexHull^\circ(\{v_i\mid \delta_i=\varepsilon\})\quad\text{and}\quad v\in\beta_1(\widetilde E_{j-1}).
\] 
Applying \eqref{e:projEj-1} to the observation on the right hand side above, we see that  
\[
v\in \Span(\{v_i\mid \delta_i\le\varepsilon_{j-1}\})\subset  \Span(\{v_i\mid \delta_i<\varepsilon\}).
\]
Therefore \eqref{e:intersection} holds again and  Claim~3 is proved. 
\end{proof}

\begin{cor}\label{c:valcrit} Suppose $W:T\to \KK$ is a positive, complete  Laurent polynomial and $p$ is a  critical point of $W$ lying in $T(\KK_{>0})$. Let $\Xi_W$ be the split complete  Newton datum associated to $W$ in Example~\ref{ex:XiW} and  $\crit(\Xi_W)$ its associated canonical point in $N_\R$. Then $\Val(p)=\crit(\Xi_W)$. 
\end{cor}

\begin{proof}
Given that $p$ is a positive critical point of $W$, then its valuation $\ValK(p)$  satisfies the tropical critical conditions by Lemma~\ref{l:btropcrit}. On the other hand the point $\crit(\Xi_W)$ satisfies the tropical critical conditions  by Proposition~\ref{p:specialpointtropcrit}. By uniqueness, Lemma~\ref{l:uniqueness},
there is {\it at most} one point in $N_\R$ satisfying the tropical critical conditions. It follows that $\Val(p)=\crit(\Xi_W)$. \end{proof}

\subsection{Leading coefficient of a positive critical point of $W$}\label{s:leadingcoeff} Recall that $W=\sum_{i=1}^n \gamma_i x^{v_i}$ denotes a positive, complete  Laurent polynomial over $\KK$, as in \eqref{e:genW}. By Corollary~\ref{c:valcrit} if a critical point $p\in T(\KK_{>0})$ of $W$ exists then its valuation is given by the canonical point  $\crit(\Xi_W)\in N_\R$ which was constructed in Section~\ref{s:critconstr}. We write $\crit$  for $\crit(\Xi_W)$, thinking of $W$ as fixed. 
% (as in Lemma~\ref{l:critmax}). 

Recall from Section~\ref{s:exp} the definition of the leading term and the logarithmic leading coefficient.
We have that the leading term $p_0$ of $p$ takes the form
\[
p_0=e^{\Log(p)}t^{\crit}.
\]
Our goal in this section is to determine $\Log(p)$ and prove that it is unique, that is, depends only on $W$.

\begin{defn}\label{d:Bep} Let $\varepsilon\ge 0$ and let $W=\sum_{i=1}^n \gamma_i x^{v_i}$ be a positive, complete  Laurent polynomial  (with its associated functions $\delta_i$ from Definition~\ref{d:deltai}, and its canonical point $\crit\in N_\R$). 
We define a linear subspace of $M_\R$ by
\[
B_{<\varepsilon}=\Span_\R(\{v_i\mid\delta_i(\crit)< \varepsilon\}). %,\qquad B_{\le\varepsilon}=\Span_\R(\{v_i\mid\delta_i(a)<\varepsilon\}).
\]
We have a set of `relevant $\varepsilon$',
\[
\mathcal E:=\{\varepsilon\in \R_{\ge 0}\mid \delta_i(\crit)=\varepsilon\text { some $i\in 1,\dotsc, n$}\}.
\] 
 We may order this set, so that  
$\mathcal E=\{\varepsilon_0=0,\varepsilon_1,\dotsc, \varepsilon_m\}$
where $\varepsilon_j<\varepsilon_{j+1}$. 
Observe that if  $\varepsilon_{j-1}<\varepsilon\le \varepsilon_j$, then $B_{<\varepsilon}=B_{<\varepsilon_{j}}$. Thus we have a nested sequence of subspaces
\[
\{0\}\subseteq B_{<\varepsilon_1}\subseteq\cdots\subseteq B_{< \varepsilon_m}\subseteq M_\R,
\]
and dually,
\[
N_\R\supseteq B_{<\varepsilon_{1}}^\perp\supseteq\cdots\supseteq B_{< \varepsilon_m}^\perp\supseteq\{0\}.
\]
Note that the subspaces $B_{<\varepsilon}$ defined above depend only on the split complete  Newton datum $\Xi_W$ of $W$.
\end{defn}

\begin{remark} We note that the $\varepsilon_j$ which came up in the proof of Proposition~\ref{p:specialpointtropcrit} and were associated to the special faces $F_j$ lie in the set $\mathcal E$ of relevant $\varepsilon$, but may not exhaust this set. Namely this can happen if there are terms $\gamma_i x^{v_i}$ in $W$ which never give rise to a vertex of one of the $F_j$. Thus there is a small change in notation here.  
\end{remark}

We now introduce a set of conditions which we will show must be satisfied by the logarithmic leading coefficient $\Log(p)$ of any positive critical point $p$ of~$W$.

\begin{defn}[Critical coefficient conditions for $W$] \label{d:coeffcritcond}
Recall that we have  $W=\sum_{i=1}^n \gamma_i x^{v_i}$ with $\crit=\crit(\Xi_W)\in N_\R$ and  the piecewise linear maps $\delta_i:N_\R\to \R_{\ge 0}$. We say that $d\in N_\R$ satisfies the {\it critical coefficient conditions for $W$}  if
\begin{equation}\label{e:coeffcritcond}
\sum_{\{i\mid\delta_i(\crit)=\varepsilon\}}\Coeff(\gamma_i) e^{\langle v_i,d\rangle} v_i\in B_{<\varepsilon} %\Span(\{v_i\mid\delta_i(\crit)<\varepsilon\}), 
\end{equation}
 for all $\varepsilon\ge 0$. 
 \end{defn}
 
 \begin{remark}
The key property of the canonical point $\crit=\crit(\Xi_W)$ is that it satisfies the tropical critical conditions \eqref{e:tropcritcond}. Using above notation, these say that there exists for every $\varepsilon\in\mathcal E$ a convex combination $v=\sum_{\{i\mid \delta_i(\crit)=\varepsilon\}} r_i v_i$ which lies in $B_{<\varepsilon}$. In light of this, the critical coefficient conditions can be interpreted as specifying that such a convex combination is given explicitly by setting the $r_i$ to be
\begin{equation}\label{e:coeffcritcondrem}
r_i:=\frac{\Coeff(\gamma_i) e^{\langle v_i,d\rangle}}{\sum_{\{k\mid\delta_k(\crit)=\varepsilon\}}\Coeff(\gamma_k) e^{\langle v_k,d\rangle}},
\end{equation}
where $i$ is such that $\delta_i(\crit)=\varepsilon$. Note  in particular that the existence of a $d$ satisfying the critical coefficient conditions requires the tropical critical conditions to hold for $\crit$. 
 \end{remark}

\begin{lemma}\label{l:coeffcrit} Let $W$ be a positive, complete  Laurent polynomial. If $p$ is a {\it positive} critical point of $W$,  then the critical coefficient conditions~\eqref{e:coeffcritcond} are satisfied for $d=\Log(p)\in N_\R$. 
 \end{lemma}

\begin{proof}
Recall that $W=\sum_{i=1}^n \gamma_i x^{v_i}$ and $c_i=\ValK(\gamma_i)$.
Consider the summand $ \gamma_i   p^{v_i}\in\KK_{>0}$ of $W(p)$ and expand it in terms of $t$ giving
\begin{equation}\label{e:lambdaiagain}
 \gamma_i   p^{v_i}=t^{c_i+\langle v_i,\crit\rangle}\lambda_i=t^{c_i+\langle v_i,\crit\rangle}(\sum_{\delta\ge 0}\lambda_{i,\delta} t^{\delta}).
 \end{equation} 
 This defines elements $\lambda_i\in \KK_{>0}$  and $\lambda_{i,\delta}\in\C$. We follow the proof of Lemma~\ref{l:btropcrit}, taking the derivatives of $W$ along $T$-invariant vector fields $\partial_u$, and deduce from the fact that $p$ is a critical point of $W$, that  the identity
\begin{equation}\label{e:keyeqagain}
\sum_{\{i\mid \delta_i(\crit)=\varepsilon\}} \lambda_{i,0} v_i =- \sum_{\{i\mid  \delta_i(\crit)<\varepsilon\}} \lambda_{i,\varepsilon-\delta_i(\crit)} v_i,
\end{equation}
holds for all $\varepsilon\ge 0$, compare \eqref{e:keyeq}.  Now observe that
\[
\lambda_{i,0}=\Coeff(\gamma_i p^{v_i})=\Coeff(\gamma_i) \Coeff(p^{v_i})=\Coeff(\gamma_i)e^{\langle v_i,\Log(p)\rangle}.
\]
The left hand side of \eqref{e:keyeqagain} therefore agrees with the left hand side of \eqref{e:coeffcritcond} for $d=\Log(p)$, and thus the equation \eqref{e:keyeqagain} implies the critical coefficient conditions for $\Log(p)$. 
\end{proof}

The goal of this subsection is to prove the following proposition.

\begin{prop}\label{p:coefficient} For $W$ a positive, complete  Laurent polynomial, there exists a unique element $\dcoeff=\dcoeff(W)\in N_\R$ such that the critical coefficient conditions \eqref{e:coeffcritcond} are satisfied for $d=\dcoeff$. 
\end{prop}

\begin{remark}
Note that the canonical point $\crit(\Xi_W)\in N_\R$ depends only on the complete  Newton datum $\Xi_W$ and its construction has a piecewise-linear flavour. The critical coefficient $\dcoeff(W)\in N_\R$ on the other hand depends on $W$ itself and its construction is real analytic. 
\end{remark}

We begin by stating a lemma which is proved by Galkin \cite[Lemma]{Galkin}, related to our Proposition~\ref{p:Galkin}.

\begin{lemma}\cite{Galkin}\label{l:GalkinGen}
Suppose a function $f:\R^r\to \R$ is given by a linear combination with positive coefficients of exponential functions,
\begin{equation}\label{e:expsum}
f(\rho)=\sum_{i=1}^n C_i e^{\langle\nu_i,\rho\rangle},
\end{equation}
where $C_i\in\R_{>0}$ and $\nu_i\in (\R^r)^*$.
If $\ConvexHull(\{\nu_i\mid i=1,\dotsc, n\})\subset (\R^r)^*$ is full-dimensional and contains $0$ in its interior, then $f$ has a unique critical point. \qed
\end{lemma}

The idea of the proof of Galkin's lemma is to observe that $f$ has values in $\R_{>0}$, and show that the special conditions imply that the values tend to positive infinity in any unbounded direction. Thus $f$ has a critical point which is a minimum. Its uniqueness is shown by observing that the Hessian of $f$ is  always positive definite. Note that Proposition~\ref{p:Galkin} about the Laurent polynomial $L$ follows from Lemma~\ref{l:GalkinGen} by writing $L$ in terms of logarithmic coordinates.

\begin{lemma}\label{l:fepd} For $W$ a positive, complete  Laurent polynomial,
$d\in N_\R$ satisfies the critical coefficient conditions for $W$ if and only if the function $f_{\varepsilon,d}:  B_{<\varepsilon}^\perp\to \R$,
\[
f_{\varepsilon,d}: \rho\to \sum_{\{i\mid\delta_i(\crit)=\varepsilon\}} \Coeff(\gamma_i)e^{\langle v_i,d+\rho\rangle}
\]
has a critical point at $0$ for all $\varepsilon\ge 0$. 
\end{lemma}

\begin{proof} 
The function  $f_{\varepsilon,d}$ has a critical point at $0$ if and only if for all $u\in B_{<\varepsilon}^\perp$ the derivative in the direction of $u$ vanishes at $0$. This derivative is given by
\[
\partial_u f_{\varepsilon,d}\,(\rho)=\sum_{\{i\mid \delta_i(a)=\varepsilon\}} \Coeff(\gamma_i) e^{\langle v_i,d+\rho\rangle}\langle u,v_i\rangle=
\langle u, \sum_{\{i\mid \delta_i(\crit)=\varepsilon\}} \Coeff(\gamma_i) e^{\langle v_i,d+\rho\rangle}v_i\rangle.
\]
Clearly, for $v\in M_\R$ we have $\langle v, u\rangle=0$ for all $u\in B_{<\varepsilon}^\perp$ if and only if $v\in B_{<\varepsilon}$. Thus
\[
\partial_u f_{\varepsilon,d}\,(0)=0\iff
 \sum_{\{i\mid \delta_i(\crit)=\varepsilon\}} \Coeff(\gamma_i) e^{\langle v_i, d\rangle}v_i\in B_{<\varepsilon}.
\]
Therefore the condition on $d$ that $\partial_u f_{\varepsilon,d}\,(0)=0$ for all $\varepsilon\ge 0$ and $u\in B_{<\varepsilon}^\perp$ is equivalent to the critical coefficient conditions from Definition~\ref{d:coeffcritcond}.
\end{proof}

\begin{proof}[Proof of Proposition~\ref{p:coefficient}] By Lemma~\ref{l:fepd} we need to show existence and uniqueness of a point $d=\dcoeff$ such that
\[
f_{\varepsilon,d}: \rho\to \sum_{\{i\mid\delta_i(\crit)=\varepsilon\}} \Coeff(\gamma_i)e^{\langle v_i,d+\rho\rangle}
\]
has a critical point at $0\in  B_{<\varepsilon}^\perp$ for all $\varepsilon\ge 0$. 

First we show existence. Consider the span $B_{\le \varepsilon}=\Span(\{ v_j\mid\delta_j(\crit)\le\varepsilon\})$, so that we have
\[
 B^\perp_{\le \varepsilon}\subseteq B^\perp_{<\varepsilon}\subseteq N_\R.
\] 
Note that if $\delta_i(\crit)=\varepsilon$ and $\rho \in B_{\le\varepsilon}^\perp$ then $\langle v_i,\rho\rangle=0$. It follows that we can define a function $\bar f_{\varepsilon,d}$ on $ B^\perp_{<\varepsilon}/ B^\perp_{\le \varepsilon}$ by the commutative diagram
\begin{equation*}
\begin{tikzcd}  B_{<\varepsilon}^\perp
\arrow{r}{\pi} \ar{dr}[swap]{f_{\varepsilon,d}} &B^\perp_{<\varepsilon}/ B^\perp_{\le \varepsilon}\arrow{d}{\bar f_{\varepsilon,d}} \\
&\R .
\end{tikzcd}
\end{equation*}
Indeed, if $\delta_i(\crit)=\varepsilon$ then  pairing with $v_i$ defines an element 
\begin{equation}\label{e:vibar}
\bar v_i\in (B_{<\varepsilon}^\perp/B_{\le\varepsilon}^\perp)^*,
\end{equation}
and the function $\bar f_{\varepsilon,d}$ is a sum of exponential functions $\bar \rho\mapsto e^{\langle \bar v_i, \bar \rho\rangle}$ with positive coefficients. In other words, 
if we choose a basis so that $B_{<\varepsilon}^\perp/ B^\perp_{\le \varepsilon}=\R^m$, then it is of the form \eqref{e:expsum}. Consider the polytope
\begin{equation}\label{e:Deltaep}
\AugDelta_\varepsilon:=\ConvexHull(\{\bar v_i\mid\delta_i(\crit)=\varepsilon\})\subset (B_{<\varepsilon}^\perp/B_{\le\varepsilon}^\perp)^*.
%B_{\le \varepsilon}/B_{<\varepsilon}\subseteq M_\R/{B_{<\varepsilon}}=(B_{<\varepsilon}^\perp)^*.
\end{equation}       

\vskip .2cm
\paragraph{\it Claim:} The polytope $\AugDelta_\varepsilon$ in $(B_{<\varepsilon}^\perp/B_{\le\varepsilon}^\perp)^*$ is full-dimensional and contains $0$ in its interior.
\vskip .2cm

\paragraph{\it Proof of the Claim:} 
Note that for any triple of finite-dimensional vector spaces $U\subseteq V\subseteq W$ we have a natural isomorphism $(U^\perp/V^\perp)^*\cong V/U$. 
Applied to the triple $B_{<\varepsilon}\subseteq B_{\le \varepsilon} \subseteq M_\R$ we obtain that 
\[
(B_{<\varepsilon}^\perp/B_{\le\varepsilon}^\perp)^*\cong B_{\le \varepsilon}/B_{<\varepsilon}.
\]
Moreover it is straightforward that under this isomorphism the vector $\bar v_i$ from \eqref{e:vibar} is identified with the coset $v_i+ B_{<\varepsilon}$ on the right hand side, which we also denote $\bar v_i$. 
It follows from its construction that $B_{\le \varepsilon}/B_{<\varepsilon}$ is spanned by the $\bar v_i$ with $\delta_i(\crit)=\varepsilon$. This implies that the polytope $\AugDelta_\varepsilon$ is full-dimensional.

Now we use that $\crit=\crit(\Xi_W)$ satisfies the tropical critical conditions \eqref{e:tropcritcond},   by Proposition~\ref{p:specialpointtropcrit}. It follows from these conditions that $0$ lies in the 
interior of $\AugDelta_\varepsilon$. See Remark~\ref{r:epsilon0}, and in particular \eqref{e:genep}. 
This  concludes the proof of the claim.
\vskip .2cm

The above claim together with Lemma~\ref{l:GalkinGen} implies that  the functions $\bar f_{\varepsilon,d}$ each have a unique critical point in $B_{<\varepsilon}^\perp/B_{\le\varepsilon}^\perp$. We denote this critical point by $\bar \rho_{\varepsilon, d}$.
Let us furthermore pick an arbitrary coset representative $\rho_{\varepsilon,d}$ of $\bar \rho_{\varepsilon, d}\in B_{<\varepsilon}^\perp/B_{\le\varepsilon}^\perp$. It follows that  $\rho_{\varepsilon,d}\in B_{<\varepsilon}^\perp$ is a  critical point of $f_{\varepsilon,d}$.

 Note that for $r\in B_{<\varepsilon}^\perp $ and any $d\in N_\R$ we have a commutative diagram
\begin{equation}\label{e:CDtransl}
\begin{tikzcd}  B_{<\varepsilon}^\perp
\arrow{r}{\trans_{r}} \ar{d}[swap]{ f_{\varepsilon,d+ r}} &B^\perp_{<\varepsilon}\arrow{dl}{ f_{\varepsilon,d}} \\
\R &
\end{tikzcd}
\end{equation}
where $\trans_{r}:\rho\mapsto\rho+ r$ is the map of  translation by $r$. It follows from this diagram that if $r=\rho_{\varepsilon,d}$, our chosen critical point of $f_{\varepsilon,d}$, then $0$ is a critical point of $f_{\varepsilon,d+\rho_{\varepsilon,d}}$.

Now recall the set  $\mathcal E=\{\varepsilon_0=0,\varepsilon_1,\dotsc, \varepsilon_m\}$ of relevant $\varepsilon$ from Definition~\ref{d:Bep} and the nested sequence of domains 
\[
N_\R\supseteq B_{<\varepsilon_{1}}^\perp\supseteq\cdots\supseteq B_{< \varepsilon_m}^\perp\supseteq\{0\}
\]
for our functions $f_{\varepsilon_j,d}$. We need to construct a $d\in N_\R$ such that  $0$ is a critical point of $f_{\varepsilon_j,d}$ for all $j$.
 We construct a sequence of $d_j\in N_\R$, where $j=-1,\dotsc, m$, by the following recursion. 
 
\begin{itemize}
\item Let $d_{-1}=0$. 
%We set $d_1:=\rho_{\varepsilon_1,0}$. Thus $d_1\in B^\perp_{<\varepsilon_1}$ and $d_1$ is a critical point for $f_{\varepsilon_1,0}$. Moreover, $0$ is a critical point of $f_{\varepsilon_1,d_1}$ by \eqref{e:CDtransl}. 
\item If $d_{j-1}$ has been constructed, then we set $d_{j}:=d_{j-1}+\rho_{\varepsilon_{j}, d_{j-1}}$. 
\end{itemize}

We record the following properties of the elements $d_j$ of $N_\R$.

\begin{enumerate}
\item By the commutative diagram \eqref{e:CDtransl} with $\varepsilon=\varepsilon_{j}$, $r=\rho_{\varepsilon_{j}, d_{j-1}}$ and $d=d_{j-1}$, 
we have that $0$ is a critical point of the left hand side vertical map, $ f_{\varepsilon_{j},d_{j}}$. 
\item For any $j$ we have $d_j-d_{j-1}\in B_{<\varepsilon_j}^\perp$.
\item Applying (2) recursively, it follows that for any $0\le h\le j$, we have 
\[d_{j}\equiv d_{h}\mod B^\perp_{<\varepsilon_{h+1}},\]
and therefore  $f_{\varepsilon_{h}, d_{j}}=f_{\varepsilon_{h}, d_{h}}$.
\item Moreover, if $0\le h\le j$ then $0$ is a critical point of $f_{\varepsilon_{h}, d_{j}}$, by the combination of (1) and (3).
\end{enumerate}

Let us set $\dcoeff:=d_m$. Then  (4) above implies that $0$ is a critical point of $f_{\varepsilon,\dcoeff}$ for all $\varepsilon\ge 0$. Therefore $\dcoeff$ satisfies the critical coefficient conditions. 

The above construction of $\dcoeff$ appears to depend on the choices of representatives $\rho_{\varepsilon_j,d}$ of the critical points $\bar \rho_{\varepsilon_j,d}$ of the $\bar f_{\varepsilon_j,d}$. It remains to prove that $\dcoeff$ is the {\it unique} element of $N_\R$ satisfying the critical coefficient conditions,  in particular that different choices of representatives lead to the same $\dcoeff$.  

To prove uniqueness, let us suppose that $\dcoeff'\in N_\R$ satisfies the critical coefficient conditions. Thus we have that $0$ is a critical point of $f_{\varepsilon_j, \dcoeff'}$ for all $j=0,\dotsc, m$. For convenience we also introduce $\varepsilon_{m+1}>\varepsilon_m$, so that $B_{<\varepsilon_{m+1}}=M_\R$, and we have $B_{<\varepsilon_{m+1}}^\perp=\{0\}$. To prove that $\dcoeff'=\dcoeff$ we show inductively that $\dcoeff'\equiv \dcoeff \mod B^\perp_{<\varepsilon_j}$ for $j=0,\dotsc m+1$, starting with the trivial case, $j=0$.
By our induction hypothesis we assume that we have shown that
\[
\dcoeff'\equiv \dcoeff \mod B_{<\varepsilon_{j-1}}^\perp.
\]
Then we have the commutative diagram 
\begin{equation}\label{e:CDtransl3}
\begin{tikzcd}  B^\perp_{<\varepsilon_{j-1}}
\arrow{r}{\trans_{{\dcoeff'-\dcoeff}}} \ar{d}[swap]{ f_{\varepsilon_{j-1}, \dcoeff'}} &\quad B^\perp_{<\varepsilon_{j-1}}\arrow{dl}{ f_{\varepsilon_{j-1},\dcoeff}} \\
\R &
\end{tikzcd}.
\end{equation}
Since $0$ is a critical point of $f_{\varepsilon_{j-1},{\dcoeff'}}:B^\perp_{<\varepsilon_{j-1}}\to \R$, by assumption, it follows from the diagram that $\dcoeff'-\dcoeff$ is a critical point of the right hand map $f_{\varepsilon_{j-1},\dcoeff}$. On the other hand, by assumption on $\dcoeff$, the right hand map has critical point $0$ as well.  
Moreover, by Lemma~\ref{l:GalkinGen}  the function $\bar f_{\varepsilon_{j-1},\dcoeff}:B^\perp_{<\varepsilon_{j-1}}/B_{<\varepsilon_{j}}^\perp\to \R$ has a unique critical point, $\bar\rho_{\varepsilon_{j-1},\dcoeff}$. Therefore it follows that  
\[
\dcoeff'-\dcoeff\equiv 0\mod B_{<\varepsilon_{j}}^\perp.
\]
In other words, we have shown that $\dcoeff'\equiv \dcoeff \mod B^\perp_{<\varepsilon_j}$ and the induction is complete. Setting $j=m+1$ so that $B^\perp_{<\varepsilon_{m+1}}=\{0\}$ we see that $\dcoeff'=\dcoeff$, which completes the proof of uniqueness. 
\end{proof}

\subsection{Construction of the positive critical point}\label{s:extension}

 Consider the positive, complete  Laurent polynomial $W:T(\KK)\to\KK$, given by $W=\sum \gamma_i x^{v_i}$.
Suppose that $p$ is a positive critical point of $W$. So far we have determined the leading term $p_0$ of $p$  completely. Namely, by Lemmas~\ref{l:btropcrit} and \ref{l:uniqueness} and  Proposition~\ref{p:specialpointtropcrit} regarding $\crit=\crit(\Xi_W)\in N_\R$ and Lemma~\ref{l:coeffcrit} and Proposition~\ref{p:coefficient} regarding $\dcoeff=\dcoeff(W)\in N_\R$ it follows that $p_0=e^{\dcoeff}t^{\crit}$.
Therefore $p$ must be of the form $e^{\dcoeff}t^{\crit}\exp_T(w)$ for some $w\in N_\mathfrak m$. 

In this section we construct a $\wcrit\in N_\mathfrak m$ such that $\pcrit:=e^{\dcoeff}t^{\crit}\exp_T(\wcrit)$ is a critical point of $W$, and we prove that this condition determines $\wcrit$ uniquely. Thus our constructed $\pcrit$  is the only critical point of $W$ in $T(\KK_{>0})$. %Finally we prove that the Hessian lies in $\KK_{>0}$, which concludes the proof of Theorem~\ref{t:main}. 

In order to describe the critical points of $W$ we associate to $W$ the
 function $G:T(\KK)\to M_{\KK}$ defined by
\begin{equation}\label{e:G}
G(x):=\sum_{i=1}^n \gamma_i x^{v_i}v_i.
\end{equation}
The function $G$ encodes all of the derivatives $\partial_u W$ of $W$, where $u\in N_\R$, via $(\partial_uW)(x)=\langle G(x),u\rangle$. We recall that this function $G$ was first introduced in \eqref{e:Gofx}. Clearly $p$ is a critical point of $W$ if and only if $G(p)=0$. From now on let us write $\delta_i$ for the value $\delta_i(\crit)$ of the piecewise linear function from Definition~\ref{d:deltai}. We define the following partial versions of the function $G:T(\KK)\to M_{\KK}$.

\begin{defn}[$G_{ \varepsilon},G_{\le \varepsilon}$ and $G_{< \varepsilon}$% and $G^{(j)}$
]
Let $\varepsilon\ge 0$. Define maps $G_{\varepsilon}, G_{\le\varepsilon}$ and $G_{< \varepsilon}
:T(\KK)\to M_\KK$ by
\begin{eqnarray*}
G_{\varepsilon}(x)&:=& \sum_{\{i\mid \delta_i=\varepsilon\}} \gamma_i x^{v_i}v_i,\\
G_{\le\varepsilon}(x)&:= &\sum_{\{i\mid \delta_i\le\varepsilon\}} \gamma_i x^{v_i}v_i,\\\
G_{<\varepsilon}(x)&:=& \sum_{\{i\mid \delta_i<\varepsilon\}} \gamma_i x^{v_i}v_i.
\end{eqnarray*}
%with $\delta_i= \delta_i(\crit)$. 
Recall that we denote by $\varepsilon_0=0<\varepsilon_1<\dotsc< \varepsilon_m$ the elements of the set  $\mathcal E=\{\delta_i\mid i=1,\dotsc, n\}$.  Therefore $G(x)=\sum_{j=0}^m  G_{\varepsilon_j}(x)=G_{\le\varepsilon_m}(x)$. Observe that $G_{\le \varepsilon}$ takes values in $B_{\le \varepsilon}\otimes \KK$, see Definition~\ref{d:Bep}.

\end{defn}

\begin{defn}[$\VALVK$]
Suppose $V_\KK=V_\R\otimes_\R\KK$ where $V_\R$ is a real vector space and $\langle\ ,\ \rangle: V_{\R}\x V_{\R}^*\to \R$ denotes the dual pairing (and its bilinear extension over $\KK$). For any nonzero $F\in V_\KK$ we define $\VALVK(F)\in\R$ by 
\[
\VALVK(F):=\min_{u\in V_{\R}^*} (\ValK(\langle F,u\rangle)).
\]
Equivalently, $\VALVK(F)$ is the unique real number such that the expansion of $F$ in $t$ takes the form
\begin{equation}\label{e:F}
F=t^{\VALVK(F)}\sum_{\delta\ge 0} f_\delta t^\delta,
\end{equation}
with $f_\delta\in V_\C$ and $f_0\ne 0$. As usual $\VALVK(0):=\infty$. We also define $
\Coeff_{\mu}(F)\in V_\C$ to be the coefficient of $t^\mu$ in the expansion of $F$, thus if $F\ne 0$ is expressed as in \eqref{e:F}, then
\[
\Coeff_{\mu}(F)=f_{\delta}, \quad\text{ where $\delta=\mu-\VALVK(F)$.}
\]
 The coefficient $\Coeff_{\VAL_{V_{\KK}}(F)}(F)=f_0\in V_\C$ is also referred to as the leading vector-coefficient of $F$ and denoted simply $\Coeff(F)$, generalising Definition \ref{d:Coeff}. 
\end{defn}

Our construction of the positive critical point of $W$ will be recursive and relies on the following proposition. 

\begin{prop}\label{p:recursion} Let $W=\sum_i\gamma_i x^{v_i}$ be a positive, complete  Laurent polynomial. 
Let $\crit=\crit(\Xi_W)$ be the canonical point of the associated complete  Newton datum $\Xi_W$ and  $\dcoeff=\dcoeff(W)$ the critical coefficient for $W$. Set $\tau:=\Trop(W)(\crit)$. 

Let $G=\sum_i\gamma_i x^{v_i}v_i$ be the $M_\KK$-valued function associated to $W$ from \eqref{e:G}. Suppose $p\in T(\KK_{>0})$ has leading term $p_0=e^{\dcoeff}t^{\crit}$ and satisfies $G(p)\ne 0$. 
Set $
 \nu:=\VAL_{M_{\KK}}(G(p))-\tau.
 $
\begin{enumerate}
\item 
Suppose $\varepsilon_h$ is the unique minimal element of $\mathcal E=\{\varepsilon_0,\dotsc,\varepsilon_m\}$ such that  
\[
\Coeff_{\tau+\nu}(G(p))\in B_{\le\varepsilon_h}\otimes\C,
\] 
compare Definition~\ref{d:Bep}. Then we have that $0\le \varepsilon_h<\nu$. In particular $\nu>0$.  
\item 
There exists an element $u'\in B^\perp_{<\varepsilon_h}\otimes\C$ such that $ p':=p\exp_T(t^{\nu-\varepsilon_h}u')$
satisfies 
\begin{equation}\label{e:p'1}
\VAL_{M_\KK}(G(p'))\ge\tau+\nu
\end{equation}
and
\begin{equation}\label{e:p'2}
\Coeff_{\tau+\nu}(G(p'))\in B_{<\varepsilon_h}\otimes\C.
\end{equation}
In particular if $\varepsilon_h=0$ then $\Coeff_{\tau+\nu}(G(p'))=0$ and $\VAL_{M_\KK}(G(p'))>\tau+\nu$.
\end{enumerate}
\end{prop}

\begin{proof} First observe that $\nu\ge 0$. This is because the leading term of $p$ is $e^{\dcoeff}t^{\crit}$, and therefore each individual summand of $G(p)$ has valuation bounded below by $\Trop(W)(\crit)$.  
Note also that $\Coeff_{\tau+\nu}(G(p))$ is non-zero, by the definition of $\nu$. Thus our assumption on $\varepsilon_h$ can be stated as saying that 
\begin{equation}\label{e:eph}
\Coeff_{\tau+\nu}(G(p))\in ( B_{\le\varepsilon_h}\otimes\C )\setminus (B_{<\varepsilon_h}\otimes\C).
\end{equation}

First let us prove (1). It is equivalent to prove that $\nu>0$ and $ \Coeff_{\tau+\nu}(G(p))\in B_{<\nu}\otimes\C$.  If $\nu>\varepsilon_m$ then (1) holds automatically. Let us now suppose that $0\le\nu\le\varepsilon_m$. Thus either $\varepsilon_{j-1}<\nu\le \varepsilon_{j}$ for some $j\in[1,m]$, or $\nu=0$. Observe that 
\[
\Coeff_{\tau+\nu}(G(p)) = \Coeff_{\tau+\nu}(G_{\le\nu}(p))\in B_{\le \nu}\otimes\C.
\]
This is because for $\varepsilon>\nu$ all the summands  $\gamma_ip^{v_i}v_i$ in $G_{\varepsilon}$ have lowest degree $c_i+\langle v_i, \crit\rangle=\tau+\varepsilon>\tau+\nu$, and therefore these $G_\varepsilon$ do not contribute to $\Coeff_{\tau+\nu}(G(p)) $, leaving only $G_{\le \nu}$.  If $\nu\notin\mathcal E$ then we have that $B_{\le\nu}=B_{<\nu}$ and thus (1) follows immediately. We are left with the case where $\nu=\varepsilon_j$ for $j\in[0,m]$. In this case we have
\[
G_{\le \varepsilon_j}(p)=G_{<\varepsilon_j}(p)+G_{\varepsilon_j}(p),
\]
and since $G_{<\varepsilon_j}(p)\in B_{<\varepsilon_j}\otimes M_\KK$, by its definition, it remains to show that the coefficient
\[
\Coeff_{\tau+\nu}(G_{\varepsilon_j}(p))
\]
of $G_{\varepsilon_j}(p)$ lies in $B_{<\varepsilon_j}\otimes\C$.
The valuation of each summand $\gamma_i p^{v_i}v_i$  of $G_{\varepsilon_j}(p)$ is given by $\tau+\varepsilon_j=\tau+\nu$, and the corresponding lowest order term  is $\Coeff(\gamma_i)e^{\langle v_i,\dcoeff\rangle}v_i$, determined by the leading term of $p$. Therefore we have that
\[
\Coeff_{\tau+\nu}(G_{\varepsilon_j}(p))=\sum_{\{i\mid\delta_i=\nu\}}\Coeff(\gamma_i)e^{\langle v_i,\dcoeff\rangle}v_i.
\]
Now the critical coefficient condition \eqref{e:coeffcritcond}, satisfied by $\dcoeff$, implies that the above sum lies in $B_{<\varepsilon_j}=B_{<\nu}$. If $\nu>0$ we are therefore done. To rule out the case $\nu=0$ note that the critical coefficient condition implies in that case that $\Coeff_{\tau}(G_{0}(p))=0$, but since 
\[
\Coeff_{\tau}(G_{0}(p))=\Coeff_{\tau}(G_{\le 0}(p))=\Coeff_{\tau}(G(p))=\Coeff_{\tau+\nu}(G(p))\ne 0
\]
with the inequality following from the definition of $\nu$, we obtain a contradiction.  Thus we must have that $\nu>0$ and 
the proof of part (1) of the proposition is complete. 

We now turn to (2). Let $\varepsilon_h$ be as defined in part (1).  In order to define $u'$ we introduce a linear map $\phi_h: B_{<\varepsilon_h}^\perp/B_{\le\varepsilon_h}^\perp\to B_{\le\varepsilon_h}/B_{<\varepsilon_h}$ by
\[
\phi_h(\bar u):=\sum_{\{i\mid \delta_i=\varepsilon_h\}} \Coeff(\gamma_{i}) e^{\langle v_i,\dcoeff\rangle}\langle v_i, u\rangle  v_i \mod B_{<\varepsilon_h},
\]
where $\bar u$ denotes the coset in $ B_{<\varepsilon_h}^\perp/B_{\le\varepsilon_h}^\perp$ represented by $u$. Then $\phi_h$ is related to the symmetric bilinear form on $B_{<\varepsilon_h}^\perp/B_{\le\varepsilon_h}^\perp$,
\begin{equation}\label{e:Bh}
\mathbb B_h(\bar u_1,\bar u_2):= \sum_{\{i\mid \delta_i=\varepsilon_h\}} \Coeff(\gamma_{i}) e^{\langle v_i,\dcoeff\rangle}\langle v_i, u_1\rangle \langle v_i, u_2\rangle,
\end{equation}
by the equality
\[
\mathbb B_h(\bar u_1,\bar u_2)=\langle\phi_h(\bar u_1),\bar u_2\rangle.
\]
Note that $\mathbb B_h$ is positive definite and hence non-degenerate. This implies that $\phi_h$ is invertible. We  extend coefficients by tensoring with $\C$, but keep the notation $\phi_h$ for the resulting isomorphism of $\C$-vector spaces.

Recall that we have $\Coeff_{\tau+\nu}(G(p))\in B_{\le\varepsilon_h}\otimes\C$. 
Using the invertibility of $\phi_h$ we may choose a  $u'\in B_{<\varepsilon_h}^\perp\otimes \C$ so that 
\[
\phi_h(\bar u')=-\Coeff_{\tau+\nu}(G(p)) \mod B_{<\varepsilon_h}\otimes\C.
\]
By its definition, $u'$ has the property
\begin{equation}\label{e:ukj}
\Coeff_{\tau+\nu}(G(p))+\sum_{\{i\mid \delta_i=\varepsilon_h\}} \Coeff(\gamma_{i}) e^{\langle v_i,\dcoeff\rangle}\langle v_i, u'\rangle  v_i  \in B_{<\varepsilon_h}\otimes\C.
\end{equation}
 We set $p'=p\exp_T(t^{\nu-\varepsilon_h}u')$ for this $u'$. Note that the element  $t^{\nu-\varepsilon_h}u' $ of  $N_\KK$ now lies in $B_{<\varepsilon_h}^\perp\otimes \mathfrak m$, so that in particular $\exp_T(t^{\nu-\varepsilon_h}u')$ is well-defined. It remains to show that $p'$ satisfies both \eqref{e:p'1} and \eqref{e:p'2}.
Thus to finish the proof we need to analyse $G(p')$. 

We start by noting that $G(p')=\sum_{\varepsilon\in\mathcal E}G_\varepsilon(p')$ and 
\[
G_{\varepsilon}(p')=\sum_{\{i\mid\delta_i=\varepsilon\}}\gamma_i [p']^{v_i}v_i=\sum_{\{i\mid\delta_i=\varepsilon\}}\gamma_i p^{v_i}\exp(t^{\nu-\varepsilon_h}\langle v_i,u'\rangle) v_i,
\]
where the valuation of each summand is equal to $\tau+\varepsilon$. Simultaneously beginning to expand out the exponentials we also have that
\[
G_{\varepsilon}(p')=G_\varepsilon(p)+ t^{\nu-\varepsilon_h}\sum_{\{i\mid\delta_i=\varepsilon\}}\langle v_i,u'\rangle \left(\gamma_i p^{v_i}v_i  + t^{\tau+\varepsilon}\ N_{\mathfrak m} \right).
\]
If $\varepsilon>\varepsilon_h$, then each term in the second sum has valuation greater than  $\tau+\nu$.  Summing over all the $\varepsilon\in\mathcal E$ we may therefore write
\begin{equation*}
G(p')=G(p)+ t^{\nu-\varepsilon_h}\sum_{\{i\mid\delta_i\le \varepsilon_h\}}\langle v_i,u'\rangle\left(\gamma_i p^{v_i} v_i +  t^{\tau+\delta_i}\ N_{\mathfrak m}\right ) +   t^{\tau+\nu}\ N_{\mathfrak m}.
\end{equation*}
 Using that $u'\in B_{<\varepsilon_h}^\perp\otimes\C$, we can rewrite the middle term,
\begin{multline*}
t^{\nu-\varepsilon_h}\sum_{\{i\mid\delta_i\le \varepsilon_h\}}\langle v_i,u'\rangle\left(\gamma_i p^{v_i} v_i +  t^{\tau+\delta_i}\ N_{\mathfrak m}\right )\\
=t^{\nu-\varepsilon_h}\sum_{\{i\mid\delta_i= \varepsilon_h\}}\langle v_i,u'\rangle\left(\gamma_i p^{v_i} v_i +  t^{\tau+\varepsilon_h}\ N_{\mathfrak m}\right )\\
=t^{\nu-\varepsilon_h}\left(\sum_{\{i\mid\delta_i= \varepsilon_h\}}\langle v_i,u'\rangle\gamma_i p^{v_i} v_i\right ) +  t^{\tau+\nu} N_{\mathfrak m}.
\end{multline*}
Therefore all in all we have that 
\begin{equation}\label{e:G(p')}
G(p')=G(p)+ t^{\nu-\varepsilon_h}\left(\sum_{\{i\mid\delta_i= \varepsilon_h\}}\gamma_i p^{v_i}\langle v_i,u'\rangle v_i \right )+ t^{\tau+\nu} N_{\mathfrak m}.
\end{equation}
Recall that the first summand, namely $G(p)$, has valuation $\tau+\nu$. Each summand from the middle term also has valuation $\tau+\nu$. The remaining expression has valuation greater than $\tau+\nu$. 
It follows that $\VAL_M(G(p'))\ge \tau+\nu$. Thus \eqref{e:p'1} is proved.

Now for \eqref{e:p'2} we work out the vector coefficient $\Coeff_{\tau+\nu}(G(p'))$ explicitly. Note that
\begin{multline*}
\sum_{\{i\mid\delta_i= \varepsilon_h\}}\gamma_i p^{v_i}\langle v_i,u'\rangle v_i =t^{\tau+\varepsilon_h}
\left(\sum_{\{i\mid\delta_i=\varepsilon_h\}}
\Coeff(\gamma_i) e^{\langle v_i,\dcoeff\rangle}\langle v_i,u'\rangle v_i \right) \\ + t^{\tau+\varepsilon_h}N_{\mathfrak m}.
\end{multline*}
This expansion combined with  \eqref{e:G(p')} implies that   
\begin{equation*}
\Coeff_{\tau+\nu}(G(p'))
%=\Coeff_{\Trop(W)(\crit)+\nu}\left(G(p)+ t^{\nu-\varepsilon_h}\sum_{\{i\mid\delta_i=\varepsilon_h\}}\gamma_i p^{v_i}\langle v_i,u'\rangle v_i \right)\\
=\Coeff_{\tau+\nu}\left(
G(p)\right)+ \sum_{\{i\mid\delta_i= \varepsilon_h\}}\Coeff(\gamma_i) e^{\langle v_i,\dcoeff\rangle}\langle v_i,u'\rangle v_i.
\end{equation*}
Finally, comparing with \eqref{e:ukj} we deduce that $\Coeff_{\tau+\nu}(G(p'))$ lies in $B_{<\varepsilon_h}\otimes\C$ and \eqref{e:p'2} is proved. 
\end{proof}

 Our next step will be to recursively construct a sequence $(p_j)_j$ of elements of $T(\KK_{>0})$ of the form $p_{j}:=e^{\dcoeff}t^{\crit}\exp_T(w_j)$, starting with  $p_{0}:=e^{\dcoeff}t^{\crit}$. 
%Note that $p_{0}=e^{\dcoeff}t^{\crit}$ is also the leading term of each $p_k$. 

\vskip .2cm 

\begin{recursiveconstruction}\rm
Let $w_0=0$. Suppose $w_{k}\in N_{\mathfrak m}$ and hence $p_k=e^{\dcoeff}t^{\crit}\exp_T(w_k)$ in $T(\KK_{>0})$ has been defined. 
Recall Proposition~\ref{p:recursion} with all of its notations. We use this proposition to construct $w_{k+1}$ and $p_{k+1}$ as follows.

\begin{itemize}
\item If $G(p_{k})=0$ we set $w_{k+1}=w_k$, so that $p_{k+1}=p_k$.  In this case the sequence $(p_j)_j$ is constant for $j\ge k$.
\item If $G(p_k)\ne 0$, set
\begin{equation}\label{e:nuk}
 \nu_{k}:=\VAL_{M_{\KK}}(G(p_{k}))-\tau.
 \end{equation}
Then from Proposition~\ref{p:recursion} with $p=p_k$ and $\nu=\nu_k$ we have an associated $\varepsilon_{h(k)}:=\varepsilon_h\in\mathcal E$ with $\varepsilon_{h(k)}<\nu_k$ and a $u'_k:=u'\in B_{<\varepsilon_{h(k)}}^\perp\otimes\C$. We define $w_{k+1} \in N_{\mathfrak m}$ by 
\[
w_{k+1}:=w_k+t^{\nu_k-\varepsilon_{h(k)}} u'_k,
\]
and set $p_{k+1}:=e^{\dcoeff}t^{\crit}\exp_T(w_{k+1})$.
\end{itemize}
 By its definition, $\varepsilon_{h(k)}$ is the minimal element of $\mathcal E$ for which 
\[
\Coeff_{\tau+\nu_k}(G(p_{k}))\in B_{\le\varepsilon_{h(k)}}\otimes\C,
\]
while $u'_k$ is constructed so that $p_{k+1}$ satisfies  $\nu_{k+1}:=\VAL_{M_{\KK}}(G(p_{k+1}))-\tau\ge \nu_k$ and 
\begin{equation}\label{e:Coeffnew}
\Coeff_{\tau+\nu_k}(G(p_{k+1}))\in B_{<\varepsilon_{h(k)}}\otimes\C.
\end{equation}
The proposition thus implies that we have either 
\begin{enumerate}
\item $\nu_{k+1}=\nu_k$ and $\varepsilon_{h(k+1)}<\varepsilon_{h(k)}$, in the case that the coefficient \eqref{e:Coeffnew} is nonzero, 
\item or we have that  $\Coeff_{\tau+\nu_k}(G(p_{k+1}))=0$,   and therefore $\nu_{k+1}>\nu_k$. 
\end{enumerate}
Note that if $\varepsilon_{h(k)}=0$ then $B_{<\varepsilon_{h(k)}}=0$ and we are necessarily in the second case. 
\end{recursiveconstruction}
\begin{remark}\label{r:almostmonotone}
The points (1) and (2) above imply that not only is the sequence $0<\nu_0\le\nu_1\le \nu_2\le \dotsc$ monotonely increasing, also for each value $\nu\in\R_{> 0}$ the number of $j$ with $\nu_j=\nu$ is bounded above by the cardinality of $\mathcal E$ and in particular is finite. 
\end{remark}

\begin{remark}\label{r:nonuniqueness}
We are suppressing in our notation for the element $p_k$ its dependence on the choices of elements $u'_{0},\dotsc, u'_{k-1}$. These elements were not uniquely determined. Indeed, given $u'_{0},\dotsc, u'_{j-1}$ the next $u'_{j}\in B_{<\varepsilon_{h(j)}}^\perp\otimes\C$ is unique precisely up to $B_{\le\varepsilon_{h(j)}}^\perp\otimes \C$, as follows from the proof of Proposition~\ref{p:recursion}. 
\end{remark}

\begin{lemma} \label{l:convergence} Let $p_k\in T(\KK_{>0})$ and $w_k=\sum_{j=0}^{k-1} t^{\nu_j-\varepsilon_{h(j)}}u'_j\in N_{\mathfrak m}$ be elements constructed as above, where we assume that $G(p_k)\ne 0$ for all $k$. Then for every $\delta\ge 0$ the set $S_\delta=\{j\in\Z_{\ge 0}\mid \nu_j-\varepsilon_{h(j)}=\delta \}$ is finite. Moreover the set $\mathcal D=\{\nu_k-\varepsilon_{h(k)}\mid k\in\Z_{\ge 0}\}$  lies in $\MonSeq$. Note that $\mathcal D$ is also the set of all $ \delta$ such that $ S_{\delta}\ne 0$.   We have that 
\begin{equation}\label{e:winfty}
w_\infty:=\sum_{\delta\in\mathcal D}t^\delta \left(\sum_{j\in S_{\delta}} u'_j\right)
\end{equation}
is a well-defined element of $N_{\mathfrak m}$ and is the limit of the sequence $(w_k)_k$. 
\end{lemma}

\begin{proof} The assertion that $S_{\delta}$ is finite follows from Remark~\ref{r:almostmonotone} together with the fact that the $\varepsilon_{h(j)}$ lie in the finite set $\mathcal E$. 
We now prove that $\mathcal D\in\MonSeq$. The strategy of the proof will be to construct a set $\mathcal S\in\MonSeq$ which has the property that $\mathcal D\subset\mathcal S$.

Let $\Gamma_i:=\{\delta\in\R_{\ge 0}\mid \Coeff_{c_i +\delta}(\gamma_i)\ne 0\}$, so that
\[
\gamma_i=t^{c_i}\tilde\gamma_i=t^{c_i}\left(\sum_{\delta\in\Gamma_i}m_\delta t^\delta\right),
\]
and let $\partial\mathcal E:=\{\varepsilon'-\varepsilon\mid  \varepsilon,\varepsilon'\in \mathcal E, \varepsilon'\ge \varepsilon\}$.   Clearly $\Gamma_i\in\MonSeq$, since $\gamma_i\in\KK$. We define $\mathcal S$ 
to be the additive semigroup in $\R_{\ge 0}$ generated by the sets $\Gamma_i$ and $\partial\mathcal E$,
\[
\mathcal S:=\langle\Gamma_1,\dotsc,  \Gamma_n,\partial\mathcal E \rangle_{\Z_{\ge 0}}.
\]
 Since each $\Gamma_i\in\MonSeq$, and $\partial\mathcal E$ is finite, it follows that also $\mathcal S\in\MonSeq$. We now prove the following claim.
\vskip .2cm

\begin{paragraph}{\it Claim} $\mathcal D$ is a subset of $\mathcal S$, and therefore $\mathcal D$ lies in $\MonSeq$.
\end{paragraph}
\vskip .2cm

\begin{paragraph}{\it Proof of Claim} We need to show that for any $k\in\Z_{\ge 0}$ we have that $\nu_k-\varepsilon_{h(k)}$ lies in $\mathcal S$.
 Recall that $\mathcal\varepsilon_{h(k)}\in\mathcal E$ is characterised by the property that 
\begin{equation}\label{e:ephj}
\Coeff_{\tau+\nu_k}(G(p_k)) \in (B_{\le \varepsilon_{h(k)}}\otimes\C)\setminus (B_{< \varepsilon_{h(k)}}\otimes\C).
\end{equation}
 By definition, $\tau+\nu_k$ is the valuation of  
\begin{multline*}
G(p_k)=\sum_{i=1}^n \gamma_ip_j^{v_i}v_i=\sum_{\varepsilon\in\mathcal E}\sum_{\{i\mid\delta_i=\varepsilon\}}\gamma_ip_k^{v_i}v_i
\\=\sum_{\varepsilon\in\mathcal E}\sum_{\{i\mid\delta_i=\varepsilon\}}e^{\langle v_i,\dcoeff\rangle}t^{\tau+\varepsilon}\tilde\gamma_i \exp_T(\langle v_i,w_k\rangle)v_i
\\= \sum_{\varepsilon\in\mathcal E}\sum_{\{i\mid\delta_i=\varepsilon\}}e^{\langle v_i,\dcoeff\rangle}t^{\tau+\varepsilon}\tilde\gamma_i \exp_T\left(\sum_{j=0}^{k-1
} t^{\nu_j-\varepsilon_{h(j)}}\langle v_i,u'_j\rangle\right)v_i.
\end{multline*}
Note that $\Coeff_{\tau+\nu_k}(G(p_k)) \notin B_{< \varepsilon_{h(k)}}\otimes\C$ by \eqref{e:ephj}, therefore $t^{\tau+\nu_k}$ must appear with non-zero coefficient in  some $\delta_i=\varepsilon$ summand, where $\varepsilon\ge\varepsilon_{h(k)}$. Thus $\tau+\nu_k$ must be of the form
\[ \tau+\nu_k=\tau+\varepsilon+\delta + \sum_{j=0}^{k-1} \Z_{\ge 0}(\nu_j-\varepsilon_{h(j)}),
\] for some $\delta\in \Gamma$ and $\varepsilon\ge \varepsilon_{h(k)}$.  It follows that 
\begin{equation}\label{e:indclaim}
\nu_k-\varepsilon_{h(k)}=(\varepsilon-\varepsilon_{h(k)})+\delta + \sum_{j=0}^{k-1} \Z_{\ge 0}(\nu_j-\varepsilon_{h(j)}).
\end{equation}
If $k=0$ we have that $\nu_0-\varepsilon_{h(0)}=(\varepsilon-\varepsilon_{h(0)})+\delta $, which clearly lies in $\mathcal S$.  The claim for general $k$ now follows by induction. Namely suppose the claim is true for all $j\le k-1$. Then  all the terms in the right hand side of \eqref{e:indclaim} lie in $\mathcal S$, and therefore $\nu_k-\varepsilon_{h(k)}$ lies in $\mathcal S$, as required.
 \end{paragraph}
\vskip .2cm

Finally, note that since $\mathcal D\in\MonSeq$, we have for any $R>0$ an index $k=k(R)$ such that  $\nu_{\ell}-\varepsilon_{h(\ell)}>R$ for all $\ell\ge k$. Therefore, for this index $k$,
\[
w_\infty-w_k=\sum_{\ell=k}^{\infty}t^{\nu_\ell-\varepsilon_{h(\ell)}}u'_{\ell}\ \in\  t^R N_{\mathfrak m},
\] 
implying that $w_\infty-w_k$ tends to $0$ in the $t$-adic topology. Thus we have that $w_{\infty}=\lim_{k\to\infty}w_k$. 
\end{proof}

\begin{remark}\label{r:nuiinfty}
We note that this lemma implies, and indeed is equivalent to, the statement that if $G(p_k)\ne 0$ for all $k$, then the sequence $(\nu_k)_k$ from the Recursive Construction tends to infinity.
\end{remark} 
\begin{defn}\label{d:pinfty} If $p_k=e^{\dcoeff}t^{\crit}\exp_T(w_k)$ is a sequence of elements of $T(\KK_{>0})$ constructed as in the Recursive Construction above, then either $G(p_k)=0$ for some $k$ and we set $w_\infty=w_k$, or we set $w_\infty=\lim_{k\to\infty}w_k$ using Lemma~\ref{l:convergence}. We define
\[
p_{\infty}:=e^{\dcoeff}t^{\crit}\exp_T(w_\infty).
\]
The element $p_{\infty}$ of $T(\KK_{>0})$ is the limit of the sequence $(p_j)_j$. 
\end{defn}

\begin{cor}
If $(p_k)_k$ with $p_k=e^{\dcoeff}t^{\crit}\exp_T(w_k)$ is a sequence of elements of $T(\KK_{>0})$ constructed as in the Recursive Construction above, then its limit, $p_\infty$, is a positive critical point of $W$. 
\end{cor}

\begin{proof} If $G(p_k)=0$ for some $k$ then $p_\infty=p_k$ and is clearly a critical point of $W$. Otherwise, by Remark~\ref{r:nuiinfty}, the sequence $\nu_k=\VAL_{M_{\KK}}(G(p_k))-\tau$ tends to infinity. For any $\eta\in\R$ we therefore have that  $\eta<\nu_k$ for $k$  large enough (depending on $\eta$). But then $\tau+\eta<\VAL_{M_{\KK}}(G(p_k))$, implying that $\Coeff_{\tau+\eta}(G(p_k))=0$ for large enough $k$. Therefore $\Coeff_{\tau+\eta}(G(p_\infty))=0$, for all $\eta$ and $G(p_{\infty})=0$, implying that $p_{\infty}$ is a critical point of $W$.    
\end{proof}

We now prove an auxiliary lemma which will be used in the proof of uniqueness. 

\begin{lemma}\label{l:wvanishing}
Suppose $w\in N_\C$ has the property that 
\begin{equation}\label{e:Sep2}
\sum_{\{i\mid \delta_i=\varepsilon\}}\Coeff(\gamma_i)e^{\langle v_i,\dcoeff\rangle}\langle v_i,w\rangle v_i\in B_{<\varepsilon}
\end{equation}
for all $\varepsilon\ge 0$. Then $w=0$.
\end{lemma}

\begin{proof}
Dividing up $w$ into real and imaginary parts, we may assume that $w\in N_\R$. Recall that we have isomorphisms 
\begin{eqnarray*}
\phi_{h}:B_{< \varepsilon_h}^\perp/B_{\le\varepsilon_h}^{\perp}&\to &B_{\le \varepsilon_h}/B_{<\varepsilon_h}\\
 \omega + B_{\le\varepsilon_{h}}^\perp&\mapsto &\sum_{\{i\mid\delta_i= \varepsilon_h \}}\Coeff(\gamma_i)e^{\langle v_i,\crit\rangle}\langle v_i,\omega\rangle v_i\mod B_{<\varepsilon_h}. 
\end{eqnarray*}
We prove inductively that $w\in B_{\le \varepsilon_h}^\perp$ for all $h$. Since $B_{\le \varepsilon_m}^\perp=\{0\}$ this will imply that $w=0$. 

The start of the induction is given by $w\in B_{\le \varepsilon_0}^\perp=N_\R$. Suppose we have proved that  $w\in B_{\le \varepsilon_{h-1}}^\perp$. Then since $B_{\le \varepsilon_{h-1}}^\perp=B_{< \varepsilon_{h}}^\perp$ we may apply $\phi_h$ to $\bar w=w + B_{\le\varepsilon_{h}}^\perp$.   The condition~\eqref{e:Sep2} implies that 
$\phi_h(\bar w)=0$. But since $\phi_h$ is an isomorphism, this implies that $\bar w=0$, hence that $w \in B_{\le\varepsilon_{h}}^\perp$. This completes the induction.
\end{proof}

\begin{prop}[Uniqueness]\label{p:uniqueness} 
Suppose $p$ and $p'$ are two positive critical points of $W$, then $p=p'$.
\end{prop}

\begin{proof} Suppose $p=p_0\exp_T(w)$ and $p'=p_0\exp_T(w')$ where $p_0=e^{\dcoeff}t^{\crit}$.
Let us write $w=\sum_{\delta\in \mathcal C}t^\delta w_\delta$ and 
$w'=\sum_{\delta\in \mathcal C}t^\delta w'_\delta$
where $\mathcal C\in\MonSeq$ is the same set for $w$ as for $w'$, but the coefficients $w_\delta,w'_\delta\in N_\C$ are allowed to be zero.   We prove that $w=w'$ by induction on the totally ordered set $\mathcal C$. Let $\sigma\in \mathcal C$. Suppose that we have shown that $w_\delta=w'_\delta$ for all $\delta<\sigma$. If $\sigma$ is the minimal element of $\mathcal C$, then this induction assumption is trivially true, giving the start of the induction. Let us set
\[
G_{\ge\varepsilon}(p)=G(p)-G_{<\varepsilon}(p).
\]
Consider for any $\varepsilon\ge 0$  
\[
G_{\ge \varepsilon}(p)=\sum_{\{i\mid \delta_i\ge \varepsilon\}} \gamma_i p^{v_i}v_i= 
\sum_{\{i\mid \delta_i\ge \varepsilon\}} \gamma_i p_0^{v_i}\exp\left (\sum_{\delta\in\mathcal C}t^\delta \langle v_i,w_\delta\rangle\right ) v_i,
\]
and similarly for $ G_{\ge \varepsilon}(p')$ with each $w_\delta$ replaced by $w'_{\delta}$. Since $G(p)=G(p')$, we have that $G_{\ge \varepsilon}(p)- G_{\ge \varepsilon}(p')=0-(G_{< \varepsilon}(p)- G_{<\varepsilon}(p'))\in  B_{<\varepsilon}\otimes \KK$. Thus,
\begin{equation}\label{e:difference}
\sum_{\{i\mid \delta_i\ge \varepsilon\}} \gamma_i p_0^{v_i}\left[\exp\left (\sum_{\delta\in\mathcal C}t^\delta \langle v_i,w_\delta\rangle\right ) -\exp\left (\sum_{\delta\in\mathcal C}t^\delta \langle v_i,w'_\delta\rangle\right )\right] v_i\in B_{<\varepsilon}\otimes \KK.
\end{equation}
Note that by the induction hypothesis, any terms in the expansions of $ G(p)$ and $ G(p')$ involving only $w_\delta$ for $\delta<\sigma$ cancel out. 
The expression \eqref{e:difference} can be divided up into the sum $S_\varepsilon=\sum_{\{i\mid\delta_i=\varepsilon\}}(...)$ and $S_{>\varepsilon}=\sum_{\{i\mid\delta_i>\varepsilon\}}(...)$, giving $ G_{\ge\varepsilon}(p) -G_{\ge \varepsilon}(p')=S_{\varepsilon}+S_{>\varepsilon}\in B_{<\varepsilon}\otimes\KK$.
We consider now $
%\Coeff_{\tau+\sigma+\varepsilon}(\langle G(p),u\rangle -\langle G(p'),u\rangle)=
\Coeff_{\tau+\sigma+\varepsilon}(S_{\varepsilon}+S_{>\varepsilon})$.  
For the $S_{>\varepsilon}$ summand we have
\begin{equation*}
\Coeff_{\tau+\sigma+\varepsilon}(S_{>\varepsilon})=0.
\end{equation*}
This is because no $t^\delta\langle v_i,w_\delta\rangle$ with $\delta\ge\sigma$ can contribute to the  $\Coeff_{\tau+\sigma+\varepsilon}$ if already $\delta_i>\varepsilon$. But for all $\delta<\sigma$ we have that $w_\delta=w'_\delta$ and any contributions to  $\Coeff_{\tau+\sigma+\varepsilon}$ will cancel out. 

For the $S_{\varepsilon}$ summand we get
\begin{equation}\label{e:Sep}
\Coeff_{\tau+\sigma+\varepsilon}(S_{\varepsilon})=\sum_{\{i\mid \delta_i=\varepsilon\}}\Coeff(\gamma_i)e^{\langle v_i,\dcoeff\rangle}\langle v_i,w_\sigma-w'_{\sigma}\rangle v_i,
\end{equation}
using again the induction hypothesis to cancel out any other terms. Since $S_{\varepsilon}+S_{>\varepsilon}\in B_{<\varepsilon}\otimes\KK$, it now follows that \eqref{e:Sep} lies in $B_{<\varepsilon}\otimes\C$. Thus we have shown the following property of $w_{\sigma}-w'_{\sigma}$. 
\vskip .2cm
\begin{paragraph}{\it Property} For the element $w=w_{\sigma}-w'_{\sigma}$ of $N_\C$ we have that 
\begin{equation}\label{e:Sep2}
\sum_{\{i\mid \delta_i=\varepsilon\}}\Coeff(\gamma_i)e^{\langle v_i,\dcoeff\rangle}\langle v_i,w\rangle v_i\in B_{<\varepsilon}\otimes\C
\end{equation}
for all $\varepsilon\ge 0$.
\end{paragraph}
\vskip .2cm
It therefore follows by Lemma~\ref{l:wvanishing} that $w_{\sigma}-w'_{\sigma}=0$, completing the induction step. Thus the positive critical point is unique. 
\end{proof}

\begin{remark} While the $u'_j$ in the Recursive Construction of $p_k$ and hence of $p_{\infty}$ are non-unique, see Remark~\ref{r:nonuniqueness}, the uniqueness of the positive critical point, Proposition~\ref{p:uniqueness}, implies that the set $\mathcal D$ and the sums $\sum_{j\in S_{\delta}} u'_j\in N_\C$ from \eqref{e:winfty} are unique. 
\end{remark}

\begin{defn}  Let $W=\sum_i\gamma_i x^{v_i}$ be a positive, complete  Laurent polynomial. We denote by $\pcrit=\pcrit(W)$ the critical point of $W$ in $T(\KK_{>0})$, which is defined recursively above, see Definition~\ref{d:pinfty}, and whose uniqueness is proved in Proposition~\ref{p:uniqueness}. We refer to $\pcrit$ as the \emph{positive critical point} of $W$.
\end{defn}

\begin{proof}[Proof of Theorem~\ref{t:critmax}] We have already proved that the positive critical point of a positive, complete  Laurent polynomial exists and is unique, and its valuation is the canonical point associated to the complete  Newton datum. Therefore Theorem~\ref{t:critmax} follows from Lemma~\ref{l:critmax}. 
\end{proof}

 We end this section with a small, illustrative  example.
\begin{example} \label{ex:GWexample} Let $T=\KK^{*}$, and consider the complete and positive Laurent polynomial on $T$ given by 
\[
W=t^3 x^{-2}+tx\inv +x.
\]
This Laurent polynomial is taken from \cite[Example~1.11(b)]{GonzalezWoodward:Adv19}, where it appears as mirror to a GIT quotient of $\C^3$ by a particular polarized action of $(\C^*)^2$ (``projective line with extra term''). 

The complete Newton datum $\Xi_W=(M_\R,\R\oplus M_\R,\eta,\beta,\mathcal W)$ associated to $W$ as in Example~\ref{ex:XiW}, has $M_\R=\R$ and $\mathcal W=\{(3,-2),(1,-1),(0,1)\}$. The map $\eta:\R\to \R\oplus M_\R$ is the inclusion of the $\R$-axis, and the map $\beta=\pr_2$ is shown below, along with the augmented Newton polytope (a triangle), and its image the Newton polytope (the interval $[-2,1]$).

  \begin{tikzpicture}
\draw (-2,-2) node[anchor=west]{$(3,-2)$}
  -- (-4,-1) node[anchor=north]{$(1,-1)\quad $}
  -- (-5,1) node[anchor=east]{$(0,1)$}
  -- cycle;\fill[green!20!white](-2,-2)   -- (-4,-1)-- (-5,1) 
  -- cycle;
  \draw[thin] (-7,0) -- (-2,0) node[anchor=north]{${\R}$};
  \draw[thin] (-5,-2.5) -- (-5,2) node[anchor=west]{$M_\R$};
  \draw [->] (0,0) -- node[anchor=south]{$\pr_2$} (2,0) ;
 \draw[thin] (3,-2.5) -- (3,2) node[anchor=west]{$M_\R$};
  \draw (3,-2) node[anchor=west]{$-2$} -- (3,1)  node[anchor=west]{$1$} ;
    \draw[thick,green!40!white] (3,-2) -- (3,1)   ;
    \draw (2.95,-2)--(3.05,-2);
        \draw (2.95,1)--(3.05,1);
\end{tikzpicture}

Recall from Definition~\ref{d:minimalheightAugNewton} and Lemma~\ref{l:max} that the maximal value $\tau$ of $\Trop(W)$ is the `height above $0\in M_\R$' of $\AugNewton(W)$, namely $\tau$ is minimal for $\eta(\tau)=(\tau,0)$ to lie in $\AugNewton(W)$. This point $\eta(\tau)$ is clearly $(\frac 12,0)$, therefore $\tau=\frac 12$. The minimal face $F$  containing this point has codimension $1$. Therefore there is a unique vector $\crit$ in $N_\R$ such that $(1,\crit)$, viewed as an element of $(\R\oplus M_{\R})^*$, is constant on $F$. Namely we see that $\crit=\frac 12$. (Alternatively, $\crit$ is the gradient at zero of the function defined in a neighbourhood of $0$ in $M_\R$, whose graph is given by $F$). 

As a check, observe that $\Trop(W)(d)=\min(3-2d,1-d,d)$ indeed takes the value $\tau=\frac 12$ for $\crit=\frac 12$, and for any other $d\in N_\R$ the value $\Trop(W)(d)$ is clearly smaller. In particular $\crit$ lies in the polytope $\mathcal P_W=\{d\mid\Trop(W)(d)\ge 0\}=[0,1]$. Finally the Laurent polynomial $W$ has three critical points overall, which are $\pm t^{\frac 12}$ and $-2t^2$, see \cite[Example 1.14]{GonzalezWoodward:Adv19}. In agreement with our main theorem, we see that precisely one of these is positive, namely $\pcrit=t^{\frac 12}$. As computed above its valuation is $\crit=\frac 12$, and, as we see, $\dcoeff=1$. In this example there are no higher order terms.
\end{example}
\begin{example}
An example where the lowest face above $0$ of the augmented Newton polytope does \textit{not} have codimension $1$, so that the construction of the tropical critical point therefore requires more than one stage, is given by $T=(\KK^*)^2$ and the Laurent polynomial
\[
W=x_1 + x_2+t^{\frac 1 2}\frac 1 {x_2}+t^{-\frac 1 4}x_1 x_2 +t\frac{1}{x_1 x_2}.
\]
This example is related to a $2$-point blowup of $\C P^2$ discussed in \cite{KimLeeSanda:19}. The moment polytope $\mathcal P_W$ is a pentagon, each of the five terms of $W$ describing one of the sides. In \cite[Example 4.10]{KimLeeSanda:19}, Kim, Lee and Sanda find an entire line segment $\{(d_1,\frac 14)\mid \frac 1 4\le d_1\le \frac 3 8\}$ in the interior of $\mathcal P_W$ where the moment map fibers are non-displaceable Lagrangians. Nevertheless $W$ has only one tropical critical point. It is left to the reader to check that this special point agrees with the right-most point of the line segment, namely $\crit=(\frac 3 8,\frac 1 4)$. 
\end{example}

\begin{comment}
\begin{tikzpicture} \draw (-2,-2) node[anchor=west]{$(3,-2)$}
  -- (-4,1) node[anchor=south]{$(1,1)$}
  -- (-5,1) node[anchor=east]{$(0,1)$}
  -- cycle;\fill[green!20!white](-2,-2)   -- (-4,1)-- (-5,1) 
  -- cycle;
  \draw[thin] (-7,0) -- (-2,0) node[anchor=north]{${\R}$};
  \draw[thin] (-5,-2.5) -- (-5,2) node[anchor=west]{$M_\R$};
  \draw [->] (0,0) -- node[anchor=south]{$\pr_2$} (2,0) ;
 \draw[thin] (3,-2.5) -- (3,2) node[anchor=west]{$M_\R$};
  \draw (3,-2) node[anchor=west]{$-2$} -- (3,1)  node[anchor=west]{$1$} ;
    \draw[thick,green!40!white] (3,-2) -- (3,1)   ;
    \draw (2.95,-2)--(3.05,-2);
        \draw (2.95,1)--(3.05,1);
\end{tikzpicture}
\end{comment}

\section{Rationality, generalised Puiseaux polynomials, and nondegeneracy}\label{s:rationalityetc}

\subsection{Rationality}\label{s:rationality}
In this section we observe that Theorem~\ref{t:main} holds also over the subfield $\mathcal K$ of Puiseaux series with its positive part $\mathcal K_{>0}$, see Definition~\ref{d:rationalT}. 
The following result is a corollary of the proof of Theorem~\ref{t:main} given in Section~\ref{s:mainproof}.

\begin{prop}\label{p:rational} Suppose $W=\sum\gamma_i x^{v_i}$ is a complete  Laurent polynomial with coefficients $\gamma_i\in\mathcal K_{>0}$. Then the unique positive critical point $\pcrit$ from Theorem~\ref{t:main} lies in $T(\mathcal K_{>0})$. In particular its associated tropical point  $[\pcrit]$ gives rise to a canonical point $\crit\in N_{\mathbb Q}$.
\end{prop}

\begin{proof} A priori $W$ has a unique positive critical point in $\KK_{>0}$. From the construction in Section~\ref{s:critconstr} it is clear that $\crit$ lies in $N_{\Q}$ if all the $v_i$ have valuation in $\Q$. This proposition follows by tracing through the proofs in Section~\ref{s:mainproof} to see that the remaining terms in the construction of $\pcrit$ also have rational exponents only. 
\end{proof}

\subsection{Generalised Puiseaux polynomials} We recall the Lemma~\ref{l:GalkinGen} of Galkin. This lemma can be restated the following way. Suppose we have a function
\[
W: T(\R_{>0})\to \R_{>0},\qquad W=\sum c_i x^{\nu_i} 
\]
where $\nu_i\in M_\R$ and $c_i\in\R_{>0}$. Thus $W$ is `positive' over $\R$, but the exponents need not be integral. Then $W$  still has a Newton polytope, namely $\ConvexHull(\{\nu_i\})$, and Galkin's Lemma says that if the Newton polytope  is full-dimensional with $0$ in its interior, then $W$ has a unique critical point in $T(\R_{>0})$. 

As an analogue of this result we have the following generalisation of Theorem~\ref{t:main}. Consider a split torus $T$ over $\KK$ with character group $M$, and call $W=\sum \gamma_i x^{v_i}$ where $\gamma_i\in \KK_{>0}$ and now $v_i\in M_{\R}$ a {\it generalised Puiseaux polynomial} over $\KK_{>0}$. Note that $W$ still gives a well-defined map
\[W: T(\KK_{>0})\to \KK_{>0}.
\]
Indeed recall that elements of $T(\KK_{>0})$ are of the form $p=e^u t^d\exp(w)$ for $u,d\in N_\R$ and $w\in \mathfrak m$, compare Section~\ref{s:exp}. And therefore we can evaluate $x^{v_i}$ on $p $ by setting $p^{v_i}=e^{\left< v_i,u \right>}t^{\left< v_i,d \right>}\exp({\left< v_i,w \right>})$.

To $W$ we associate its Newton polytope $\Newton(W):=\ConvexHull(\{v_i\})$ inside $M_{\R}$, and we again call $W$ complete  if its Newton polytope is full-dimensional and contains $0$ in its interior.

\begin{theorem}\label{t:genmain} Suppose $W=\sum_{i=1}^n \gamma_i x^{v_i}$ is a complete, generalised Puiseaux polynomial over $\KK_{>0}$. Then $W$ has a unique critical point $\pcrit$ in $T(\KK_{>0})$, and the valuation $\crit$ of $\pcrit$ is the canonical point defined in Section~\ref{s:critconstr} of the complete  Newton datum $\Xi_W=(M_\R,\R\oplus M_\R,\eta,\beta,\{v_i\})$, compare Example~\ref{ex:XiW}.
\end{theorem}

\begin{proof} The proof of Theorem~\ref{t:main} given in Section~\ref{s:mainproof} does not make any use of the integrality of the exponents $v_i$. Therefore exactly the same proof shows the generalised version, Theorem~\ref{t:genmain}.
\end{proof}

\begin{remark}\label{r:rational2} We also get a rational version of the above theorem if we replace {\it generalised Puiseaux polynomials} over $\KK_{>0}$ by {\it Puiseaux polynomials} over $\mathcal K_{>0}$, where Puiseaux polynomial means that the exponents are rational numbers.  
\end{remark}

\subsection{Nondegeneracy}
Finally, the following observation can be seen as a direct analogue of the analytic property enjoyed by Galkin's positive critical point. Namely a positive, complete  Laurent polynomial over $\R$ defines a convex function $f:\R_{>0}^d\to \R$  (that is whose Hessian is everywhere positive definite), and this has a unique {\it nondegenerate} critical point which is a global minimum. 
\begin{lemma}\label{l:nondeg} Given any positive generalised Puiseaux polynomial $W=\sum_{i=1}^n \gamma_i x^{v_i}$ and any $p\in T(\KK_{>0})$, the bilinear map $\mathcal H_p: N_\R\x N_\R\to \KK$ defined by $\mathcal H_p(u_1,u_2):=\partial_{u_1}\partial_{u_2}W(p)$, has the property  that $\mathcal H_p(u,u)\in \KK_{>0}$ for all $u\in N_\R$.
\end{lemma}

\begin{proof} It follows immediately from the formula for $W$ and the definition of $\partial_u$ that
\[
\mathcal H_p(u_1,u_2)=\sum_{i=1}^n \gamma_i\langle v_i,u_1\rangle\langle v_i, u_2\rangle p^{v_i}.
\] 
Thus for $u_1=u_2=u$ we get $
\mathcal H_p(u,u)=\sum_{i=1}^n \gamma_i\langle v_i,u\rangle^2 p^{v_i}$, which is clearly in $\KK_{>0}$, since  $p\in T(\KK_{>0})$ and also $\gamma_i\in \KK_{>0}$, and moreover  $\langle v_i,u\rangle^2\in\R_{>0}$. \end{proof}

\begin{remark} There is also a stronger  version of nondegeneracy which can be required of a critical point of $W$. This property, called {\it strongly nondegenerate}, was introduced  by Fukaya, Oh, Ohta and Ono in the context of critical points of potential functions for symplectic toric manifolds, see \cite[Definition 10.2]{FOOO:I}. We think this generalises in our setting to the condition that the quadratic form $\mathbb B_h$ from \eqref{e:Bh} be non-degenerate (whenever the underlying vector space, $B_{<\varepsilon_h}^\perp/B_{\le\varepsilon_h}^\perp$ is nonzero). Our positive critical point ${\pcrit}$ is strongly nondegenerate in this sense, since the bilinear form $\mathbb B_h$ in the case of $\pcrit$ is always positive definite.
\end{remark}

\section{Applications in toric algebraic geometry} \label{s:toric}
Let us consider a complete normal toric variety $X$ for the torus $T^\vee$ over $\C$ with cocharacter group $M\cong \Z^r$. Namely $X=X_\Sigma$ for a complete rational polyhedral fan $\Sigma$ in $M_\R$. We refer to \cite{Fulton:toricbook} for background. Let $v_1,\dotsc, v_n\in M$ denote the primitive generators of the rays of the fan $\Sigma$. The associated torus-invariant divisors of $X$ are denoted $D_1,\dotsc, D_n$. The toric variety $X$ will be fixed throughout this section.

To any $T^\vee$-invariant $\R$-Weil divisor $D=\sum c_i D_i$ of $X$ we associate a Laurent polynomial $W=W_{(X,D)}:T(\KK)\to \KK$ by
\begin{equation}\label{e:WD}
W_{(X,D)}=\sum_{i=1}^{n}t^{c_i}x^{v_i},
\end{equation}
see Sections~\ref{s:applicationsintro1} and \ref{s:applicationsintro2}. We refer to  Section~\ref{s:symplectic} for background on this function associated to $X$ in the case $X$ is a smooth projective toric variety,  where it was previously introduced in the context of mirror symmetry. In this section we do not make any smoothness assumptions.  
Note that the Laurent polynomial \eqref{e:WD} is clearly positive and it is also complete, by completeness of the fan $\Sigma$.  If $D$ is  a $\Q$-Weil divisor then $W_{(X,D)}$ is defined over the field of Puiseux series $\mathcal K$.  

Let us consider the tropical critical point $\crit\in N_\R$ of $W_D$. In case $D$ is  a $\Q$-Weil divisor, the tropical critical point $\crit$ lies in $N_\Q$, by Section~\ref{s:rationality}.
If $D$ is a Weil divisor on the other hand (i.e. with $\Z$-coefficients), it does not follow that the tropical critical point lies in the $\Z$-lattice $N$. Therefore we make the following definition.

\begin{defn}\label{d:intbal} Call a $T^\vee$-invariant Weil divisor $D=\sum c_i D_i$ on $X_{\Sigma}$ {\it integrally balanced} if the tropical critical point of the associated Laurent polynomial $W_D$ lies in $N$.  
\end{defn}

It is easy to see that being integrally balanced is a property of the divisor class $[D]$. Namely we have the following lemma.  

\begin{lemma}\label{l:intbal} Suppose $D$ and $D'$ are linearly equivalent $T^\vee$-invariant Weil divisors on $X_\Sigma$. That is $D'=D+(\chi^d)$, for some $d\in N$. Then the tropical critical point $\crit$ of $W_D$ and the tropical critical point $\crit'$ of $W_{D'}$ are related by the formula
\[
\crit'=\crit-d.
\] 
In particular $D$ is integrally balanced if and only if $D'$ is integrally balanced.  
\end{lemma}

\begin{proof}
 Recall that the relation between divisors $D_i$ and the primitive vectors $v_i$ is encapsulated in the equality $(\chi^d)=\sum_{i=1}^n \langle v_i, d\rangle D_i$. Therefore if $D=\sum c_i D_i$, then we have that $D'=\sum (c_i+
 \langle v_i,d \rangle) D_i$ and $
W_{D'}=\sum_{i=1}^n t^{c_i+\langle v_i, d\rangle } x^{v_i}$.

Let $p':= t^{-d } \,\pcrit$, where we are using the notation from Definition~\ref{d:t^d}, and observe that it is a critical point of $W_{D'}$. 
Namely the logarithmic derivative along $u\in N_\R$ gives
\[
\partial_u W_{D'}(p')=\sum_{i=1}^n \langle v_i, u\rangle  t^{c_i+\langle v_i, d\rangle } t^{-\langle v_i, d\rangle } p^{v_i}=\sum_{i=1}^n \langle v_i, u\rangle  t^{c_i }  p^{v_i}=\partial_u W_{D}(p)=0.
\]
By the uniqueness statement in Theorem~\ref{t:main} it follows that $p'$ is the positive critical point of $W_{D'}$. Hence $\crit'=\Val(p')=\Val(t^{-d}\pcrit)=\crit-d$.
\end{proof}

\begin{cor}\label{c:critdivisor} In any integrally balanced divisor class $[D]$ there is a unique choice of $T^\vee$-invariant divisor $D'$ such that the tropical critical point $\crit'$ of $W_{D'}$ is equal to zero. \qed
\end{cor}

\begin{remark} The anticanonical class is integrally balanced for any toric variety $X_\Sigma$. In this case the unique $T^\vee$-invariant divisor with associated tropical critical point equal to $0$ just recovers the standard choice of anticanonical divisor given by the toric boundary divisor  $\sum_{i=1}^n D_i$.  
\end{remark} 

\begin{remark} The interpretations in this section, and particularly Corollary~\ref{c:critdivisor}, are a toric variety version of the first author's work in the setting of the flag variety $GL_n(\C)/B$, see \cite{Judd:Flag}. In that case the tropical critical point of the mirror Landau-Ginzburg model, with quantum parameters $q_i=t^{\langle\lambda, \omega_i^\vee\rangle}$, was computed directly and it picks out a special point in the Gelfand-Zetlin  polytope of the representation $V(\lambda)$. The paper \cite{Judd:Flag} also includes an analysis of which $\lambda$ correspond to `integrally balanced' divisor classes in $GL_n(\C)/B$.
\end{remark}

\begin{remark}\label{r:PD} If $D$ is very ample then we also note that the polytope    
\[
\mathcal P_{W_{(X,D)}}=\{d\in N_\R\mid \Trop(W_{(X,D)})(d)\ge 0\}
\] 
has an interpretation as a moment polytope $P_D$ of the toric variety $X$. Moreover the tropical critical point $\crit$ of $W_{(X,D)}$ automatically lies inside $\mathcal P_{W_{(X,D)}}$ by Theorem~\ref{t:critmax}. 
\end{remark}

\section{Applications in symplectic toric geometry} \label{s:symplectic}
We recall an approach to constructing non-displaceable Lagrangians using mirror symmetry that was introduced in the work of Fukaya, Oh, Ono and Ota for toric symplectic manifolds, see \cite{FOOO:survey,FOOO:I, FOOO:II,KimLeeSanda:19}, and in another version, for toric symplectic orbifolds, in work of Woodward \cite{Woodward:11}. We carry on using the notations from Section~\ref{s:tropicalisation}, but in this setting the {\it compact} form $T^\vee_c$ of the torus $T^\vee$ is acting, and we identify $M_\R$ with the Lie algebra $\mathfrak t_c^\vee$ of $T^\vee_c$, and thus $N_\R$ with its dual.

\subsection{Symplectic toric manifolds} Suppose $X$ is a smooth, compact toric symplectic manifold associated to a Delzant polytope\footnote{A Delzant polytope in $N_\R\cong\R^r$ is a polytope for which every vertex $d$  has $r$ edges whose edge directions can be represented by a $\Z$-basis of $N$, and these are precisely the moment polytopes of toric symplectic manifolds, see \cite{Guillemin:book,Canasdasilva:notes}.} $P$ in $N_\R$. Therefore $X$ has a Hamiltonian action of the torus $T^\vee_c$ and $P$ is the image is of the associated moment map, $P=\mu_X(X)$. Consider the canonical  description of $P$ as intersection of half-spaces, 
\[
P=\bigcap_{F_i} \{ d\in N_\R\mid \langle v_i,d\rangle + c_i\ge 0\},
\]
where $v_i\in M$ are the primitive inward pointing normals to the facets of $P$ and $c_i\in \R$.  This data is equivalently encoded in the Laurent polynomial 
\begin{equation}\label{e:WX}
W_{X}:=\sum_{i=1}^n t^{c_i}x^{v_i}.
\end{equation} This function $W_{X}$ is called the `leading order potential function' in \cite{FOOO:I} and is a kind of superpotential associated to $X$, as in the work of  \cite{Givental:Toric,Hori-Vafa, Batyrev:QCoh}. 
%, though the full potential function constructed in \cite{FOOO:I} may contain infinitely many higher order terms if $X_P$ is not Fano.  
In \cite{FOOO:I} the function $W_{X}$ is a first approximation to a full potential function $\mathfrak {PO}_{X}$ constructed using moduli spaces of holomorphic discs in $X$. This full potential function, considered in the correct coordinates, exactly recovers $W_{X}$ in the case where $X$ is Fano. If $X$ is not Fano, then the potential function $\mathfrak {PO}_{X}$ still has $W_{X}$ as its leading term, but contains possibly infinitely many additional summands of higher order. The definition of $\mathcal P_W$ (as in \eqref{e:PW})  still makes sense for $W= {\mathfrak {PO}_{X}}$, though, and agrees with the polytope  $\mathcal P_{W_{X}}$, which in turn agrees with the moment polytope $P$ of $X$.  An application of the potential function  from \cite{FOOO:I,FOOO:II} is that if $p$ is any critical point of $\mathfrak {PO}_{X}$ such that the valuation $d$ of $p$ lies in $P$, then the Lagrangian torus fiber $\mu_X\inv (d)$  is a {\it non-displaceable} Lagrangian in $X$. Moreover in many cases the condition can be weakened, using some approximations so that it may suffice to consider $W_X$. 
Indeed, Fukaya, Oh, Ohta and Ono show the existence of a non-displaceable Lagrangian in $X_P$ by constructing a particular point $u_0$ in $P$ which they show is the valuation of a critical point of $W_X$. This point $u_0$ turns out to agree precisely with the tropical critical point $\crit$, namely we have.

\begin{lemma}\label{l:u0} 
The tropical critical point $\crit $ of $W_X$ from \eqref{e:WX} agrees with the special point $u_0$ of the polytope $P$ constructed in   \cite[Proposition~9.1]{FOOO:I}, see also \cite {McDuff:probes}. In particular $u_0$ is the unique point of $N_\R$ arising as the valuation of a critical point in $T(\KK_{>0})$ of $W_X$. 
\end{lemma}

\begin{proof}
The critical point of $W_X$ with valuation $u_0$ constructed in \cite{FOOO:I} has positive leading coefficient, see \cite[Lemma~9.15]{FOOO:I}. Therefore by our uniqueness result, Proposition~\ref{p:uniqueness}, we see that it must agree with our positive critical point, and hence $u_0=\crit$.  
It is also possible to directly compare the construction of $u_0$ in $P$ from \cite[Proposition~9.1]{FOOO:I} to our construction of $\crit$  from Section~\ref{s:critconstr}. Namely the two constructions can be related by a version of Lemma~\ref{l:critmax}  adapted to each recursion step and proved  using the ingredients of Lemmas~\ref{l:max} and \ref{l:tauattainment} in the same way as Lemma~\ref{l:critmax}.
\end{proof}

As a consequence we have the following corollary which is {\cite[Theorem 1.5]{FOOO:I}} with $u_0$ replaced by $\crit$. 
\begin{cor} \label{c:symplecticsmooth} The canonical Lagrangian torus fiber $L_X:=\mu_X\inv(\crit)$ is a non-displaceable Lagrangian in $X$. Moreover if $\psi$ is a Hamiltonian diffeomorphism of $X$ such that  $L_X$ and $\psi(L_X)$ are transversal, then 
\[\# (\psi(L_X)\cap L_X)\ge 2^r,
\]
where $r$ is the dimension of $L_X$.
\end{cor}

\subsection{Symplectic toric orbifolds} Suppose now that $X$ is a quasiprojective toric orbifold, see \cite{Woodward:11}. Namely $X$ can be constructed as a symplectic quotient as follows. Suppose $G\subset (U(1))^n$ is a subtorus acting on $\C^n$ and the associated moment map $\Phi:\C^n\to \mathfrak g^*$ is such that $\Phi\inv(0)$ only has finite stabilisers as $G$-space. Then $X:=\C^n/G$ is a quasi-projective toric orbifold for the action of the residual torus $T^\vee_c= (U(1))^n/G$. Here as before we consider $T^\vee_c$ to be the compact torus whose Lie algebra we identify with $M_\R$, and we denote by $\mu_X: X\to N_\R$ the moment map associated to the action of  $T^\vee_c$ on $X$. 
In this setting Woodward constructs an $A_{\infty}$ algebra  associated to a Lagrangian torus fiber of $X$ (via its lift from $X$ to $\C^n$) which he calls the `gauged' or `equivariant Fukaya algebra' and which relates to a `quasimaps' version of Floer cohomology which can detect non-displaceability of Lagrangians. Moreover he constructs out of his $A_\infty$ algebra a `gauged potential function' $W_X$, which can be described combinatorially as sum of $n$ terms in an analogous way to \eqref{e:WX}, see \cite[Corollary 6.4]{Woodward:11}, and again determines the moment polytope $P$. Woodward then shows that if a point $d$ of the moment polytope $P$ of $X$ is the valuation of a critical point of $W_X$, then the associated Lagrangian torus fiber $L=\mu_X\inv(d)$ is  non-displaceable by any Hamiltonian diffeomorphism $\psi$, with lower bound on the number of intersection points in $L\cap\psi(L)$ whenever the intersection is transversal.
Combining these results with our Theorems~\ref{t:main} and \ref{t:critmax} we obtain the following corollary. 
      
\begin{cor}\label{c:symplecticorbifold} Let $X$ be a compact toric symplectic orbifold  of (real) dimension $2 r$. Then the gauged potential function $W_X$ of $X$ has a unique positive critical point $\pcrit$. Moreover the associated tropical critical point $\crit$ is a canonical point in the moment polytope  $P$ of $X$.

Denote the moment map fiber of $\crit$ by $L_X$. Then $L_X$ is a non-displaceable Lagrangian torus with the property that if $\psi$ is any Hamiltonian diffeomorphism of $X$ for which $L_X$ and $\psi(L_X)$ are transversal, then 
\[\# (\psi(L_X)\cap L_X)\ge 2^r.
\]
\end{cor}
 We think of $L_X$ as a canonical choice of `central' Lagrangian torus fiber $L_X=\mu_X\inv(\crit)$ in any toric orbifold $X$. 

\begin{proof}[{Proof of Corollary~\ref{c:symplecticorbifold}}]
  $W_X$ is clearly a positive Laurent polynomial. If $X$ is compact then the associated fan is complete and hence $W_X$ is also complete.  Thus the existence and uniqueness of $\pcrit$ follows from Theorem~\ref{t:main}. That $\crit$ lies in the moment polytope $P$ follows from Theorem~\ref{t:critmax}. The non-displaceability of $L_X$ follows from \cite[Theorem 1.1]{Woodward:11}, with the lower bound on intersection points shown in \cite[Proposition~6.7]{Woodward:11}
\end{proof}

\section{Mutations and extension from tori to cluster varieties}\label{s:mutations}

The theory of cluster algebras was introduced by Fomin and Zelevinsky~\cite{FZ1,FZ2,FZ3,FZ4} as a way to try to understand total positivity and the dual canonical basis of Lusztig.  Associated \emph{cluster varieties} give rise to an interesting generalisation of toric varieties. Here a single torus is replaced by multiple tori (called cluster tori), with fixed coordinates (constituting together with the cluster torus a \emph {cluster seed}). These are glued together along prescribed birational maps, called mutations.  A major development in the theory of cluster algebras was the duality conjectures of Fock and Goncharov \cite{FG}, and the introduction of the `theta basis' by Gross, Hacking, Keel and Kontsevich \cite{GHKK}. Moreover \cite{GHKK} use the theta basis to construct a kind of superpotential associated to a compactification of an affine cluster variety in the case that the Fock-Goncharov conjectures hold. Namely this \emph{GHKK-superpotential} is a sum of theta-functions on the dual cluster variety, and in particular gives rise to a Laurent polynomial on every cluster torus.  Another setting where mutations of a more general type have appeared is in the general setting of Fano mirror symmetry via Laurent polynomials, see \cite{ACGK:mutations, CKPT}. 

\begin{defn}\label{d:posbirat}
Suppose $\phi:T^{(1)}\dashrightarrow T^{(2)}$ is a positive rational map between two tori over $\KK$, as in Section~\ref{s:tptropicalisation}. We call $\phi$ \emph{positively birational} if it has a positive rational inverse $\phi\inv:T^{(2)}\dashrightarrow T^{(1)}$. 
\end{defn}

We note that all cluster mutations are positively birational maps. In this section we show that our positive critical point is preserved by any positively birational map of tori under which $W$ stays positive and Laurent. Thus in the setting of cluster varieties, by the `positive Laurent phenomenon' \cite{LeeSchiffler,GHKK}, the positive critical point is well-defined independently of a choice of `seed'. 

\begin{prop}  Suppose $\phi:T^{(1)}\dashrightarrow T^{(2)}$ is a positively birational map between two tori over $\KK$, as in Definition~\ref{d:posbirat}. Let $W: T^{(2)}\to\KK$ be a positive Laurent polynomial such that  $\phi^*(W)$ is a positive Laurent polynomial on $T^{(1)}$. Assume furthermore that $W$ is complete.
Then we have that 
\begin{enumerate}
\item
 $\phi^*(W)$ is complete. 
 \item
If  $\pcrit'$  is a unique positive critical point of  $\phi^*(W)$ from Theorem~\ref{t:main}, then  $\phi(\pcrit')\in T^{(2)}(\KK_{>0})$ is the unique positive critical point $\pcrit$ of $W$.
% then $\pcrit':=\phi\inv(\pcrit)\in T^{(1)}(\KK_{>0})$ is a critical point of $\phi^*(W)$.
\item
The tropical critical points ${\crit}=[\pcrit]$ and ${\crit'}=[\pcrit']$, of $W$ and $\phi^*(W)$, respectively, are related by  $\Trop(\phi)(\crit')=\crit$.   
\end{enumerate}
\end{prop}

Let us first prove the following two Lemmas. 

\begin{lemma}\label{l:contra1}
Suppose $W$ is a positive Laurent polynomial on the torus $T$ over $\KK$ with full-dimensional Newton polytope. If $W$ has a critical point in $T(\KK_{>0})$ then the Newton polytope of $W$ must have $0$ in its interior, i.e. then $W$ is complete. 
\end{lemma}
\begin{proof}
Suppose indirectly that $W$ has a positive critical point $p$, but $W$ is not complete. In this case we can find a primitive element $u_0\in N$ 
%with associated hyperplane  \[ H_0=\{v\mid \langle v ,d_0\rangle=0 \} \] in $N_\R$, 
 such that the Newton polytope $\Newton(W)$ lies entirely in the nonnegative half-space  $H_{\ge 0}=\{v\in M_\R\mid \langle v ,u_0\rangle\ge 0 \}$ of $M_{\R}$.
Since  $p$ is a critical point, the logarithmic derivatives $\partial_u W$ vanish at $p$. Therefore in particular we obtain
\[
\partial_{u_0}W (p)=\sum_i \gamma_i p^{v_i}\langle v_i,u_0\rangle=0.
\] 
However $\langle v_i,u_0\rangle \ge 0$ for all $i$ since $v_i$ lies in the Newton polytope which lies in $H_{\ge 0}$. Moreover since $\Newton(W)$ is full-dimensional, there exists a $v_i$ with $\langle v_i,u_0\rangle >0$, that is, strictly bigger than $0$. On the other hand each $p^{v_i}$ lies in $\KK_{>0}$, and the $\gamma_i$ lie in $\KK_{>0}$. 
It follows that  
$ \sum \gamma_i p^{v_i}\langle v_i,u_0\rangle \in \KK_{>0}$, given that at least one summand is non-vanishing and all leading terms have $\ge 0 $ coefficients. As a result we obtain a contradiction to the critical point equation $\partial_{u_0}W (p)=0$.
 \end{proof}

\begin{lemma}\label{l:contra2}
Suppose $W$ is a positive Laurent polynomial on the torus $T$ over $\KK$ whose Newton polytope is not full-dimensional. If $W$ has a critical point in $T(\KK_{>0})$ then $W$ has a $1$-parameter family of critical points in $T(\KK_{>0})$. 
\end{lemma}
\begin{proof}
Let $u_0\in N$ be chosen so that the Newton polytope of $W$ is contained in the affine hyperplane $H_c:=\{v\mid \langle v, u_0\rangle=c\}$.  Let $p_0$ be a critical point of $W$ in $T(\KK_{>0})$. Then for any $s\in\R$ we have that $p_s:=p_0 e^{s u_0} $ is another critical point of $W$ inside $T(\KK_{>0})$. Namely, for $W=\sum_i\gamma_i x^{v_i}$ and for any $u\in N$,
\[
\partial_u W(p_s)=\sum_i \langle v_i,u\rangle \gamma_i p_0^{v_i} e^{s \langle v_i,u_0\rangle}=e^{sc}\sum_i \langle v_i,u\rangle \gamma_i p_0^{v_i} =e^{sc} \partial_u W(p_0)=0.
\]
Here we used first that $\langle v_i,u_0\rangle=c$, since $v_i$ is a vertex of $\Newton(W)$ and thus lies in $H_c$, and then we used that $p_0$ is a critical point for $W$.   
\end{proof}

\begin{proof}[Proof of the Proposition]
Since $\phi$ is birational we can remove a divisor $D^{(1)}$ from $T^{(1)}$ and $D^{(2)}$ from $T^{(2)}$ and obtain an isomorphism, which we still call $\phi$,   
\begin{equation*}
\begin{tikzcd}  T^{(1)}\setminus D^{(1)} 
\arrow{r}{\phi} \ar{dr}[swap]{\phi^*(W)} &T^{(2)}\setminus D^{(2)}\arrow{d}{W}\\
& \KK.
\end{tikzcd}
\end{equation*}
Consider the corresponding algebra isomorphism
$\phi^*: A^{(2)}=\KK[T^{(2)}\setminus D^{(2)}]\to A^{(1)}=\KK[T^{(1)}\setminus D^{(1)}]$. If $\partial^{(1)}$ is a $\KK$-derivation of $A^{(1)}$ then $(\phi^*)\inv\circ\partial^{(1)}\circ \phi^*$ is a $\KK$-derivation of $A^{(2)}$ and vice versa. 
Moreover the derivations of $A^{(i)}$ are generated over $A^{(i)}$ by the $\partial_u$ where $u\in N^{(i)}$. Thus we see that the property of a point $p$ being a critical point  of $W$ is equivalent to the condition $\partial^{(2)} W(p)=0$ for all $\KK$-derivations $\partial^{(2)}$ of $A^{(2)}$, and analogously for $\phi^*(W)$ and $A^{(1)}$. 

Now observe that if $p$ is a critical point of $W$ lying in $T^{(2)}\setminus D^{(2)}$, then $\phi\inv(p)$ is a critical point for $\phi^*(W)$. Namely for any $\KK$-derivation $\partial^{(1)}$ of $A^{(1)}$ we have that $(\phi^*)\inv\circ\partial^{(1)}\circ \phi^* $ is a $\KK$-derivation of $A^{(2)}$, so that
\[
0=[(\phi^*)\inv\circ \partial^{(1)}\circ \phi^*] (W) (p)=
\partial^{(1)} ( \phi^*(W)) (\phi\inv(p)),
\]  
and thus $\phi\inv(p)$ is a critical point of $\phi^*(W)$. Clearly this observation is true for any regular map $W$ with no assumption on Laurentness, completeness or positivity. If we consider $\phi\inv$ in place of $\phi$, then for any critical point $p'$ of $\phi^*(W)$ we have $\phi(p')$ is a critical point of $W$. 

Now assume that $W$ is a positive, complete  Laurent polynomial. Then by Theorem~\ref{t:main} there is  a unique positive critical point $\pcrit$. 
Since $\phi$ is positively birational it restricts to a well-defined bijection $\phi_{>0}:T^{(1)}(\KK_{>0})\to T^{(2)}(\KK_{>0})$, and we can assume that $T^{(1)}(\KK_{>0})\subset T^{(1)}\setminus D^{(1)}$ and $T^{(2)}(\KK_{>0})\subset T^{(2)}\setminus D^{(2)}$.  
Since $W$ and $\phi^*(W)$ are both positive Laurent polynomials, we also have  the following commutative diagram obtained by restriction to the positive parts
\begin{equation*}
\begin{tikzcd}  T^{(1)}(\KK_{>0})
\arrow{r}{\phi_{>0}} \ar{dr}[swap]{\phi^*(W)} &T^{(2)}(\KK_{>0})\arrow{d}{W}\\
& \KK_{>0}.
\end{tikzcd}
\end{equation*}
By the previous argument, we have that $\phi_{>0}\inv(\pcrit)=\phi\inv(\pcrit)$ is a positive critical point of $\phi^*(W)$. We have by assumption that $\phi^*(W)$ is a positive Laurent polynomial. Let us now prove it is complete. Suppose indirectly its Newton polytope does not have full dimension. Then since
$\phi^*(W)$ has a positive critical point, by Lemma~\ref{l:contra2} it must have infinitely many positive critical points. However these critical points are mapped by $\phi_{>0}$ to distinct positive critical points of $W$, which contradicts the fact (Theorem~\ref{t:main}) that the positive critical point of $W$ is unique. Thus we see that the Newton polytope of $\phi^*(W)$ must be full-dimensional. 

Now we are in the situation of Lemma~\ref{l:contra1}. Since $\phi^*(W)$ is a positive Laurent polynomial with full-dimensional Newton polytope, and since  $\phi^*(W)$ possesses a positive critical point, it must follow that the Newton polytope of $\phi^*(W)$ contains $0$ in its interior. Hence we have shown that  $\phi^*(W)$ is complete. Clearly, if $\pcrit'$ is the unique critical point of  $\phi^*(W)$ then $\phi(\pcrit')=\pcrit$. Therefore (1) and (2) are proved.

Finally, (3) follows from (2) and the definitions in Section~\ref{s:tptropicalisation}. 
\end{proof}
 
We remark that the combination of the two lemmas used in this section imply a kind of converse to our main theorem. Namely we have the following corollary.  

\begin{cor}\label{c:converse} Suppose $W$ is a positive Laurent polynomial. Then $W$ has a unique positive critical point {\it if and only if} the Newton polytope of $W$ is full-dimensional with $0$ in the interior. \qed
\end{cor}

\bibliographystyle{amsalpha}
\bibliography{biblio}

\end{document}